\documentclass{article}
\usepackage[english]{babel}
\usepackage[letterpaper,top=2cm,bottom=2cm,left=3cm,right=3cm,marginparwidth=1.75cm]{geometry}
\usepackage{amsmath}
\usepackage{dsfont}
\usepackage{amsthm}
\usepackage{amssymb}
\usepackage{graphicx,xcolor,float}
\usepackage{subfigure}
\usepackage{epstopdf}
\usepackage[colorlinks=true, allcolors=blue]{hyperref}
\usepackage{cite}
\usepackage{authblk}
\title{Tightness for random walks driven by the two-dimensional Gaussian free field at high temperature}
\author[*]{Jian Ding}
\author[*]{Jiamin Wang}
\affil[*]{School of Mathematical Sciences, Peking University.}

\begin{document}
	\maketitle
	\newtheorem{thm}{Theorem}[section]
	\newtheorem{prop}[thm]{Proposition}
	\newtheorem{cor}[thm]{Corollary}
	\newtheorem{lem}[thm]{Lemma}
	\newtheorem{defi}[thm]{Definition}
	\newtheorem{rmk}[thm]{Remark}
	\numberwithin{equation}{section}
	
	\begin{abstract}	
		We study random walks in random environments generated by the two-dimensional Gaussian free field. More specifically, we consider a rescaled lattice with a small mesh size and view it as a random network where each edge is equipped with an electric resistance given by a regularization for the exponentiation of the Gaussian free field. We prove the tightness of random walks on such random networks at high temperature as the mesh size tends to 0. Our proof is based on a careful analysis of the (random) effective resistances as well as their connections to random walks.
	\end{abstract}
	
	\section{Introduction}
	In this paper, we consider a random walk on a random environment generated by the whole-plane Gaussian free field (see, e.g., \cite{sheffield2007gaussian,werner2020lecture,berestycki2024gaussian} for more details on the GFF). Since the GFF is only well-defined as a distribution, as usual a regularization procedure is required to make precise meaning of a random environment generated by the GFF. For convenience, we will work with the white noise approximation for the GFF, defined as follows:

	\begin{equation}
		\phi_{\delta}(x)=\sqrt{\pi}\int_{\delta^2}^{1}\int_{\mathbb{R}^2}p_{\frac{t}{2}}(x-y)W(dy,dt)
	\end{equation}
	for $x\in\mathbb{R}^2$ and $\delta\in(0,1)$, where $p_t(x)=\frac{1}{2\pi t}e^{-\frac{|x|^2}{2t}}$ is the heat kernel on $\mathbb{R}^2$ and $W$ is a standard space-time white noise.
	
	Write $\phi_n=\phi_{2^{-n}}$ for integers $n\geq 1$, and we consider the rescaled lattice $\mathbb Z_n^2=2^{-n}\zeta_n^{-1}\mathbb{Z}^2$ with edges given by naturally scaled nearest-neighboring edges in $\mathbb Z^2$, where $\zeta_n $ is a positive integer with $\zeta_n\geq \sqrt{n}$. (We remark that with the introduction of $\zeta_n$ the mesh size of the rescaled lattice is smaller than the scale at which the GFF is regularized, thus providing convenient smooth properties of the field; this is a sensible choice as justified in Proposition \ref{mesh}.) For $\gamma>0$, we will consider the following electric network on $\mathbb Z_n^2$: for each edge $e = (u, v)$ in $\mathbb Z_n^2$, we associate a resistance $r_e=\exp(\gamma\phi_n(m_e))$ where $m_e$ denotes the midpoint of $e$. Let $B(a,b)=[-a,a]\times [-b,b]\subset \mathbb{R}^2$ and $S^n_{a,b,\zeta_n}=\mathbb Z_n^2\cap B(a,b)$. In addition, let $R^{n}_{a,b,\zeta_n}$ be the effective resistance between the left and right sides of $S^n_{a,b,\zeta_n}$, i.e., $$R^{n}_{a,b,\zeta_n}=\inf_{\theta}\sum_{e\in S^n_{a,b,\zeta_n}} \theta(e)^2r_e,$$ where the infimum is taken over all unit flows $\theta$ from the left side of $S^n_{a,b,\zeta_n}$ to its right side. We will abbreviate it as left-right resistance for simplicity, and we may drop the subscript $\zeta_n$ for simplicity in the context of no ambiguity. Throughout the paper, $\gamma_0$ denotes a small but absolute constant and our main theorems hold for $\gamma < \gamma_0$ (i.e., at high temperature). Our first result is on the tightness of the effective resistance. We write $\mathbb{R}^+$ (or $\mathbb Z^+$) for the set of all positive real numbers (or integers).
	
	\begin{thm}\label{thm1}
		For $\gamma<\gamma_0$ and $q\in \mathbb{R}^+$, the family of distributions of $\{R^{n}_{a,qa,\zeta_n}\}_{n\in \mathbb{Z}^+, a\in \mathbb{R}^+, \zeta_n\geq \sqrt{n}}$ is tight.
	\end{thm}
	
	We now consider a random walk on $\mathbb Z_n^2$, where at each step the walk moves to a neighbor chosen with probability inverse-proportional to the edge resistance. With Theorem \ref{thm1}, we can then derive the tightness for the traces of random walks. For notation clarity, we denote by $\mathbb P$ and $\mathbb E$ the probability and expectation with respect to the random field, and by $P$ and $E$ with respect to the random walk (we may also use superscripts to indicate the starting point of the random walk, e.g., we denote by $E^0$ the expectation with respect to the random walk started at the origin $0$). In what follows, we denote $B(a) = B(a, a)$ for short.

	\begin{thm}\label{thm2}
		For $\gamma<\gamma_0$, let $X^{(n)}$ be the random walk on $\mathbb Z_n^2$ started at the origin and stopped upon exiting $B(1)$. Then for any subsequence $\{n_m\}$, there exists a further subsequence $\{n_{m_k}\}$ such that $\mathbb P$-almost surely, $\{X^{(n_{m_k})}\}$ converges weakly with respect to the topology on curves modulo time parameterization.
	\end{thm}
	
	\begin{rmk}\label{rmk1}
		The topology on curves modulo time parameterization in Theorem \ref{thm2} is induced by the following metric on the curve space. If $\beta_1:[0,1]\to \mathbb{C}$ and $\beta_2:[0,1]\to \mathbb{C}$ are continuous curves, we set
		$$
		d_{\mathrm{CMP}}(\beta_1,\beta_2):=\inf_{\alpha}\sup_{t\in[0,1]}|\beta_1(t)-\beta_2(\alpha(t))|,
		$$
		where the infimum is taken over all increasing homeomorphisms $\alpha:[0,1]\to[0,1]$.
	\end{rmk}
	
	In light of Theorem \ref{thm2}, we wish to further incorporate the time parameterization. To this end, it is natural to suspect that the suitable scaling has to do with the expected exit time from the unit ball, and as a result we need an a priori input on the tightness for such exit times. Let $\tau_n$ denote the first time the random walk $X^{(n)}$ exists $B(1)$, and let $\chi_n = 2^{(2+\gamma^2/2)n} \zeta_n^2$.
	
	\begin{thm}\label{exptime}
		The family of laws of $\{\log (\chi_n^{-1}E^0\tau_n)\}_{n\in \mathbb{Z}^+}$ is tight.
	\end{thm}
	
	With Theorem \ref{exptime}, we can derive the tightness with time scaling $t\to t/\chi_n$. For $x\in \mathbb Z_n^2$, let $X^{(n,x)}$ denote the random walk started from the vertex $x$. Let $\hat{X}^{(n,x)}:[0,\infty)\to \mathbb{R}^2$ be the piecewise linear interpolation of the process $t\to X^{(n,x)}_{\chi_nt}$. That is (we denote $\lfloor x\rfloor$ as the maximal integer that is not greater than $x$), $$\hat{X}^{(n,x)}_{t}=X^{(n,x)}_{\lfloor \chi_n t \rfloor}+(\chi_n t-\lfloor \chi_n t \rfloor)(X^{(n,x)}_{\lfloor \chi_n t \rfloor+1}-X^{(n,x)}_{\lfloor \chi_n t \rfloor}).$$
	We define $P^{(n,x)}$ to be the law of $\hat{X}^{(n,x)}$ given $\phi_n$. We also extend this definition to general $x\in \mathbb R^2$ by linear interpolation as follows. Let $a,b,c\in \mathbb Z_n^2$ be the vertices of the isosceles right triangle containing $x$. There exist $\lambda_1,\lambda_2,\lambda_3\in [0,1]$ with $\lambda_1+\lambda_2+\lambda_3=1$ such that $x=\lambda_1a+\lambda_2b+\lambda_3c$. We then define 
	\begin{equation}\label{Pdef}
		P^{(n,x)}=\lambda_1P^{(n,a)}+\lambda_2P^{(n,b)}+\lambda_3P^{(n,c)}.
	\end{equation}
	We will prove the tightness of the random functions $P^{(n)}:x\to P^{(n,x)}$ for $n\geq1$.

	\begin{thm}\label{thm3}
		For $\gamma < \gamma_0$, $\mathbb{P}$-almost surely, $P^{(n)}$ is continuous and the law of $\{P^{(n)}\}$ is tight, with respect to the local uniform metric on the probability measure space $\mathrm{Prob}(C([0,\infty),\mathbb{R}^2)$, associated with the Prokhorov topology induced by the local uniform metric on curves.
	\end{thm}

	\begin{rmk}\label{rmk2}
		We now remind the reader the detailed definition of the topology as mentioned in Theorem \ref{thm3}.
		We emphasize that in this context our metric on the curve space does depend on the time parameterization.
		\begin{itemize}
			\item The space $C([0,\infty), \mathbb{R}^2)$ is equipped with the local uniform metric
			$$d_{\mathrm{Unif}}(X,X^{\prime})=\int_{0}^{\infty}e^{-T}\wedge \sup_{t\in [0,T]}|X_t-X^{\prime}_t|dT.
			$$
			\item The space $\mathrm{Prob}(C([0,\infty), \mathbb{R}^2))$ is equipped with the prokhorov distance $d_{\mathrm{Prok}}$ induced by $d_{\mathrm{Unif}}$ on $C([0,\infty), \mathbb{R}^2)$.
			\item The space of functions $P:\mathbb{R}^2\to \mathrm{Prob}(C([0,\infty), \mathbb{R}^2))$ is equipped with the metric $d$, where
			\begin{equation}\label{metric}
				d(P,Q)=\int_0^{\infty}e^{-L}\wedge\sup_{x\in B(L)}d_{\mathrm{Prok}}(P^x,Q^x)dL.
			\end{equation}
		\end{itemize}
	\end{rmk}
	
	\begin{rmk}
		We conjecture that the scaling limit will be drastically different from a Brownian motion with time change (so, in particular different from the Liouville Brownian motion). For instance, we conjecture that the exit measure from the unit ball (i.e., the distribution of the exit point) should be supported on a set with dimension strictly less than 1.
	\end{rmk}

	\begin{rmk}
		It is natural to ask what to expect for large $\gamma$. In a way, the question is two-folded: whether the tightness of the resistance as in Theorem \ref{thm1} still holds and whether the tightness of the rescaled random walk as in Theorem \ref{thm3} still holds. Our feeling is that the analogue of Theorem \ref{thm1} probably extends to all $\gamma$ (partly based on the analogous result in the supercritical Liouville quantum gravity metric as in \cite{ding2020tightness1}). As for the random walk, possibly there is a phase transition in $\gamma$ (essentially governed by the phase transition of the Liouville quantum gravity measure).
	\end{rmk}
	
	\subsection{Background and related work}
	In this subsection, we discuss previous works related to the present paper.
	
	The most closely related work is \cite{biskup2020return}, which conceptually considers the same model but with the following difference on technical choices: (1) in \cite{biskup2020return} the random environment is given by the discrete  Gaussian free field whereas we consider (the regularization of) the (continuum) Gaussian free field; this allows us to take advantage of the approximate conformal covariance property and thus to derive an RSW estimate in a different way from \cite{biskup2020return}. (2) In \cite{biskup2020return} the conductance on an edge is given by the exponential of the sum of GFF values on the two endpoints whereas we replace this by the exponential of the midpoint; this conveniently allows us to derive a sharp duality relation (see Lemma \ref{duallem}).  In terms of results, for both effective resistances and random walk exit times, in \cite{biskup2020return} correct exponents were computed (for all $\gamma$), whereas we obtain up-to-constant estimates (for small $\gamma$). As a result, we obtain a complete tightness result as in Theorem \ref{thm3}. We emphasize that, while our technical choices above are of help, such improvements are mostly due to conquering major obstacles encountered in \cite{biskup2020return}, as we elaborate next.
	
	The first major obstacle is the concentration of the effective resistance. In \cite{biskup2020return}, this was achieved by the general Gaussian concentration inequality, which implies that the variance for the logarithm of the resistance at most diverges logrithmically. In this work, we prove in Theorem \ref{thm1} that the logarithm of the effective resistance has an $O(1)$ fluctuation. Our proof follows the general framework for the tightness of Liouville quantum gravity (LQG) metrics \cite{ding2020tightness,ding2020tightness1}, but we have to deal with some fundamental challenge arising from the fact that the effective resistance does not grow in the scale, in contrast to the scaling growth for the LQG distance. To address this challenge, we need to prove that the energy for the electric current restricted in a small box only constitutes a small fraction of the total energy (see Lemma \ref{decaylem}), which then allows us to derive that resampling the field locally has a small contribution to the variability of the global resistance (the analogue for the LQG distance is the key for proving LQG tightness).
	
	The second major obstacle is to obtain random walk estimates from resistance estimates. In \cite{biskup2020return}, some rather non-trivial techniques were developed in order to accomplish this, but in various places these estimates lose diverging factors. In this paper, we substantially further develop such techniques, thereby obtaining up-to-constant estimates. This is highly non-trivial, and  for instance this seems to require a very good understanding of the electric current (or equivalently, the harmonic measure of the random walk), which can be highly sensitive with resistances (it is well-known that the harmonic measure can change drastically even if each edge resistance is changed up to a factor of 2; see \cite{Benjamini1991,ding2013sensitivity} for applications of this property). To this end, we have to develop some delicate comparisons between effective resistances (see Lemma \ref{returnlem3}), which cannot be directly implied by up-to-constant estimates.
	
	Another related model is the Liouville Brownian motion (LBM), which was constructed in \cite{garban2016liouville, Berestycki2013DiffusionIP} as a time-change of Brownian motion where the time-change (very roughly speaking) is given by the LQG. Afterwards, the heat kernel was constructed in \cite{garban2014on} and the short-time off-diagonal heat kernel was related to the LQG metric in \cite{ding2019heat}. In addition, in \cite{gwynne2022invariance,berestycki2022random} it was shown that random walks on mated-CRT maps converge to the LBM. We believe that there is a conceptual difference between our model and the LBM, as we conjecture that the scaling limit for the trace of our walk is drastically different from that of the Brownian motion.
	
	Our Theorem \ref{thm1} can be seen from a metric perspective, and is thus related to the LQG metric; for instance, Theorem \ref{thm1} is an analogue of \cite[Theorem~20]{ding2020tightness} and \cite[Proposition~3.2]{ding2020tightness1}. In fact, as aforementioned, our proof for the tightness of resistance draws inspiration from the tightness of LQG metrics developed in \cite{ding2019lioville,ding2020subsequential,dubedat2020liouville,ding2020tightness,ding2020tightness1}. One important remaining problem is to derive the convergence of the resistance, i.e., the uniqueness, and it is natural to try to again draw inspiration from the framework of LQG uniqueness as initiated in \cite{gwynne2021existence} and developed in \cite{ding2023uniqueness}. However, there seems to be a fundamental challenge in adapting such a framework since, for instance, the effective resistance is not a length metric.
	
	Finally, we point out that our model can also be seen as a random conductance model where the conductance is generated by exponentiating the GFF. There has been extensive study on the random conductance model, mostly focusing on the elliptic case where the conductances are uniformly bounded from below and above.  We refer to \cite{biskup2011recent} for a survey on this topic. An important method to study the random conductance model (and random walk in random environment in general) is to consider the operator perspective, along which line the important method of homogenization has been developed (see, e.g., the recent book \cite{Armstrong_2019} for an excellent account on this topic). Notably, recent progress on homogenization allows to deal with highly-nonuniform case \cite{armstrong2024superdiffusive,armstrong2024renormalization}. In particular, \cite{armstrong2024superdiffusive,armstrong2024renormalization} considered the random diffusion with random drift given by a divergence free vector field generated from the GFF and proved an invariance principle with a superdiffusive scaling. While the model in \cite{armstrong2024superdiffusive,armstrong2024renormalization} is conceptually different from ours (e.g., it converges to the Brownian motion and it does not see traps), it is quite possible that the method therein is robust enough to treat more general models including ours (likely with additional and substantial work) and to derive results that seem out of reach by tools we are familiar with. 
	\footnote{Informal note to readers: In a conversation with Scott Armstrong, we see the optimism in extending methods in \cite{armstrong2024superdiffusive,armstrong2024renormalization} to prove \emph{the full convergence} for random walks/diffusions with drifts given by the gradients of the Gaussian free field (the two models are similar in the following sense: in each step, the bias of the random walk is asymptotic to $a_\gamma \times \mbox{square of step size} \times \mbox{gradient}$, where $a_\gamma$ is a constant depending on $\gamma$). More importantly, their method seems to be not limited to the two-dimensional case, whereas our work employs planarity in various places in a crucial way. Finally, it is also surprising to us that their methods seem to suggest a possibility of coarse-graining the electric flows and proving their convergence directly.  With all the potential (amazing, in our opinion) power, somewhat interestingly, it seems the methods of \cite{armstrong2024superdiffusive,armstrong2024renormalization} would also be restricted to the case for small $\gamma$, as in our case.}
	
	\subsection{Outline}
	The proof in this paper can be naturally divided into two parts: the resistance estimate and the random walk estimate (provided with the resistance estimate). In the previous subsection, we have briefly mentioned major obstacles in these two parts when making comparison with \cite{biskup2020return}, and in what follows we elaborate this further. We also provide a detailed outline by describing the main content in each subsequent section.
	
	\subsubsection{Resistance estimate}
	
	We first discuss the proof of Theorem \ref{thm1}, the tightness of effective resistance through a rectangle. To this end, we (as mentioned before) borrow the framework for tightness of LQG metrics as in \cite{ding2020tightness,ding2020tightness1}. In the overview level, the idea was to combine multi-scale analysis with the Efron-Stein inequality, and to show that resampling the field locally only makes a small perturbation to the global distance. Along the way, percolation-type arguments were extensively employed, among which the RSW estimate plays a crucial role. 
	
	In order to argue that local resampling only slightly perturbs the LQG metric, a key ingredient is that the LQG distance crossing through a box grows polynomially with the side length of the box; note that this holds for all $\gamma > 0$ (see \cite{ding2021distance}). To the contrary, by duality of the effective resistance and by symmetry of the GFF, the logarithm of the effective resistance centers around 0 and thus the effective resistance stays bounded as the side length of the box grows. 
	
	In order to address the above challenge, our intuition is that the electric current that flows into a small box should have a polynomial decay in its side length, and thus the energy generated therein only contributes to a small fraction of the total Dirichlet energy. The implementation of this requires an intrinsic induction between the effective resistance and the electric current for the following reason: we can only expect that polynomial decay of the electric current provided with the assumptions that resistances in different scales are comparable with high probability, which in turn calls for a concentration bound on the effective resistance. The necessity of such intrinsic induction is the main reason why our method can only work for small $\gamma$. We next describe the main contents for Section \ref{sec2}-\ref{sec6}, which altogether complete the proof of Theorem \ref{thm1}.
	
	In \textbf{Section \ref{sec2}}, we prove properties for general (deterministic) electric networks. We review the generalized parallel and series law for resistances, from which we derive the duality between the resistance and the conductance. We also prove a triangle inequality in Proposition \ref{series} which allows us to concatenate paths later (and thus apply percolation techniques). In addition, we prove Lemma \ref{resdif} which will be applied in Section \ref{sec6} to infer the decay of electric energy between successive scales. Roughly speaking, Lemma \ref{resdif} bounds the difference of the effective resistance after removing a sub-graph $\mathrm{D}$ in terms of the ``energy around $\mathrm{D}$'' and the ``energy within $\mathrm{D}$''.
	
	In \textbf{Section \ref{sec3}}, we introduce another natural approximation $\psi_\delta$ for the GFF which has finite-range dependence. We prove some basic properties of $\phi_{\delta}$ and $\psi_{\delta}$, including the comparison between $\phi_{\delta}$ and $\psi_{\delta}$, as well as estimates of their maxima and oscillations. We will also give a bound comparing effective resistances with different mesh sizes. In fact, for two different mesh sizes $2^{-n}\zeta_n^{-1}$ and $2^{-n}(\zeta_n^\prime)^{-1}$, we will show that whenever $\zeta_n, \zeta_n^\prime \geq \sqrt{n}$, the effective resistances $R^{n}_{a,b,\zeta_n}$ and $R^{n}_{a,b,\zeta_n^\prime}$ are comparable (see Proposition \ref{mesh}).
	
	In \textbf{Section \ref{sec4}}, we prove the following RSW estimate for the effective resistance: the quantiles of the effective resistances between the sides of rectangles with different aspect ratios are comparable. Our RSW proof follows that for the LQG metric as in \cite{dubedat2020liouville}, which employs the conformal invariance of the GFF. In the context of the LQG metric, the optimization is over all continuous paths and thus this is more naturally consistent with the composition of a conformal map. In our context, we have a lattice structure (which is the underlying graph for the random walk), and one difficulty is that a conformal map destroys the lattice structure. In order to address this, we need to compare the Dirichlet energy in the continuous case and in the discrete case. This is where we use our assumption that $\zeta_n \geq \sqrt{n}$, which ensures that the oscillation of the regularized field within the mesh size is under control.
	
	In \textbf{Section \ref{sec5}}, we give tail estimates for effective resistances by quantiles of resistances crossing through a square. A supercritical percolation argument, combined with the triangle inequality, will be applied to obtain such tail bounds.
	
	In \textbf{Section \ref{sec6}}, we put together estimates in previous sections and \emph{inductively} bound the variance for the logarithm of the crossing resistance. As aforementioned, the key here is to control the change of the resistance when resampling a local field. To this end, we first apply Lemma \ref{resdif} to obtain a bound on the change of resistance in Lemma \ref{Efron}, which in turn gives a bound in Lemma \ref{intervar} on the total variability from resampling fields in a local scale. In order for the bound to be effective, it is crucial that the maximal energy restricted to a local box is very small compared to the total energy, which is proved in Lemma \ref{decaylem} with an important application of the induction hypothesis on bounded variation for resistances in previous scales. Note that, it is for the reason of verifying the initial basis of this intrinsic induction that we need the assumption on small $\gamma$. 
	
	\subsubsection{Random walk estimate}
	We now explain how to derive random walk estimates provided with Theorem \ref{thm1}. 
	
	In order to compute the correct exponents for random walk estimates, the authors of \cite{biskup2020return} sometimes consider two scales that are significantly different so that the respective resistances can be compared more easily. While it is already rather nontrivial to implement this in \cite{biskup2020return}, in our paper we have to compare successive scales since we aim at up-to-constant estimates. To this end, we  need to carefully investigate the electric current.

	In \textbf{Section \ref{sec7}}, by taking advantage of planarity, we prove in Lemma \ref{returnlem} an estimate for the return probability provided with assumptions on effective resistances around and across a few annuli. Combining Theorem \ref{thm1} and Lemma \ref{returnlem}, we can get that with high $\mathbb{P}$-probability, the random walk stays in a tube with positive $P$-probability. This allows us to prove that the exit measure is continuous with respect to the starting point, which then yields the tightness for random walk traces.
	
	In \textbf{Section \ref{sec8}}, we prove the tightness of the suitably normalized expected exit time. It turns out that the appropriate normalizing constant is given by the LQG measure; the proof of this crucially employs Lemma \ref{returnlem3} which provides a delicate comparison between effective resistances. At this point, we can prove the tightness for the expected exit time by the tightness of the LQG measure as well as the tightness of the effective resistance as in Theorem \ref{thm1}. 
	
	In \textbf{Section \ref{sec9}}, we use the result in Section \ref{sec8} as a main input to prove the tightness of our random walks with suitable time parameterization. Similar to \cite{berestycki2022random}, we will use Arzela-Ascoli's Theorem to obtain the tightness. This will be divided into two steps: the equicontinuity of paths and the equicontinuity of laws. Here we will also need the tightness of the second moment of the exit time, which can be obtained in the same way as that for the first moment (by considering the Green's function). Altogether, this completes the proof of Theorem \ref{thm3}.
	
	\subsection*{Convention on Constants}
	In this paper, we use the convention that $C, c$ are for absolute constants whose values may change from line to line, and we use the convention that $C, c$ with decorations such as $C_1, c^\prime$ are constants whose values are fixed within each section but may vary from section to section. We will often specify the dependence of a constant upon its introduction. For instance, there exists $\delta = \delta(K, M, \epsilon) >0$ means that the constant $\delta$ depends on $K, M, \epsilon$. For two sequence of positive numbers $(a_n)$ and $(b_n)$, we write $a_n = O(b_n)$ or $a_n \preceq b_n$ if there exists a constant $C>0$ such that $a_b \leq C b_n$ for all $n\geq 1$.
	
	\subsection*{Acknowledgement}
	We warmly and sincerely thank Subhajit Goswami and Ewain Gwynne for many stimulating discussions during various stages of this project. J. Ding is partially supported by NSFC Key Program Project No. 12231002 and the Xplorer prize.
	
	\section{Preliminaries on electric networks}\label{sec2}
	We will briefly review some basic definitions for electric networks, and we refer readers to e.g. \cite{doyle1984random,Lyons_Peres_2017} for more details. Let $\mathrm{G}$ be a finite, undirected, and connected simple graph, where the vertex set is $V(\mathrm{G})$ and the edge set is $E(\mathrm{G})$. We will consider $\mathrm{G}$ as an electric network, where we associate to each edge $e\in E(\mathrm{G})$ a resistance $r_e\in \mathbb R^+$ (and thus correspondingly a conductance $c_e = 1/r_e$). 
	
	We say two vertices $u,v\in V(\mathrm{G})$ are adjacent (denoted by $u\sim v$) if $(u,v)\in E(\mathrm{G})$, and we say two edges $e, e^{\prime}$ are adjacent if their endpoints share a vertex. A path is a sequence of vertices such that any two successive vertices are adjacent. For $e\in E(\mathrm{G})$ and a path $P$, we write $e\in P$ if $e$ is an edge between two successive vertices in $P$. A cycle in $\mathrm{G}$ is a path where the starting and ending points coincide.
	
	Let $\overrightarrow{E}(\mathrm{G}) \subset V(\mathrm{G}) \times V(\mathrm{G})$ be the collection of (directed) neighboring pairs, which can be obtained from $E$ by assigning each edge two orientations. In addition, for all $A\subset E$, we denote by $\overrightarrow{A}$ the collection of neighboring pairs in $A$. A flow from $u$ to $v$ is then a function $\theta: \overrightarrow{E} (\mathrm{G}) \to \mathbb{R}$ such that $\theta(x,y)=-\theta(y,x)$ for every $(x, y)\in \overrightarrow{E}(G)$, $\sum_{y:y\sim x}\theta(x,y)=0$ for every $x\in V(G) \setminus \{u, v\}$, and $\sum_{x: x\sim u} \theta(u, x) \geq 0$. The strength of the flow $\theta$ is defined by $\sum_{x:x\sim u}\theta(u,x)$. A flow with strength 1 is called a unit flow. For a sub-graph $\mathrm{H}$ of $\mathrm{G}$ and a function $\theta:E(\mathrm{G})\to \mathbb R$, we write
	\begin{equation}
		\mathcal{E}(\theta,\mathrm{H})=\sum_{e\in E(\mathrm{H})} \theta(e)^2 r_e
	\end{equation}
	to be the energy of $\theta$ on $\mathrm{H}$. We denote by $\mathcal{E} (\theta) = \mathcal{E} (\theta, \mathrm{G})$ for short. With these notations, we can define the effective resistance between $u$ and $v$ in $\mathrm{G}$ by
	\begin{equation}\label{resistancedef}
		R_{\mathrm{G}}(u,v)=\inf_{\theta}\sum_{e\in E(\mathrm{G})} \theta(e)^2r_e,
	\end{equation}
	where the infimum is taken over all unit flows from $u$ to $v$. The unit flow which achieves the minimum in \eqref{resistancedef} is called an electric current (or a current for short). Sometimes the sum is over $e\in \overrightarrow{E}(\mathrm{G})$ which then comes with a factor of 1/2. 
	
	By definition, if $C$ is a cycle in $\mathrm{G}$, we can define a flow $\theta_C$ supported on $C$ such that $\theta_C(e)=1$ for every $e\in C$ oriented in the direction of $C$. If $\theta$ is a unit current, then $\theta+a\theta_C$ is a unit flow for all $a\in\mathbb{R}$. By taking derivatives, we have $\sum_{e\in C}\theta(e)r_e=0$, which is called Kirchhoff's law. Thus, we can assign a potential $f:E(\mathrm{G})\to\mathbb{R}$ at each vertex such that $f(x)-f(y)=\theta(x,y)r(x,y)$ (such a potential function is also referred to as a \emph{voltage}). Direct calculations show that $f(u)-f(v)=R_{\mathrm{G}}(u,v)$.
	
	For a unit electric current $\theta$ from $u$ to $v$ and for $A \subset E$, define the current through $A$ by
	\begin{equation}\label{setcur}
		\theta(A) = 1-\varphi(A,\theta),
	\end{equation}
	where $\varphi(A,\theta)$ is the maximal flow $\tau$ one can send from $u$ to $v$ such that $|\tau(e)|\leq |\theta(e)|$ for all $e\in \overrightarrow{E}(\mathrm{G})$ and $\tau(e)=0$ for all $e\in A$.
	
	Finally, sometimes we consider flows from a set $A$ to $B$, and this is generalized naturally by identifying $A$ (and respectively $B$) as a single vertex.
	
	\subsection{Characterization of effective resistance}
	We review the generalized parallel and series laws for effective resistances below. We will provide the proof for Proposition \ref{res} for completeness, and omit the proof for Proposition \ref{conchar} due to similarity. One may consult e.g. \cite{biskup2020return} for thorough discussions and omitted proofs.
	
	\begin{prop}\label{res}
		Let $\mathcal{P}_{u,v}$ denote the collection of self-avoiding paths from $u$ to $v$. Then we have 
		\begin{equation}\label{reschar}
			R_{\mathrm{G}}(u,v)=\inf \sum_{k}\alpha_k^2\sum_{e\in P_k}r_{e,P_k},
		\end{equation}
		where the infimum is taken over all finite sets $\left\{P_1,\cdots,P_n\right\}\subset \mathcal{P}_{u,v}$, $\{\alpha_k\}\subset \mathbb{R}^+$ and $\{r_{e, P_k}\}\subset \mathbb{R}^+$ such that
		\begin{equation}
			\sum_k\alpha_k=1\And\sum_{k:e\in P_k}\frac{1}{r_{e,P_k}}\leq\frac{1}{r_e} \mbox{ for each edge } e.
		\end{equation}
		The infima in \eqref{reschar} are (jointly) achieved.
	\end{prop}
	
	\begin{proof}
		The proof is taken from \cite{biskup2020return}. On the one hand, for every $\left\{P_1,\cdots,P_n\right\}\subset \mathcal{P}$ and $\{r_{e,P_k}\}\subset \mathbb{R}^+$ such that $\sum_{k:e\in P_k}1/r_{e,P_k}\leq1/r_e$, we can view an edge $e$ in the network $\mathrm{G}$ as a parallel of edges $e_1, \ldots, e_n$ and $\tilde e$, where $e_k$ has resistance $r_{e, P_k}$ (for $1\leq k\leq n$) and $\tilde e$ has resistance $r_{\tilde e}$ such that $1/r_{\tilde{e}}=1/r_e-\sum_{k:e\in P_k}1/r_{e,P_k}$. (We allow the resistance of $\tilde{e}$ to be infinity, in which case it is equivalent to delete the edge $\tilde{e}$.) Then we can send a unit flow from $u$ to $v$ by assigning a flow of strength $\alpha_k$ along each path $P_k$ for $1\leq k\leq n$. Thus, by \eqref{resistancedef}, it holds that 
		\begin{equation}
			R_{\mathrm{G}}\leq \sum_k\alpha_k^2\sum_{e\in P_k}r_{e,P_k}.
		\end{equation}
		
		On the other hand, let $\theta_0$ be the unit electric current. First we observe that if a flow $\theta$ has nonzero strength, then there must be a path $P$ from $u$ to $v$ with $\theta(e)>0$ for each $e\in P_k$ oriented in the same direction of $P_k$. Note that this can be proved by an induction on the number of vertices in $\mathrm{G}$.
		
		With the observation, we define $P_k$, $\alpha_k$ and $\theta_k$ inductively in the following way. If $\theta_{k-1}$ has nonzero strength, then there is a path $P_k$ such that $\theta_{k-1}(e)>0$ for each $e\in P$ oriented in the same direction of $P$. Let $\alpha_k=\min_{e\in P_k}\theta_{k-1}(e)>0$, and let
		\begin{equation}
			\theta_k(e)=\theta_{k-1}(e)-\alpha_k \mbox{sign}(\theta_{k-1}(e))\mathds{1}_{\left\{e\in P_k\right\}}.
		\end{equation}
		Since the number of edges with nonzero values in $\theta_k$ is decreasing, the induction must stop in finite steps. Then $\sum_k\alpha_k=1$. Setting $r_{e,P_k}=r_e|\theta_0(e)|/\alpha_k$, we then have that 
		\begin{equation}
			\sum_{k}\alpha_k^2\sum_{e\in P_k}r_{e,P_k}=\sum_k\alpha_k\sum_{e\in P_k}r_e|\theta_0(e)|=\sum_k\alpha_k U=U=R_{\mathrm{G}}(u,v),
		\end{equation}
		where $U$ denotes the potential (i.e., voltage) difference between $u$ and $v$.
	\end{proof}
	
	Taking infimum with respect to $\alpha_k$, we can deduce the following corollary.
	
	\begin{cor}\label{rescharcor}
		Let $\mathcal{P}_{u,v}$ denote the collection of paths from $u$ to $v$. Then we have
		\begin{equation}\label{reschar1}
			R_{\mathrm{G}}(u,v)=\inf \left(\sum_{k}\frac{1}{\sum_{e\in P_k}r_{e,P_k}}\right)^{-1},
		\end{equation}
		where the infimum is taken over all finite set $\left\{P_1,\cdots,P_n\right\}\subset \mathcal{P}_{u,v}$ and $\{ r_{e, P_k}\} \subset \mathbb{R}^+$ such that
		\begin{equation}
			\sum_{k:e\in P_k}\frac{1}{r_{e,P_k}}\leq\frac{1}{r_e} \mbox{ for each edge } e.
		\end{equation}
		The infima in \eqref{reschar1} are (jointly) achieved.
	\end{cor}
	
	\begin{proof}
		For fixed $\{P_k\}$ and $\{r_{e,P_k}\}$, \eqref{reschar} can be minimized by taking 
		\begin{equation}
			\alpha_k=\frac {\left(\sum_{e\in P_k}r_{e,P_k}\right)^{-1}}{\sum_{j} \left(\sum_{e\in P_j}r_{e,P_j}\right)^{-1}} \mbox{ for } 1\leq k\leq n,
		\end{equation}
		and we complete the proof.
	\end{proof}

	We next generalize the series law, which can be seen as a generalization for the Nash-Williams' criterion. We define the effective conductance between $u$ and $v$ by (below $e^+, e^-$ denote the two endpoints of an edge $e$)
	\begin{equation}\label{condef}
		C_{\mathrm{G}}(u,v)=\inf_{f}\sum_{e\in E}c_e(f(e^+)-f(e^-))^2,
	\end{equation}
	where the infimum is taken over all $f:V(\mathrm{G})\to \mathbb{R}$ that satisfies $f(u)=1$ and $f(v)=0$. The fundamental electrostatic duality implies that (see e.g. \cite{Lyons_Peres_2017})
	\begin{equation}
		C_{\mathrm{G}}(u,v)=\frac{1}{R_{\mathrm{G}}(u,v)}.
	\end{equation}
	The Nash-Williams' criterion provides a bound on the effective conductance via considering disjoint cut sets, where $\pi\subset E(\mathrm{G})$ is a cut set between $u$ and $v$ if every path from $u$ to $v$ intersects with $\pi$. We next characterize the effective conductance via cut sets, in the same spirit of Proposition \ref{res}.
	
	\begin{prop}\label{conchar}
		\textup{(\hspace{-0.15mm}\cite[Proposition~2.3]{biskup2020return})}
		Let $\Pi_{u,v}$ be the collection of cut sets between $u$ and $v$. Then we have
		\begin{equation}\label{reschar3}
			C_{\mathrm{G}}(u,v)=\inf\left(\sum_k\frac{1}{\sum_{e\in\pi_k}c_{e,\pi_k}}\right)^{-1},
		\end{equation}
		where the infimum is taken over all finite set $\left\{\pi_1,\cdots,\pi_n\right\}\subset\Pi_{u,v}$, $\{c_{e, \pi_k}\} \subset \mathbb{R}^+$ such that
		\begin{equation}
			\sum_{k:e\in \pi_k}\frac{1}{c_{e,\pi_k}}\leq\frac{1}{c_e} \mbox{ for every edge }e.
		\end{equation}
		The infima in \eqref{reschar3} are (jointly) achieved.
	\end{prop}
	
	\subsection{Triangle inequality for effective resistance}
	Following Proposition \ref{res} which characterizes the effective resistance via collections of paths, we now generalize the definition of effective resistance. For $\mathcal{A} \subset \{A: A \subset E(\mathrm{G})\}$, let
	\begin{equation}
		R_{\mathrm{G}}(\mathcal{A})=\inf \left(\sum_{k}\frac{1}{\sum_{e\in P_k}r_{e,P_k}}\right)^{-1},
	\end{equation}
	where the infimum is taken over all finite sets $\left\{P_1,\cdots,P_n\right\}\subset \mathcal{A}$ and $\{r_{e, P_k}\} \subset \mathbb{R}^+$ such that 
	\begin{equation}
		\sum_{k:e\in P_k}\frac{1}{r_{e,P_k}}\leq\frac{1}{r_e}\mbox{ for all } e\in E(\mathrm{G}).
	\end{equation}
	
	It is easy to see that $\mathcal{A}\to R_{\mathrm{G}}(\mathcal{A})$ is non-increasing. In addition, this generalizes the effective resistance in the following sense: for disjoint vertex sets $A,B\subset V(\mathrm{G})$, the effective resistance between $A,B$ is $R_{\mathrm{G}}(\mathcal{P}_{A,B})$, where $\mathcal{P}_{A,B}$ denotes the collection of self-avoiding paths joining $A$ and $B$ (note that this includes, in particular, the case when $A$ and $B$ are singletons). We next generalize the parallel and series law to this generalized notion of resistance.
	\begin{prop}\label{parallel}
		Let $\mathcal{A}_1,\cdots,\mathcal{A}_n$ and $\mathcal{P}$ be collections of edge sets. Suppose that $\mathcal{P}\subset \bigcup_i \mathcal{A}_i$. Then we have
		\begin{equation}\label{paralleleq}
			R_{\mathrm{G}}(\mathcal{P})^{-1}\leq\sum_i R_{\mathrm{G}}(\mathcal{A}_i)^{-1}.
		\end{equation}
	\end{prop}
	
	\begin{proof}
		For each $\left\{P_1,\cdots,P_n\right\}\subset \mathcal{P}$ and $r_{e,P_k}\in \mathbb{R}^+$ such that $\sum_{k:e\in P_k}1/r_{e,P_k}\leq1/r_e$ for every edge $e$, we have $\sum_{k: P_k \in \mathcal A_i} 1/r_{e, P_k} \leq 1/r_e$ for each $i$ and every edge $e$, and thus we have
		\begin{equation}
			\sum_{k:P_k\in \mathcal{A}_i}\frac{1}{\sum_{e\in P_k}r_{e,P_k}} \leq R_{\mathrm{G}}(\mathcal{A}_i)^{-1}, \mbox{ for each } i.
		\end{equation}
		Summing the above over $i$ and taking the supremum over all legitimate choices for ${P_1, \ldots, P_n}$ and $\{r_{e, P_k}\}$, we conclude the proof of \eqref{paralleleq}.
	\end{proof}

	The generalized series law is somewhat more complicated, and we will prove it using the Max-Flow-Min-Cut Theorem (see e.g. \cite[Section~3.1]{Lyons_Peres_2017}), which is stated below for convenience.
	
	\begin{lem}\label{maxflowmincut}
		Let $\mathrm{G}$ be a finite network where each edge $e\in E(\mathrm{G})$ is equipped with a capacity $w(e) \geq 0$. Then for any disjoint vertex sets $A$ and $Z$, we have
		\begin{equation}\label{maxflowmincuteq}
			\begin{aligned}
				\max \bigg\{\mbox{Strength}(\theta); \theta \mbox{ flows from } A\mbox{ to }Z \mbox{ such that } |\theta(e)| \leq w(e) \mbox{ for all }  e\in \overrightarrow{E} (\mathrm{G})\bigg\}\\
				=
				\min \bigg\{\sum_{e\in \pi}w(e); \pi \mbox{ is a cut set separating } A \mbox{ and } Z \bigg\}.
			\end{aligned}
		\end{equation}
	\end{lem}
	
	\begin{prop}\label{series}
		Let $G$ be a network, and let $\mathrm{H}_1,\cdots,\mathrm{H}_n$ be sub-graphs of $\mathrm{G}$. For each $1\leq i\leq n$, let $A_i,Z_i$ be disjoint subsets in $\mathrm{H}_i$. Let $A,Z$ be disjoint subsets in $\mathrm{G}$. Suppose that $\cup_i P_i$ contains a path from $A$ to $Z$, where $P_i$ is an arbitrary path within $\mathrm{H}_i$ connecting $A_i$ and $Z_i$ for $1\leq i\leq n$. Then, we have
		\begin{equation}
			R_{\mathrm{G}}(A,Z)\leq \sum_{i=1}^{n}R_{\mathrm{H}_i}(A_i,Z_i).
		\end{equation}
	\end{prop}
	
	\begin{proof}
		Let $\theta_i$ be the unit electric current in the network $H_i$ from $A_i$ to $Z_i$. We extend $\theta_i$ to a flow on $G$ by defining $\theta_i(e) = 0$ for $e\in E(\mathrm G)\setminus E(\mathrm H_i)$. We consider a flow $\theta$ from $A$ to $Z$ with the restriction
		\begin{equation}\label{restric}
			|\theta(e)|\leq \max_{1\leq i\leq n}|\theta_i(e)|:=w(e).
		\end{equation}
		By Lemma \ref{maxflowmincut}, the maximal strength of $\theta$ is equal to the minimal cut (i.e., of the form on the right-hand side of \eqref{maxflowmincuteq}) separating $A$ and $Z$. For a cut set $\pi$ of $\mathrm{G}$ separating $A$ and $Z$, suppose that for each $i$ there exists $P_i$ in $\mathrm{H}_i$ connecting $A_i$ and $Z_i$ such that $\pi$ does not intersect with $P_i$. By our assumption in the lemma statement,  there exists a path $P\subset\cup_{1\leq i\leq n}P_i$ connecting $A$ to $Z$. We see that $P$ does not intersect with $\pi$, which contradicts the fact that $\pi$ is a cut set. Thus, there exists an $i$ such that $\pi$ must intersect with every path in $\mathrm{H}_i$ joining $A_i$ and $Z_i$. That is to say, $\pi$ is the cut set separating $A_i$ and $Z_i$ in $\mathrm{H}_i$. By Lemma \ref{maxflowmincut}, we then get that
		\begin{equation}
			\sum_{e\in \pi} w(e)\geq \sum_{e\in \pi} |\theta_i(e)|\geq 1.
		\end{equation}
		Thus, the minimal cut between $A$ and $Z$ (i.e., the right-hand side of \eqref{maxflowmincuteq} with $w$ given in \eqref{restric}) is at least 1. By Lemma \ref{maxflowmincut}, there exists a unit flow $\theta$ from $A$ to $Z$ with $|\theta(e)|\leq w(e)$. Therefore, we have
		\begin{equation}
			R_{\mathrm{G}}(A,Z)\leq\sum_{e\in E(\mathrm{G})}\theta(e)^2r(e)\leq \sum_{e\in E(\mathrm{G})}w(e)^2r(e)\leq \sum_{e\in E(\mathrm{G})}\sum_{i=1}^{n}\theta_i(e)^2r(e)= \sum_{i=1}^{n}R_{\mathrm{H}_i}(A_i,Z_i),
		\end{equation}
		completing the proof of the lemma.
	\end{proof}
	
	\subsection{Duality of effective resistance and effective conductance}
	In this subsection, we prove the duality between the effective resistance and the effective conductance for a network with $\mathbb Z^2$ (and by a natural scaling, this applies to our $\mathbb Z_n^2$) as its underlying graph, which essentially follows from that in \cite{biskup2020return}. Let $\{\eta_x\}_{x\in \mathbb{R}^2}$ be a field such that $\eta\overset{d}{=}-\eta$, where $X\overset{d}{=}Y$ means that $X$ and $Y$ have the same distribution. For each edge $e\in\mathbb{Z}^2$, recall that $m_e$ denotes the midpoint of $e$. If $S$ is a sub-graph of $\mathbb{Z}^2$, we can define $S_{\eta}$ as the network on $S$, where the edge set $E(S)$ contains all edges with both endpoints in $S$ and each edge $e\in E(S)$ is equipped with a resistance $\exp(\eta_{m_e})$.
	
	Taking advantage of the variational characterization of the effective resistance and the effective conductance, we in what follows deduce the self-duality between the effective resistance and the effective conductance with the observation on the duality between paths and cut sets.
	
	We consider the dual domain $(\mathbb{Z}^2)^*=(1/2,1/2)+\mathbb{Z}^2$, where the (dual) edge set is given by vertex pairs with Euclidean distance 1. More specifically, we say an edge $e\in \mathbb Z^2$ is dual to an edge $e^*\in (\mathbb Z^2)^*$ if $m_e = m_{e^*}$, and we say an edge set $E$ in $\mathbb Z^2$ is dual to an edge set $E^*$ in $(\mathbb Z^2)^*$ if each edge in $E$ is dual to an edge in $E^*$ and vice versa. Write $S=B(m,n)$. For the network $S_{\eta}$, we define the up-down dual network $S^*_{\eta}$ where the vertex set is given by $S^*=[-m-1/2,m+1/2]\times[-n+1/2,n-1/2]\cap(\mathbb{Z}^2)^*$, and the edge set is given by (dual) edges within $S^*$. We equip each edge $e^*\in E(S^*_{\eta})$ with a resistance $r_{e^*}=\exp(\eta_{m_{e^*}})$.
	
	If $S$ is a rectangle in $\mathbb{Z}^2$ with sides parallel to axises, we write $\partial_{\mathrm{left}}S, \partial_{\mathrm{right}}S, \partial_{\mathrm{up}}S,$ and $\partial_{\mathrm{down}}S$ as the vertex set of the left, right, up, and down side of $S$, respectively.
	
	\begin{lem}\label{duallem}
		Suppose that $\eta \overset{d}{=} -\eta$. Then for any rectangle $S$, we have
		\begin{equation}\label{dual1}
			R_{S_{\eta}}(\partial_{\mathrm{left}}S,\partial_{\mathrm{right}}S)\overset{d}{=}R_{S_{\eta}^*}(\partial_{\mathrm{up}}S^*,\partial_{\mathrm{down}}S^*)^{-1}.
		\end{equation}
	\end{lem}
	
	\begin{proof}
		Notice that in $S_{\eta}$, an edge set contains a left-right crossing path if and only if its dual set is a cut set separating $\partial_{\mathrm{up}}S^*$ and $\partial_{\mathrm{down}}S^*$ in $S^*_{-\eta}$. Since the edge resistance in $S_\eta$ is the same as the conductance for  its dual edge in $S^*_{-\eta}$, by Corollary \ref{rescharcor} and Proposition \ref{conchar} we get that
		\begin{equation}\label{dualeq}
			R_{S_{\eta}}(\partial_{\mathrm{left}}S,\partial_{\mathrm{right}}S)=C_{S^*_{-\eta}}(\partial_{\mathrm{up}}S^*,\partial_{\mathrm{down}}S^*).
		\end{equation}
		Since we have $\eta\overset{d}{=}-\eta$, it holds that
		\begin{equation}
			C_{S^*_{-\eta}}(\partial_{\mathrm{up}}S^*,\partial_{\mathrm{down}}S^*)=R_{S^*_{-\eta}}(\partial_{\mathrm{up}}S^*,\partial_{\mathrm{down}}S^*)^{-1}\overset{d}{=}R_{S^*_{\eta}}(\partial_{\mathrm{up}}S^*,\partial_{\mathrm{down}}S^*)^{-1},
		\end{equation}
		completing the proof of the lemma by combining \eqref{dualeq}.
	\end{proof}

	\subsection{Local bounds for effective resistances}
	When applying the Efron-Stein inequality, it is important to bound the change of the resistance after resampling the field in a local region. To this end, we prove the next lemma, which will be used in the proof of Proposition \ref{mainprop} later (more precisely, in Lemmas \ref{energy} and \ref{decaylem}). For notation convenience, in this subsection we consider some fixed source and sink in the network $\mathrm{G}$, and denote by $R$ and $\theta$ the corresponding effective resistance and the unit current, respectively. For a sub-graph $\mathrm{D}$ of $\mathrm{G}$, we let $R^{\setminus D}$ be the resistance after setting the resistance of each edge in D to be infinite. Recall \eqref{setcur} for the definition of electric current through a set.
	
	\begin{lem}\label{resdif}
		Let $\mathrm{D},\mathrm{H}$ be disjoint sub-graphs of $\mathrm{G}$. Let $\mathcal P$ be the collection of all paths from the source to the sink and let $\mathcal P^\prime$ be a collection of edge sets in $\mathrm H$. Suppose that for every $P\in\mathcal P$ and $P^\prime\in \mathcal P^\prime$, whenever $P\cap E(\mathrm{D})\neq \emptyset$, there exists a subset $Q_{P, P^\prime} \subset P\cup P^\prime$ such that $Q_{P, P^\prime}\in\mathcal P$ and $Q_{P, P^\prime}\cap E(\mathrm{D})=\emptyset$. Then we have
		\begin{equation}\label{resdifeq}
			R^{\setminus \mathrm{D}}-R\leq\mathcal{E}(\theta,\mathrm{H})+2\theta(\mathrm{D})^2R(\mathcal P^\prime)-\mathcal{E}(\theta,\mathrm{D}).
		\end{equation}
	\end{lem}
	\begin{proof}
		By Corollary \ref{rescharcor} we can choose $\left\{r_{e,P}:e\in E(\mathrm{G}),P\in\mathcal{P}\right\}$ satisfying that $\sum_{P\in\mathcal{P}}1/r_{e,P}\leq1/r_e$ for each $e\in E(\mathrm{G})$ such that $R(\mathcal{P})= \left(\sum_{P\in \mathcal{P}}(\sum_{e\in P}r_{e,P})^{-1}\right)^{-1}$. In addition,  an analogue of the prime version holds too, i.e., we can define $\{r^\prime_{e, P^\prime}: e\in E(H), P^\prime\in \mathcal P^\prime\}$ such that $\sum_{P^\prime\in\mathcal{P}^\prime}1/r^\prime_{e,P^\prime}\leq1/r_e$ for each $e\in E(\mathrm{G})$ and $R(\mathcal{P}^\prime)= \left(\sum_{P^\prime\in \mathcal{P}^\prime}(\sum_{e\in P^\prime}r^\prime_{e,P^\prime})^{-1}\right)^{-1}$.
		
		For $P\in \mathcal P$ and $P^\prime\in \mathcal P^\prime$, define $\alpha_{P}$ and $\alpha^\prime_{P^\prime}$ respectively by
		\begin{equation}
			\alpha_{P}=\frac{\left(\sum_{e\in P}r_{e,P}\right)^{-1}}{\sum_{\tilde P\in \mathcal{P}}\left(\sum_{e\in \tilde P}r_{e,\tilde P}\right)^{-1}} \And \alpha^\prime_{P^\prime}=\frac{\left(\sum_{e\in P^\prime}r^\prime_{e,P^\prime}\right)^{-1}}{\sum_{\tilde P^\prime\in \mathcal{P}^\prime}\left(\sum_{e\in \tilde P^\prime}r^\prime_{e,\tilde P^\prime}\right)^{-1}}.
		\end{equation}
		If we view an edge $e$ as a parallel of edges and assign a flow $\alpha_{P}$ to each path $P\in \mathcal{P}$, then the energy of this flow on this equivalent network is exactly $R$. Since the electric current is unique, it holds that $|\theta(e)|=\sum_{P:e\in P}\alpha_{P}$ and
		\begin{equation}\label{resdif3}
			|\theta(e)|r_e=\alpha_{P}r_{e,P}, \mbox{ for all } e\in P \mbox{ and } P\in \mathcal P.
		\end{equation}
		We wish to define $r_{e, Q_{P, P^\prime}}$ for every $P\in\mathcal P$ and $P^\prime\in\mathcal P^\prime$. To this end, we set $r_{e, Q_{P, P^\prime}} = \infty$ for $e\not\in Q_{P, P^\prime}$, and in what follows we consider $e\in Q_{P, P^\prime}$. If $P\cap E(D)\neq\emptyset$, define
		\begin{equation}\label{resdif1}
			r_{e,Q_{P,P^\prime}}=\left\{
			\begin{aligned}
				&\frac{r_{e,P}}{\alpha^\prime_{P^\prime}}, &e\in P\setminus (E(\mathrm{H})\cup E(\mathrm{D})),\\
				&2\frac{r_{e,P}}{\alpha^\prime_{P^\prime}}, &e\in P\cap E(\mathrm{H}),\\
				&2\frac{\theta(\mathrm{D})r_{e,P^\prime}}{\alpha_{P}}, &e\in P^\prime\setminus P.\\
			\end{aligned}
			\right.
		\end{equation}
		If $P\cap E(\mathrm{D})=\emptyset$, we set $Q_{P,P^\prime}=P$ and
		\begin{equation}\label{resdif2}
			r_{e,Q_{P,P^\prime}}=\left\{
			\begin{aligned}
				&\frac{r_{e,P}}{\alpha^\prime_{P^\prime}}, &e\in P\setminus (E(\mathrm{H})\cup E(\mathrm{D})),\\
				&2\frac{r_{e,P}}{\alpha^\prime_{P^\prime}}, &e\in P\cap E(\mathrm{H}).\\
			\end{aligned}
			\right.
		\end{equation}
		For $e\in E(\mathrm{G})\setminus E(\mathrm{D})$, we see that the following hold: if $e\in E(\mathrm{G})\setminus E(\mathrm{H})$,
		\begin{equation}
			\sum_{P\in\mathcal P, P^\prime\in \mathcal P^\prime}\frac{1}{r_{e,Q_{P,P^\prime}}}\leq\sum_{P\in\mathcal P, P^\prime\in \mathcal P^\prime}\frac{\alpha^\prime_{P^\prime}}{r_{e,P}}\leq\sum_{P\in\mathcal P}\frac{1}{r_{e,P}}\leq \frac{1}{r_e};
		\end{equation}
		if $e\in E(\mathrm{H})\setminus E(\mathrm{D})$,
		\begin{equation}
			\begin{aligned}
				\sum_{P\in\mathcal P, P^\prime\in \mathcal P^\prime}\frac{1}{r_{e,Q_{P,P^\prime}}}
				&\leq 
				\sum_{P^\prime\in \mathcal P^\prime, P\cap E(\mathrm{D})=\emptyset }\frac{\alpha^\prime_{P^\prime}}{2r_{e,P}}+\sum_{P^\prime\in\mathcal P^\prime,P\cap E(\mathrm{D})\neq\emptyset}\left(\frac{\alpha^\prime_{P^\prime}}{2r_{e,P}}+\frac{\alpha_{P}}{2\theta(\mathrm{D})r_{e,P^\prime}}\right)\\
				&\leq\sum_{P\in \mathcal P}\frac{1}{2r_{e,P}}+\sum_{P^\prime\in\mathcal P^\prime}\frac{1}{2r_{e,P^\prime}}\leq \frac{1}{r_e},
			\end{aligned}
		\end{equation}
		where we used the fact that $\sum_{P^\prime\in \mathcal P^\prime}\alpha^\prime_{P^\prime}=1$ and $\sum_{P\cap E(\mathrm{D})\neq \emptyset} \alpha_{P}\leq \theta(\mathrm{D})$ in the second inequality. Assigning a flow with strength $\alpha_{P}\alpha^\prime_{P^\prime}$ on each $Q_{P,P^\prime}$, we obtain that (note that the combination of these flows is a unit flow in $\mathrm G\setminus \mathrm D$)
		\begin{equation}\label{resdif4}
			\begin{aligned}
				R^{\setminus \mathrm{D}}
				&\leq
				\sum_{P\in\mathcal P, P^\prime\in \mathcal P^\prime}\alpha_{P}^2(\alpha^\prime_{P^\prime})^2\sum_{e\in Q_{P,P^\prime}}r_{e,Q_{P,P^\prime}}\\
				&\leq
				\sum_{P\in\mathcal P, P^\prime\in \mathcal P^\prime}\alpha_{P}^2(\alpha^\prime_{P^\prime})^2\left(\sum_{e\in P\setminus E(\mathrm{H})\cup E(\mathrm{D})}\frac{r_{e,P}}{\alpha^\prime_{P^\prime}}+\sum_{e\in P\cap E(\mathrm{H})}\frac{2r_{e,P}}{\alpha^\prime_{P^\prime}}+\sum_{e\in P^\prime}\frac{2\theta(\mathrm{D})r_{e,P^\prime}}{\alpha_{P}}1_{P\cap E(\mathrm{D})\neq\emptyset}\right)\\
				&=
				\mathcal{E}(\theta,G\setminus (\mathrm{H}\cup \mathrm{D}) )+2\mathcal{E}(\theta,\mathrm{H})+2\theta(\mathrm{D})^2R_{\mathrm{G}}(\mathcal P^\prime),
			\end{aligned}
		\end{equation}
		where we used \eqref{resdif1} and \eqref{resdif2} in the second inequality, and in the equality we used the fact from \eqref{resdif3} that for any sub-graph $\mathrm{F}$, 
		$$\sum_{e\in E(\mathrm{F})}\sum_{P:e\in P}\alpha_{P}^2r_{e,P}=\sum_{e\in E(\mathrm{F})}\theta(e)r_e\sum_{P:e\in P}\alpha_{P}=\sum_{e\in E(\mathrm{F})}\theta(e)^2r_e=\mathcal{E}(\theta, \mathrm{F}).$$ 
		Combining \eqref{resdif4} with the fact that
		\begin{equation}
			R=\mathcal{E}(\theta,\mathrm{G}\setminus (\mathrm{H}\cup \mathrm{D}))+\mathcal{E}(\theta,\mathrm{H})+\mathcal{E}(\theta,\mathrm{D}),
		\end{equation}
		we complete the proof of the lemma. 
	\end{proof}
	\section{Various approximations}\label{sec3}
	In this section, we review a few properties of $\phi_\delta$, most of which are taken from \cite{ding2020tightness}. We start with some basic properties of the field in Section \ref{subsec31}, and in Section \ref{subsec32} we introduce an approximation $\psi_{\delta}$ which satisfies the following properties.
	\begin{enumerate}
		\item The $L^{\infty}$-norm between $\psi_{0,n}$ and $\phi_{0,n}$ is sub-gaussian;
		\item $\psi_{\delta}$ has finite-range dependence;
		\item $\psi_{\delta}$ is invariant under Euclidean symmetries.
	\end{enumerate}
	In Section \ref{secmesh}, we compare random resistances (generated by these two fields) with different mesh sizes; in Section \ref{subsec34}, we compare different quantiles for the resistances. 
	\subsection{Basic properties of $\phi_{\delta}$}\label{subsec31}
	For integers $n>m>0$, we define
	\begin{equation}
		\phi_{m,n}=\sqrt{\pi}\int_{2^{-2n}}^{2^{-2m}}\int_{\mathbb{R}^2}p_{\frac{t}{2}}(x,y)W(dy,dt).
	\end{equation}
	Straightforward computations show that
	\begin{equation}
		\mathrm{Var}(\phi_{\delta}(x))=-\log \delta \And \mathrm{Var}(\phi_{m,n}(x))=(n-m)\log2.
	\end{equation}
	For $x\neq y$, it holds that
	\begin{equation}\label{cov1}
		\mathbb{E}(\phi_{\delta}(x)\phi_{\delta}(y))
		=\pi\int_{\delta^2}^{1}p_{t}(x-y)dt
		=\int_{|x-y|^2/2}^{|x-y|^2/2\delta^2}\frac{1}{2 s}e^{-s}ds \leq (-\log|x-y|+1)\vee 1.
	\end{equation}
	By (the equality in) \eqref{cov1} the scaling property of $\phi_{\delta}$ can be deduced:
	\begin{equation}\label{scaling}
		\mathbb{E}(\phi_{ra,rb}(rx)\phi_{ra,rb}(ry))=\mathbb{E}(\phi_{a,b}(x)\phi_{a,b}(y)),
	\end{equation}
	implying that $\phi_{ra,rb}(r\cdot)\overset{d}{=}\phi_{a,b}(\cdot)$. The following bounds on the maximum and the oscillation of $\phi_n$ are taken from \cite[Section~2.1]{ding2020tightness}. Here we write $\|\cdot \|_{B(1)}$ for the $L^{\infty}$-norm in $B(1)$.
	
	\begin{prop}\label{max}
		\textup{(\hspace{-0.15mm}\cite[Proposition~2]{ding2020tightness})} There exists an absolute constant C such that for any $n>0$, $x>0$, we have
		\begin{equation}
			\mathbb{P}(\|\phi_n\|_{B(1)}\geq (n+C\sqrt{n})x)\leq C4^ne^{-\frac{nx^2}{\log4}}.
		\end{equation}
	\end{prop}
	
	\begin{prop}\label{gradient}
		\textup{(\hspace{-0.15mm}\cite[Proposition~3]{ding2020tightness})} There exist absolute constants $C,c>0$ such that for any $n>0$, $x>0$, we have
		\begin{equation}
			\mathbb{P}(2^{-n}\|\nabla\phi_n\|_{B(1)}\geq x)
			\leq C4^ne^{-cx^2}.
		\end{equation}
	\end{prop}

	The following corollary can be obtained directly from Proposition \ref{gradient}. We define the oscillation of $f$ in a domain $B$ by $$\mathrm{osc}(f,\epsilon,B)=\sup_{|x-y|\leq \epsilon, x,y\in B}|f(x)-f(y)|.$$
	
	\begin{cor}\label{osc}
		For $\zeta_n\geq \sqrt{n}$, there exist absolute constants $C,c>0$ such that for all $n>0$, $x>0$ we have
		\begin{equation}
			\mathbb{P}(\mathrm{osc}(\phi_n,2^{-n+1}\zeta_n^{-1},B(1))\geq x)
			\leq Ce^{-cx^2}.
		\end{equation}
	\end{cor}
	The following corollary follows immediately from a union bound over unit squares.
	\begin{cor}\label{osc1}
		For $\zeta_n\geq \sqrt{n}$, there exist absolute constants $C,c>0$ such that for all $n>0$, $a,b\geq 1$ and $x>0$ we have
		\begin{equation}
			\mathbb{P}(\mathrm{osc}(\phi_n,2^{-n+1}\zeta_n^{-1},B(a,b))\geq x)
			\leq Cabe^{-cx^2}.
		\end{equation}
	\end{cor}
	\subsection{Finite-range dependent approximation}\label{subsec32}
	When applying percolation arguments and the Efron-Stein inequality later, it is convenient to assume finite-range dependence for the field under consideration. Let $\Phi$ be a smooth, non-negative, radially symmetric bump function on $\mathbb{R}^2$ such that $0\leq\Phi\leq1$ and that $\Phi$ is equal to 1 on $\{x:|x|<1\}$ and equal to 0 on $\{x:|x|\geq 2\}$. We fix a constant $\epsilon_0=1/100$ and define
	\begin{equation}\label{psidef}
		\sigma_t=\epsilon_0\sqrt{t}|\log t|^{\epsilon_0}
		\And \tilde{p}_t(x,y)=p_t(x,y)\Phi(|x-y|/\sigma_t).
	\end{equation}
	In the same way as defining $\phi_\delta$ but replacing $p$ with $\tilde{p}$,  we define our approximation $\psi_\delta$ by
	\begin{equation}
		\psi_{\delta}(x)=\sqrt{\pi}\int_{\delta^2}^{1}\int_{\mathbb{R}^2}\tilde{p}_{\frac{t}{2}}(x,y)W(dy,dt) \mbox{ for } x\in \mathbb R^2\,.
	\end{equation}
	In addition, for integers $n>m>0$, we define
	\begin{equation}
		\psi_{m,n}(x)=\sqrt{\pi}\int_{2^{-2n}}^{2^{-2m}}\int_{\mathbb{R}^2}\tilde{p}_{\frac{t}{2}}(x,y)W(dy,dt).
	\end{equation}
	Then $\psi_{m,n}$ is a centered Gaussian process. Since $\tilde{p}_t$ is supported on the Euclidean ball of radius $2\sigma_t$, we can see that if $|x-y|\geq 2^{-m}m^{\epsilon_0}$, then $\psi_{m,n}(x)$ and $\psi_{m,n}(y)$ are independent. 
	
	In the rest of this paper, we will write $R^{m,n}(\cdot)$ and $\tilde R^{m,n}(\cdot)$ for the effective resistances on $\mathbb Z_n^2$, associated with the field $\phi_{m,n}$ and $\psi_{m,n}$, respectively. For brevity we will write $R^n(\cdot)$ for $R^{0,n}(\cdot)$ and write $\tilde R^n(\cdot)$ for $\tilde R^{0,n}(\cdot)$.
	
	The following propositions from \cite{ding2020tightness} justify that $\psi(\cdot)$ is a good approximation for $\phi(\cdot)$.
	
	\begin{prop}\label{distance}
		\textup{(\hspace{-0.15mm}\cite[Lemma~4]{ding2020tightness})} There exists an absolute constant $C$ such that for all $\delta\in(0,1)$ and $x,x^{\prime}\in\mathbb{R}^2$, we have
		\begin{equation}
			\mathrm{Var}(\phi_{\delta}(x)-\phi_{\delta}(x^{\prime}))+\mathrm{Var}(\psi_{\delta}(x)-\psi_{\delta}(x^{\prime}))\leq C\frac{|x-x^{\prime}|}{\delta}.
		\end{equation}
	\end{prop}
	
	\begin{prop}\label{normprop}
		\textup{(\hspace{-0.15mm}\cite[Proposition~5]{ding2020tightness})} There exist absolute constants $C,c>0$ such that for all $x>0$, we have
		\begin{equation}
			\mathbb{P}(\sup_{n>m>0}\|\phi_{m,n}-\psi_{m,n}\|_{B(1)}\geq x)\leq Ce^{-cx^2}.
		\end{equation}
	\end{prop}

	We will also need an estimate on the ``coarse field''.
	\begin{lem}\label{coarse}
		For $n, m\geq 1$ and disjoint domains $U_1,\ldots, U_m$,  we define $$\phi^i=\int_{2^{-2n}}^1\int_{U_i}p_{\frac{t}{2}}(x,y)W(dy,dt) \And \phi_{\mathrm c}^i=\phi_n-\phi^i.$$ 
		Then for all $\mathsf C, a>0$ and $\mathsf c\in (0,1)$, there exists a constant $C_1=C_1(\mathsf C,\mathsf c,a)$ such that the following holds. Suppose that $V_1,\ldots,V_m$ are domains such that $V_i+B(a)=\{x+y:x\in V_i,y\in B(a)\}\subset U_i$ for $1\leq i\leq m$. Let $N$ be the number of $i$ with $1\leq i\leq m$ such that
		\begin{equation}
			\max_{u_i\in V_i}|\phi_{\mathrm c}^i(u_i)|<C_1.
		\end{equation}
		Then we have
		\begin{equation}
			\mathbb P(N\geq \mathsf c m)\geq 1-\exp(-\mathsf Cm).
		\end{equation}
	\end{lem}

	\begin{proof}
		We fix $s_i\in \{-1,1\}$ for $1\leq i\leq m$. For $x,x^{\prime}\in V_i$, a direct calculation shows that
		\begin{equation}
			\mathrm{Var}(\phi_{\mathrm c}^i(x)-\phi_{\mathrm c}^i(x^{\prime}))=\int_{2^{-2n}}^{1}\int_{\mathbb{R}^2\setminus U_i}|p_{\frac{t}{2}}(x,y)-p_{\frac{t}{2}}(x^{\prime},y)|^2dydt\preceq |x-x^{\prime}|,
		\end{equation}
		where the constants in $\preceq$ throughout this proof depend only on $a$. Thus, by Fernique's inequality \cite{fernique} (see also \cite[Theorem~4.1]{adler1990introduction}),
		\begin{equation}\label{coar1}
			\mathbb{E}\left[\max_{x\in V_i}s_i\phi_{\mathrm c}^i(x)\right]\preceq 1
			\And
			\mathbb{E}\left[\max_{x_i\in V_i}\sum_{i=1}^ms_i\phi_{\mathrm c}^i(x_i)\right]\preceq m.
		\end{equation}
		Furthermore, for $x_i\in V_i$
		\begin{equation}
			\begin{aligned}
				\mathrm{Var}(s_i\phi_{\mathrm c}(x_i))
				\leq\int_{2^{-2n}}^{1}\int_{\mathbb{R}^2\setminus U_i}p_{\frac{t}{2}}(x_i-y)^2dydt\preceq \int_{2^{-2n}}^{1}\frac{1}{2\pi t}e^{-\frac{1}{t}}dt
				\preceq\int_1^{2^{2n}}\frac{1}{s}e^{-s}ds\preceq 1;
			\end{aligned}
		\end{equation}
		for $i\neq j$, $x_i\in V_i$ and $x_j\in V_j$,
		\begin{equation}
			\begin{aligned}
				\mathbb{E}(s_i\phi_{\mathrm c}^i(x_i)s_j\phi_{\mathrm c}^j(x_j))
				&\leq \int_{2^{-2n}}^{1}\int_{\mathbb{R}^2}p_{\frac{t}{2}}(x_i-y)p_{\frac{t}{2}}(x_j-y)dydt\\
				&\preceq \int_{2^{-2n}}^{1}\frac{1}{t}e^{-\frac{|x_i-x_j|^2}{t}}dt=\int_{|x_i-x_j|^2}^{2^{2n}|x_i-x_j|^2}\frac{1}{s}e^{-s}ds\preceq e^{-|x_i-x_j|^2}.
			\end{aligned}
		\end{equation}
		This implies that for $x_i \in V_i$ with $1\leq i\leq m$,
		\begin{equation}
			\mathrm{Var}\left(\sum_{i=1}^ms_i\phi_{\mathrm c}^i(x_i)\right)\preceq m.
		\end{equation}
		Thus by Borell-TIS inequality \cite{Borell1975-lg,Sudakov1978} (see also, e.g., \cite[Theorem~2.1.1]{adler2009random} and \cite[Theorem~3.25]{van2014probability}), there exists $C_2=C_2(\mathsf C,\mathsf c)$ such that
		\begin{equation}\label{coar}
			\mathbb{P}\left(\max_{x_i\in V_i}\sum_{i=1}^ms_i\phi_{\mathrm c}^i(x_i)\geq C_2 m\right)\leq \exp(-\mathsf Cm).
		\end{equation}
		Thus, by taking $C_1>C_2/(1-\mathsf c)$ we have
		\begin{equation}
			\begin{aligned}
				\mathbb P\left(N< \mathsf c m\right)
				&\leq \mathbb{P}\left(\sum_{i=1}^m\max_{x_i\in V_i}|\phi_{\mathrm c}^i(x_i)|\geq C_1(1-\mathsf c)m\right)\\
				&\leq \sum_{s_i\in \{-1,1\}}\mathbb{P}\left(\sum_{i=1}^m\max_{x_i\in V_i}s_i\phi_{\mathrm c}^i(x_i)\geq C_1(1-\mathsf c)m\right)\\
				&\leq 2^m\exp(-\mathsf C m),
			\end{aligned}
		\end{equation}
		which then completes the proof (by properly adjusting $\mathsf C$).
	\end{proof}
	
	\subsection{Comparison between different mesh sizes}\label{secmesh}
	The following proposition shows that the effective resistances with different mesh sizes are comparable provided that their mesh sizes are sufficiently small. This justifies the assumption $\zeta_n\geq \sqrt{n}$ since it is equivalent to consider the case where the mesh size is infinitesimal.
	\begin{prop}\label{mesh}
		There exist absolute constants $C,c$, such that for all $n\geq 1$ and $a,b\geq 1$ and $\zeta_n,\zeta_n^{\prime}\geq \sqrt{n}$
		\begin{equation}
			\mathbb{P}(|\log R^{n}_{a,b,\zeta_n}-\log R^{n}_{a,b,\zeta_n^{\prime}}|\geq x)\leq Cabe^{-cx^2}.
		\end{equation}
	\end{prop}
	\begin{figure}
		\centering
		\includegraphics{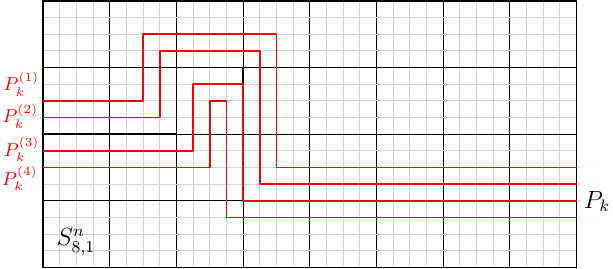}
		\caption{Illustration of left-right crossings.}
		\label{fig15}
	\end{figure}
	\begin{proof}
		Let $\eta_n=\zeta_n\zeta_n^{\prime}$, it is sufficient to prove that
		\begin{equation}\label{mesh1}
			\mathbb{P}(|\log R^{n}_{a,b,\zeta_n}-\log R^{n}_{a,b,\eta_n}|\geq x)\leq Cabe^{-cx^2}.
		\end{equation}
		Let $\theta$ be the electric current corresponding to $R^{n}_{a,b,\zeta_n}$. By Proposition \ref{res} we can choose self-avoiding crossings $P_1,\ldots,P_m$ in $S^n_{a,b,\zeta_n}$ and $\alpha_1,\ldots, \alpha_m \in \mathbb R^+$, such that $\sum_{1\leq k\leq m}1/r_{e,P_k}\leq1/r_e$ for each $e\in E(\mathrm{G})$ as well as $\sum_{1\leq k\leq m}\alpha_i=1$ and $R^{n}_{a,b,\zeta_n}= \sum_{1\leq k\leq m}\alpha_k^2\sum_{e\in P_k}r_{e,P_k}$. We see that $\theta$ can be decomposed as the sum of flows along $P_k$ with strength $\alpha_k$ (for $1\leq k\leq m$). For each $P_k$ with $1\leq k\leq m$, there exist $l>0.01 \eta_n/\zeta_n$ edge-disjoint left-right crossings of $S^{n}_{a,b,\eta_n}$ with Hausdorff distances at most $2^{-n-1}\zeta_n^{-1}$ from $P_k$, which we denote as $P_k^{(1)},\ldots, P_k^{(l)}$. (See Figure \ref{fig15} for an illustration.) We send a flow of strength $\alpha_k/l$ along $P_k^{(i)}$ for each $1\leq k\leq m$ and $1\leq i\leq l$, and as a result we get a unit flow $\theta^\prime$ from the left side to the right side of $S^{n}_{a,b,\eta_n}$. For each $e^\prime\in S^{n}_{a,b,\eta_n}$, we have $$|\theta^\prime(e^\prime)|\leq \sum_{e\in S^{n}_{a,b,\zeta_n}:|m_e-m_{e^\prime}|\leq 2^{-n}\zeta_n^{-1}}|\theta(e)|/l.$$		
		Note that for fixed $e^\prime \in S^{n}_{a,b,\eta_n}$, 
		\begin{equation}\label{mesh4}
			|\{e\in S^{n}_{a,b,\zeta_n}:|m_e-m_{e^\prime}|\leq 2^{-n}\zeta_n^{-1}\}|\leq 100.
		\end{equation}
		Thus, we have
		\begin{equation}\label{mesh2}
			\begin{aligned}
				R^{n}_{a,b,\eta_n}
				&\leq \sum_{e^{\prime}\in  S^{n}_{a,b,\eta_n}}\theta^\prime(e^{\prime})^2\exp(\gamma \phi_{n}(m_{e^{\prime}}))\\
				&\leq \sum_{e^{\prime}\in S^{n}_{a,b,\eta_n}}\frac{100^2}{l^2}\sum_{\substack{e\in S^{n}_{a,b,\zeta_n}:\\|m_e-m_{e^\prime}|\leq 2^{-n}\zeta_n^{-1}}}|\theta(e)|^2\exp(\gamma \phi_{n}(m_{e^{\prime}}))\\
				&\leq \exp(\gamma \mathrm{osc}(\phi_{n},2^{-n+1}\zeta_n^{-1},B(a,b)))\sum_{e\in S^{n}_{a,b,\zeta_n}}|\theta(e)|^2\exp(\gamma \phi_n(m_e))\sum_{\substack{e^\prime\in S^{n}_{a,b,\eta_n}:\\|m_e-m_{e^\prime}|\leq 2^{-n}\zeta_n^{-1}}}\frac{100^2}{l^2}\\
				&\leq C\exp(\gamma \mathrm{osc}(\phi_{n},2^{-n+1}\zeta_n^{-1},B(a,b)))R^{n}_{a,b,\zeta_n},
			\end{aligned}
		\end{equation}
		where we applied \eqref{mesh4} and Cauchy-Schwarz in the second inequality, and $C$ is an absolute constant. By the duality as in \eqref{dual1}, we also have
		\begin{equation}\label{mesh3}
			C^{n}_{a,b,\eta_n}\leq C\exp(\gamma \mathrm{osc}(\phi_{n},2^{-n+1}\zeta_n^{-1},B(a,b)))C^{n}_{a,b,\zeta_n}.
		\end{equation}
		Combining \eqref{mesh3} and \eqref{mesh2}, we can then infer \eqref{mesh1} from Corollary \ref{osc1}.
	\end{proof}
	
	\subsection{Quantiles of the effective resistance}\label{subsec34}
	Recall that in $\mathbb Z_n^2=2^{-n}\zeta_n^{-1}\mathbb{Z}^2$, $R^{n}_{a,b}$ denotes the left-right resistance of the rectangular domain $S^{n}_{a,b}$ equipped with resistances associated with the field $\phi_{n}$. Similarly, in $S^n_{a,b}$, if we associate the resistance according to the field $\psi_{n}$ instead of $\phi_{n}$, we get the left-right resistance $\tilde{R}^{n}_{a,b}$.
	
	We introduce the notation of the quantiles of $R^{n}_{a,b}$ and $\tilde{R}^{n}_{a,b}$, as follows:
	\begin{equation}
		l^{(n)}_{a,b}(\phi,p)=\inf\{x:\mathbb{P}(R^{n}_{a,b}\leq x)\geq p\},
	\end{equation}
	\begin{equation}
		\tilde{l}^{(n)}_{a,b}(\psi,p)=\inf\{x:\mathbb{P}(\tilde{R}^{n}_{a,b}\leq x)\geq p\}.
	\end{equation}
	By Cameron-Martin formula, since we can add a suitable bump function to the field, there is no Dirac mass in the laws of either $R^{n}_{a,b}$ or $\tilde{R}^{n}_{a,b}$. So in fact we have $\mathbb{P}(R^{n}_{a,b}\leq l^{(n)}_{a,b}(\phi,p))=p$ and $\mathbb{P}(\tilde{R}^{n}_{a,b}\leq \tilde l^{(n)}_{a,b}(\psi,p))=p$. For simplicity, we write $l^{(n)}(\phi,p)=l^{(n)}_{1,1}(\phi,p)$ and $\tilde{l}^{(n)}(\psi,p)=\tilde{l}^{(n)}_{1,1}(\psi,p)$. For the proof of the tightness, it would be convenient to consider ratios between various quantiles. To this end, for $p\in(0,1/2)$ we define 
	\begin{equation}\label{quantile}
		\Lambda_n(\phi,p)=\sup_{m\leq n } \frac{l^{(m)}(\phi,1-p)}{l^{(m)}(\phi,p)} \And \tilde{\Lambda}_n(\psi,p)=\sup_{m\leq n } \frac{\tilde{l}^{(m)}(\psi,1-p)}{\tilde{l}^{(m)}(\psi,p)}.
	\end{equation}

	\begin{lem}\label{comp}
		For a domain $U$ and a function $f:\mathbb{R}^U\to \mathbb R$ which is $1$-Lipschitz with respect to the $L^\infty$-norm on $U$, we have the following: for two random fields $\eta,\mu\in \mathbb{R}^U$ and for all $x,y\in \mathbb{R}$,
		\begin{equation}
			\mathbb{P}(f(\eta)\geq x)\leq \mathbb{P}(f(\mu)\geq x-y)+\mathbb{P}(\|\eta-\mu\|_U\geq y).
		\end{equation}
	\end{lem}
	\begin{proof}
		This lemma follows directly from a union bound.
	\end{proof}
	\begin{prop}\label{quancom}
		For all $p\in (0,1)$, there exists a constant $C_p$ depending only on $p$ such that for all $n>0$,
		\begin{equation}\label{quancomeq}
			\tilde l^{(n)}(\psi,p)\geq C_p^{-1} l^{(n)}(\phi,p/2)
			\And \tilde l^{(n)}(\psi,1-p)\leq C_pl^{(n)}(\phi,1-p/2).
		\end{equation}
	\end{prop}
	\begin{proof}
		By the definition of the effective resistance as in \eqref{resistancedef}, we have (recall $\gamma \leq 1$)
		$$
			|\log R^{n}_{1,1}-\log \tilde{R}^n_{1,1}|\leq \gamma\|\phi_{n}-\psi_n\|_{B(1)}\leq \|\phi_{n}-\psi_n\|_{B(1)}.
		$$
	Combined with Proposition \ref{normprop} and Lemma \ref{comp}, this yields the desired estimate.
	\end{proof}
	
	The next a priori bounds on quantiles follow from Lemma \ref{duallem}.
	\begin{lem}\label{prio}
		For all $n>0$ and $p\in(0,1/2)$, we have
		\begin{equation}
			l^{(n)}(\phi,1-p)\leq \Lambda_n(\phi,p)
			\And \tilde l^{(n)}(\psi,1-p)\leq \tilde \Lambda_n(\psi,p).
		\end{equation}
	\end{lem}
	\begin{proof}
		It suffices to show that $l^{(n)}(\phi,p)\leq1$ and $\tilde l^{(n)}(\psi,p)\leq1$. For any $k>0$ and a rectangle $S$ with dimensions $k2^{-n}\zeta_n^{-1} \times (k+1)2^{-n}\zeta_n^{-1}$, we can apply Lemma \ref{duallem} and derive that
		\begin{equation}
			\mathbb{P}(R_S(\partial_{\mathrm{left}}S,\partial_{\mathrm{right}}S)\leq 1)=\mathbb{P}(R_{S^*}(\partial_{\mathrm{up}}S^*,\partial_{\mathrm{down}}S^*)^{-1}\leq 1)
			=\mathbb{P}(R_S(\partial_{\mathrm{left}}S,\partial_{\mathrm{right}}S)\geq 1).
		\end{equation}
		Recalling the property of no Dirac mass, we then get that
		$$\mathbb{P}(R_S(\partial_{\mathrm{left}}S,\partial_{\mathrm{right}}S)\leq 1)=1/2,$$
		which implies that $\mathbb{P}(R^{n}_{1,1}\leq 1)\leq1/2$ by taking $k=2^{n+1}\zeta_n$. Thus, we have that $l^{(n)}(\phi,p)\leq1$, and similarly we have $\tilde l^{(n)}(\psi,p)\leq1$.
	\end{proof}
	
	\section{Russo-Seymour-Welsh estimates}\label{sec4}
	In this section, we will derive the following Russo-Seymour-Welsh estimates for quantiles of effective resistances, employing the proof idea of RSW results in \cite{dubedat2020liouville}. For a rectangle, we will refer to a crossing joining the two shorter sides as a \emph{hard} crossing, and refer to a crossing joining the two longer sides as an \emph{easy} crossing.
	\begin{prop}\label{rsw}
		For $(a,b),(\tilde a,\tilde b)$ with $a/b<1$ and $\tilde a/\tilde b>1$, there exists $C_1=C_1(a,b,\tilde a,\tilde b)$ such that for all $n\geq 1$ and $p\in (0,1)$ we have
		\begin{equation}
			l^{(n)}_{\tilde a,\tilde b}(\phi,p/C_1)\leq C_1l^{(n)}_{a,b}(\phi,p)e^{\gamma C_1\sqrt{\log(C_1/p)}},
		\end{equation}
		\begin{equation}
			l^{(n)}_{\tilde a,\tilde b}(\phi,1-p)\leq C_1l^{(n)}_{a,b}(\phi,C_1p)e^{\gamma C_1\sqrt{\log(C_1/p)}}.
		\end{equation}
	\end{prop}
	Our proof follows the RSW estimate for the LQG metric as in \cite{dubedat2020liouville}, with a notable difference that the effective resistance is defined with reference to a rescaled lattice (whereas in \cite{dubedat2020liouville} the metric is optimized over curves in $\mathbb R^2$). Such a difference causes an issue when applying a conformal map, as the lattice structure will not be preserved. As a result, we will first compare the Dirichlet energy on the rectangular network (i.e., the resistance) to its continuous analogue. It is for this reason that our assumption $\zeta_n \geq \sqrt{n}$ is important since this facilitates control of the oscillation of the field.
	
	\begin{lem}\label{diri}
		For a compact domain $K$, let $f$ be a function on $K\cap V(\mathbb Z_n^2)$. Consider a smooth field $\{\eta_x\}_{x\in \mathbb{R}^2}$. There exists an absolute constant $c$ such that $f$ can be extended to a function on $K$ such that
		\begin{equation}\label{dirieq}
			\sum_{e\in E(K)}e^{-\gamma \eta_{m_e}}(f(e^+)-f(e^-))^2
			\geq ce^{-\gamma\mathrm{osc}(\eta,2^{-n+1}\zeta_n^{-1},K)}\int_{K}\exp(-\gamma\eta_y)|\nabla f|^2dy.
		\end{equation}
	\end{lem}
	\begin{proof}
		We first extend $f$ to $\mathbb Z_n^2$ as follows: for $v \in \mathbb Z_n^2$, we let $v_K$ be the closest point to $v$ in $K\cap \mathbb Z_n^2$ (pick an arbitrary one when with multiple choices) and we let $f(v) = f(v_K)$. Now we interpolate $f$ linearly on $K$, as follows. For each mesh $B$ such that $K\cap B\neq \emptyset$, we divide it into two isosceles right triangles. For each $x\in K$, we let $u, v, w$ be the vertices of the isosceles right triangle that contains $x$. Then there exist $\lambda_1,\lambda_2,\lambda_3 \in [0,1]$ such that $x=\lambda_1u+\lambda_2v+\lambda_3w$, and we define $f(x)=\lambda_1f(u)+\lambda_2f(v)+\lambda_3f(w)$. Thus we have
		\begin{equation}\label{dirieq2}
			|\nabla f(x)|^2\leq 10\cdot2^{2n}\zeta_n^2\left((f(u)-f(v))^2+(f(u)-f(w))^2\right).
		\end{equation}
		Therefore, for some absolute constant $C$,
		\begin{equation}\label{dirieq1}
			\begin{aligned}
				\int_{K}\exp(-\gamma\eta_y)|\nabla f|^2dy
				&\leq \sum_{\Delta}\int_{\Delta}\exp(-\gamma\eta_y)|\nabla f|^2dy\\
				&\leq Ce^{\gamma \mathrm{osc}(\eta,2^{-n+1}\zeta_n^{-1},K)}\sum_{e\in E(K)}e^{-\gamma \eta_{m_e}}(f(e^+)-f(e^-))^2,
			\end{aligned}
		\end{equation}
		where $\sum_{\Delta}$ denotes summing over all isosceles right triangle $\Delta$ such that $\Delta\cap K \neq \emptyset$, and the second inequality follows from \eqref{dirieq2} (note that we denote by $e^\pm$ the two endpoints of $e$). The inequality \eqref{dirieq1} then implies \eqref{dirieq}.
	\end{proof}

	We now prove Proposition \ref{rsw} following \cite{dubedat2020liouville}. We will need an approximate conformal invariance property of $\phi_{\delta}$. Consider a conformal map $F$ between two bounded, convex and simply-connected open sets $U$ and $\tilde U$ such that $|F^{\prime}|\geq 1$ on $U$ and $\|F\|_{U}<\infty$. We consider another field
	\begin{equation}
		\tilde{\phi}_{\delta}(x)=\int_{\delta^2}^{1}\int_{\mathbb{R}^2}p_{\frac{t}{2}}(x,y)\tilde{W}(dy,dt),
	\end{equation}
	where $\tilde{W}$ is a space-time white noise that we will couple with $W$ in the following way. For $y\in U$, $t\in (0,\infty)$, let $\tilde{y}=F(y)\in \tilde U$ and $\tilde{t}=t|F^{\prime}(y)|^2$ and set $\tilde{W}(d\tilde{y},d\tilde{t})=|F^{\prime}(y)|^2W(d\tilde{y},d\tilde{t})$. The rest of the white noises are chosen to be independent, i.e., $W|_{U^c\times (0,\infty)}$, $W|_{U\times (0,\infty)}$ and $\tilde{W}|_{\tilde U^c\times (0,\infty)}$ are jointly independent.
	\begin{prop}\label{CI}
		\textup{(\hspace{-0.15mm}\cite[Lemma~6]{ding2020tightness})}
		For any compact set $K\subset U$, there exists a smooth Gaussian field $\phi_\delta^L$ whose $L^{\infty}$-norm on $K$ has uniform Gaussian tails, and there exists a smooth Gaussian field $\phi_{\delta}^H$ with uniformly bounded pointwise variance (in $\delta$ and $x\in K$). Furthermore, $\phi_{\delta}^H$ is independent of $(\phi_{\delta},\phi_{\delta}^L)$, and
		\begin{equation}
			\tilde{\phi}_{\delta}(F(x))-\phi_{\delta}(x)=\phi_{\delta}^L(x)+\phi_{\delta}^H(x).
		\end{equation}
	\end{prop}

	Let $K\subset U$ be a compact set and let $A,B$ be two boundary arcs of $K$. Let $\partial_{\mathrm{out}}(A) = \{x\in K^c\cap \mathbb Z_n^2: \exists e\in E(\mathbb Z_n^2) \mbox{ such that } x\sim e \And e\cap A\neq \emptyset\}$ be the corresponding outer boundary in $\mathbb Z_n^2$ and let $\partial_{\mathrm{in}}(A) = \{x\in K\cap \mathbb Z_n^2: \exists e\in E(\mathbb Z_n^2) \mbox{ such that } x\sim e \And e\cap A\neq \emptyset\}$ be the corresponding inner boundary in $\mathbb Z_n^2$. Here we denote by $x\sim e$ if $x$ is an endpoint of $e$, and in $e \cap A$ we view $e$ as the line segment on $\mathbb R^2$ with endpoints given by the endpoints of the edge $e$. Let $R$ be the effective resistance between $\partial_{\mathrm{out}} A$ and $\partial_{\mathrm{out}} B$ in $K\cap \mathbb Z_n^2$ with respect to the field $\phi_{n}$. We denote $\tilde{A}:=F(A)$, $\tilde{B}:= F(B)$ and $\tilde{K}:=F(K)$, and we let $\tilde R$ be
	the effective resistance between $\partial_{\mathrm{in}}\tilde{A}$ and $\partial_{\mathrm{in}}\tilde{B}$ in $\tilde{K}$ with respect to the field $\tilde{\phi}_{n}$. See Figure \ref{fig16} for an illustration.
	\begin{figure}
		\centering
		\includegraphics[width=120mm]{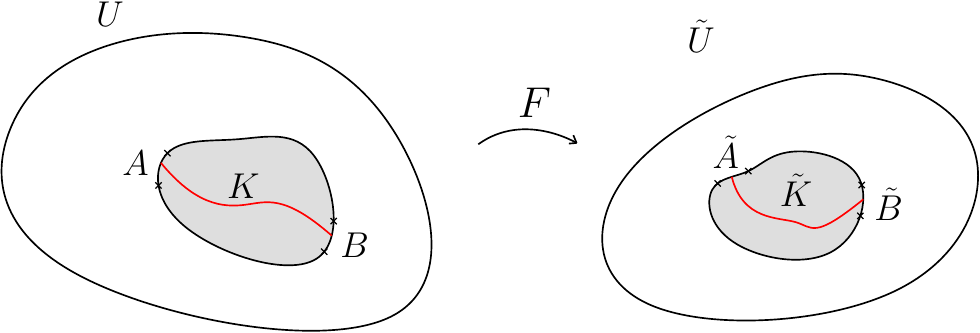}
		\caption{Illustration of domains and the conformal map.}
		\label{fig16}
	\end{figure}
	
	\begin{lem}\label{rswlem3}
		If $\mathbb{P}(R\leq l) \geq \epsilon$ for some $l>0$ and $\epsilon>0$, then 
		\begin{equation}
			\mathbb{P}(\tilde{R}\leq \tilde{l})\geq \epsilon/8.
		\end{equation}
		where $\tilde{l}=C_2le^{C_2\gamma\sqrt{\log(C_2/\epsilon)}}$ for some $C_2=C_2(U,\tilde U,A,B,K,F)$.
	\end{lem}
	\begin{proof}
		We consider the function $f$ that achieves the infimum in \eqref{condef}, i.e., $f(\partial_{\mathrm{in}}(\tilde A))=1$, $f(\partial_{\mathrm{in}}(\tilde B))=0$, and
		\begin{equation}\label{diri1}
			\tilde{R}^{-1}=\sum_{e\in E(\tilde K)}e^{-\gamma \tilde \phi_n(m_e)}(f(e^+)-f(e^-))^2.
		\end{equation}
		Extending the function $f$ as in Lemma \ref{diri}, we then get from Lemma \ref{diri} and \eqref{diri1} that
		\begin{equation}\label{rswlemeq4}
			\begin{aligned}
				\tilde{R}^{-1}
				\geq ce^{-\gamma\mathrm{osc}(\tilde{\phi}_{n},2^{-n+1}\zeta_n^{-1},\tilde U)}\int_{\tilde K}\exp(-\gamma\tilde \phi_{n}(x))|\nabla f|^2dx.
			\end{aligned}
		\end{equation}
		We now define a reasonable approximation for the potential function on $K\cup \partial_{\mathrm{out}} A \cup \partial_{\mathrm{out}} B$ by
		\begin{equation}
			g(v)=
			\begin{cases}
				f\circ F(v), &\mbox{ if }v\in K,\\
				1, &\mbox{ if } v\in \partial_{\mathrm{out}} A,\\
				0, &\mbox{ if } v\in \partial_{\mathrm{out}} B.
			\end{cases}
		\end{equation}
		By Proposition \ref{CI}, we have the decomposition
		\begin{equation}\label{decom}
			\tilde{\phi}_n\circ F=\phi_n+\phi_n^L+\phi_n^H,
		\end{equation}
		where conditioned on $(\phi_n, \phi_n^L)$ we have $\phi_n^H$ is a Gaussian process with bounded point-wise variance, and in addition $\phi_n^L$ is a Gaussian field whose $L^\infty$-norm has uniform Gaussian tails. Then by \eqref{condef} and the mean value property,
		$$
			\begin{aligned}
				&R^{-1}
				\leq \sum_{e\in E}e^{-\gamma \tilde \phi_n(m_e)}(g(e^+)-g(e^-))^2
				=\sum_{e\in E}e^{-\gamma \tilde \phi_n(m_e)}|\nabla g(u_e)|^22^{-2n}\zeta_n^{-2}\\
				&\leq C_3e^{C\gamma \mathrm{osc}(\phi_{n},2^{-n+1}\zeta_n^{-1},U)}\sum_{\Delta}\exp(-\gamma \phi_{n}\circ F^{-1}(x_{\Delta}))|\nabla f(x_{\Delta})|^22^{-2n}\zeta_n^{-2}\\
				&=C_3e^{C\gamma \mathrm{osc}(\phi_{n},2^{-n+1}\zeta_n^{-1},U)}\sum_{\Delta}\exp\left(-\gamma \tilde{\phi}(x_{\Delta})+\gamma\phi_n^L\circ F^{-1}(x_{\Delta})+\gamma\phi_n^H\circ F^{-1}(x_{\Delta})\right)|\nabla f(x_{\Delta})|^22^{-2n}\zeta_n^{-2},
			\end{aligned}
		$$
		where $u_e$ is some point in $e\cap K$ (given by the mean value property), $C_3=C_3(F)$, $\sum_{\Delta}$ is summing over all isosceles right triangle $\Delta$ such that $\Delta\cap \tilde K\neq \emptyset$, and $x_{\Delta}$ is an arbitrary point in $\Delta$ (this is fine since the gradient is a constant function in each triangle). By \eqref{decom} and by Markov's inequality we see that given $(\phi_n, \phi_n^L)$, with conditional probability at least 1/2,
		\begin{equation}\label{rswlemeq5}
			R^{-1} \leq 2C_3e^{C\gamma \mathrm{osc}(\phi_{n},2^{-n+1}\zeta_n^{-1},U)}\sum_{\Delta}\exp\left(-\gamma \tilde{\phi}(x_{\Delta})+\gamma\phi_n^L\circ F^{-1}(x_{\Delta})\right)|\nabla f(x_{\Delta})|^22^{-2n}\zeta_n^{-2},
		\end{equation}	
		where we used the fact that $\gamma<1$ to bound the moment generating function. We write $\mathcal E$ as the event in \eqref{rswlemeq5}. By Proposition \ref{gradient} and (properties described after) \eqref{decom}, if we set $x=C_4\sqrt{\log(C_4/\epsilon)}$ (for some $C_4=C_4(U, \tilde U)$), we have the following bounds:
		\begin{equation}\label{rswlemeq1}
			\mathbb{P}(\|\phi_n^L\|_{U}\geq x)\leq \epsilon/4,
		\end{equation}
		\begin{equation}\label{rswlemeq2}
			\mathbb{P}(\mathrm{osc}(\phi_{n},2^{-n+1}\zeta_n^{-1},U)\geq x)\leq \epsilon/4,
		\end{equation}
		\begin{equation}\label{rswlemeq3}
			\mathbb{P}(\mathrm{osc}(\tilde{\phi}_{n},2^{-n+1}\zeta_n^{-1}, \tilde U)\geq x)\leq \epsilon/8.
		\end{equation}
		Now we define the event $\mathcal{F}$ by 
		$$\mathcal{F}=\{R\leq l\}\cap\{\|\phi_n^L\|_{U}< x\}\cap\{\mathrm{osc}(\phi_{n},2^{-n+1}\zeta_n^{-1},U)< x\}.$$
		Combining \eqref{rswlemeq1}, \eqref{rswlemeq2} with the fact $\mathbb P(R<l)>\epsilon$ we get that $\mathbb{P}(\mathcal{F})\geq \epsilon/2$. Since $\mathcal F$ depends only on $(\phi_n,\phi_n^L)$, by \eqref{rswlemeq5} we have $\mathbb{P}(\mathcal E\cap \mathcal F)\geq \epsilon/4$. Combining this with \eqref{rswlemeq3} we have
		\begin{equation}\label{rswlemeq6}
			\mathbb{P}(\mathcal E\cap \mathcal F\cap \{\mathrm{osc}(\tilde{\phi}_{n},2^{-n+1}\zeta_n^{-1}, \tilde U)< x\})\geq \epsilon/8.
		\end{equation}
		On the event as in \eqref{rswlemeq6}, by \eqref{rswlemeq4} we get that
		\begin{equation}
			\tilde{R}^{-1}\geq cC_3^{-1}e^{-C\gamma x}R^{-1}\geq cC_3^{-1}e^{-C\gamma x}l^{-1},
		\end{equation}
		which then completes the proof.
	\end{proof}
	
	\begin{lem}\label{rswlem1}
		\textup{(\hspace{-0.15mm}\cite[Lemma~4.8]{dubedat2020liouville})}
		For $0<a<b$, there exist $j = j(b/a)$ and $j$ rectangles isometric to $B(a/2,b/2)$ such that if $\pi$ is a left-right crossing of the rectangle $B(a,b)$, then at least one of the $j$ rectangles is crossed in the hard direction by a sub-path of that crossing.
	\end{lem}
	
	\begin{lem}\label{rswlem2}
		\textup{(Step 1 in the proof of \cite[Theorem~3.1]{dubedat2020liouville})}
		If $a/b<1$ and $\tilde a/\tilde b>1$, then there exist $m,k \geq 1$ and two ellipses $\mathtt E,\tilde{\mathtt E}$ with marked arcs $(XY),(ZW)$ for $\mathtt E$ and $(\tilde X\tilde Y),(\tilde Z\tilde W)$ for $\tilde{\mathtt E}$ such that:
		\begin{enumerate}
			\item Any left-right crossing of $B(a2^{-k},b2^{-k})$ is a crossing of $\mathtt E$.
			\item Any crossing in $\tilde{\mathtt E}$ from $(\tilde X\tilde Y)$ to $(\tilde Z\tilde W)$ is a left-right crossing of $B(\tilde a,\tilde b)$.
			\item When dividing the marked sides of $\mathtt E$ into m sub-arcs of equal length, for any pair of such subarcs (one on each side), there exists a conformal map $F:\mathtt E\to \tilde{\mathtt E}$ and the pair of subarcs is mapped to subarcs of the marked sides of $\tilde{\mathtt E}$.
			\item For each pair, the associated map $F$ extends to a conformal equivalence $U\to \tilde U$ where $\mathtt E\subset U$, $\tilde{\mathtt E}\subset \tilde U$ and $|F^{\prime}|\geq 1$ on $U$.
		\end{enumerate}
	\end{lem}
	
	\begin{lem}\label{rswlem4}
		If $a/b<1$ and $\tilde a/\tilde b>1$, then there
		exists $C_5=C_5(a,b,\tilde a,\tilde b)$ such that for all $\epsilon>0$,
		\begin{enumerate}
			\item\label{rswclaim1} if $\mathbb{P}(R^{n}_{a,b}\leq l)\geq \epsilon$, then $\mathbb{P}(R^{n}_{\tilde a,\tilde b}\leq C_5le^{C_5\gamma\sqrt{\log(C_5/\epsilon)}})\geq \epsilon/C_5$;
			\item\label{rswclaim2} if $\mathbb{P}(R^{n}_{\tilde a,\tilde b}\leq l)\leq 1-\epsilon$, then $\mathbb{P}(R^{n}_{a,b}\leq C_5^{-1}le^{-C_5\gamma\sqrt{\log(C_5/\epsilon)}})\leq 1-\epsilon/C_5$.
		\end{enumerate}
	\end{lem}

	\begin{figure}
		\centering
		\includegraphics[width=80mm]{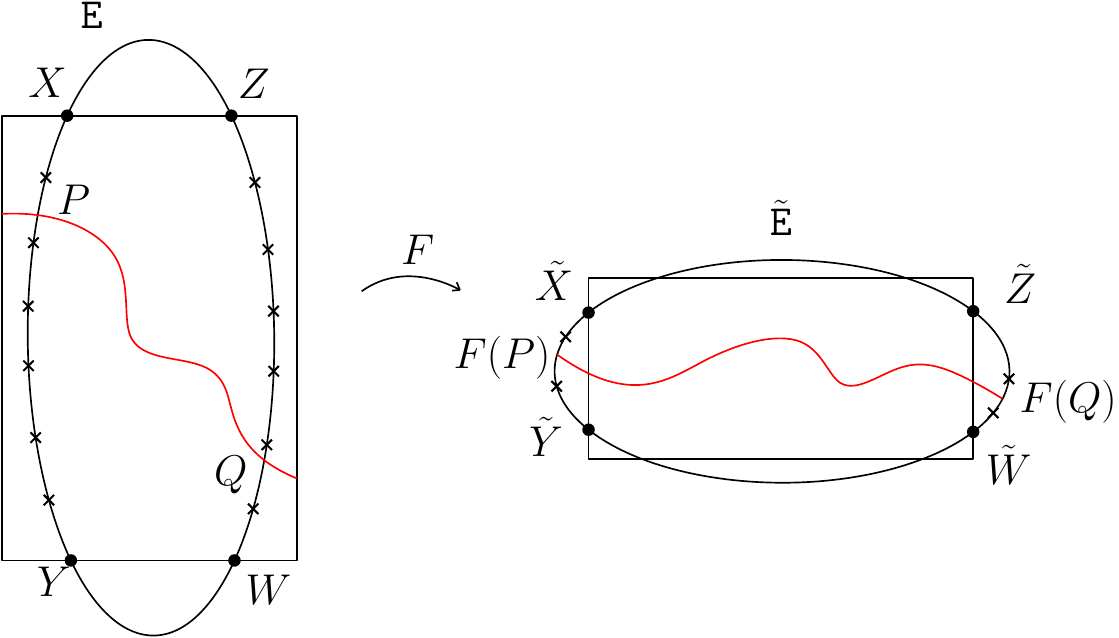}
		\caption{Illustration of Lemma \ref{rswlem4}.}
		\label{fig5}
	\end{figure}

	\begin{proof}
		We first prove Claim \ref{rswclaim1}. Let $m, k, \mathtt E, \tilde{\mathtt E}$ be chosen as in Lemma \ref{rswlem2}. By Lemma \ref{rswlem1}, there exists $j$ depending only on $a/b$ such that any left-right crossing of $B(a,b)$ must cross one of these $j$ rectangles isometric to $B(a/2,b/2)$, where these rectangles are chosen as in Lemma \ref{rswlem1}. By Proposition \ref{parallel}, it holds that $\mathbb{P}(R^{n}_{a/2,b/2}\leq jl)\geq \epsilon/j$. Iterating the preceding inequality, we have $\mathbb{P}(R_{a2^{-k},b2^{-k}}^{n}\leq j^kl)\geq \epsilon/j^k$. Thus by Lemmas \ref{parallel} and \ref{rswlem2}, one of the effective resistances between two sub-arcs $P,Q$ of $\mathtt E$ (one on each side, as illustrated in Figure \ref{fig5}) is at most $lj^km$ (recall that $\mathtt E, m, k$ are given in Lemma \ref{rswlem2}), with probability at least $\epsilon/j^km$. I.e., $$\mathbb{P}(R_{\mathtt E}(P,Q)\leq lj^km)\geq \epsilon/j^km,$$ where we denote by $R^n_{\mathtt E}(P, Q)$ the effective resistance between $P$ and $Q$ for the network on $\mathtt E\cap \mathbb Z^2_n$ with respect to the field $\phi_n$ (similar notations apply in what follows too). Thus by Lemma \ref{rswlem3}, for the conformal map $F$ from $\mathtt E$ to $\tilde{\mathtt E}$ and for some $C_5=C_5(a,b,\tilde a,\tilde b)$,
		$$\mathbb{P}(R_{\tilde{\mathtt E}}(F(P),F(Q))\leq C_5le^{C_5\gamma\sqrt{\log(C_5/\epsilon)}})\geq \epsilon/C_5.$$
		Since any crossing from $F(P)$ to $F(Q)$ in $\tilde{\mathtt E}$ is a left-right crossing of $B(\tilde a,\tilde b)$, we have
		\begin{equation}
			\mathbb{P}(R^{n}_{\tilde a,\tilde b}\leq C_5le^{C_5\gamma\sqrt{\log(C_5/\epsilon)}})\geq \epsilon/C_5.
		\end{equation}
		Thus we obtain Claim \ref{rswclaim1}. Claim \ref{rswclaim2} follows immediately by Lemma \ref{duallem} (modulo a slight change of $a,b,\tilde a,\tilde b$).
	\end{proof}
	
	\begin{proof}[Proof of Proposition \ref{rsw}]
		Proposition \ref{rsw} holds immediately from Lemma \ref{rswlem4}.
	\end{proof}
	
	\section{Tail estimates for resistances}\label{sec5}
	In this section, we prove various tail estimates for resistances in terms of their quantiles. Similar to tail estimates in \cite{ding2019lioville,dubedat2020liouville,ding2020tightness}, our proof employs percolation arguments, where we obtain finite-range dependence on the coarse-grained lattice after removing a ``coarse field''.

	\begin{lem}\label{taillem1}
		There exist absolute positive constants $p_0, C, c$ such that the following holds for all $p \leq p_0$. For all $n\geq 1$ and $k\geq 1$,
		\begin{equation}
			\mathbb{P}(R^{n}_{8k,k}>Cl^{(n)}(\phi,1-p))\leq C\exp(-ck).
		\end{equation}
	\end{lem}
	\begin{figure}
		\centering
		\includegraphics[height=60mm,width=60mm]{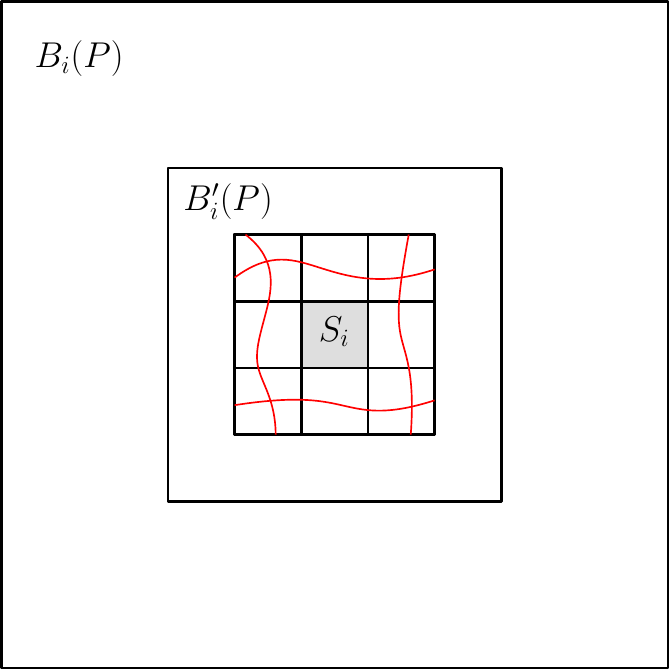}
		\caption{Illustration of $B_i(P)$ and $B_i^\prime(P)$.}
		\label{fig14}
	\end{figure}
	\begin{proof}
		By noting that $l^{(n)}(\phi,1-p)$ is decreasing in $p$, we may fix a sufficiently small $p$. For each unit square $S$ in $\mathbb{Z}^2$, we consider the four hard crossings of rectangles of dimensions $3\times 1$ surrounding $S$, and we write $A_n(\phi_n,S)$ as the sum of their effective resistances in $\mathbb Z_n^2$ associated with the field $\phi_{n}$ (that is, each effective resistance is between the two shorter sides of a surrounding rectangle). See Figure \ref{fig14} for an illustration.
		
		Consider a graph $\mathrm{H}$ with all unit squares in $\mathbb{Z}^2\cap B(8k, k)$ as its vertices, where two squares are adjacent if they share an edge in $\mathbb{Z}^2$. We fix a sufficiently large constant $C_p$ that depends only on $p$, with its value determined later. We say a square $S$ is open if $A_{n}(\phi_n,S)\leq C_p^2l^{(n)}(\phi,1-p)$, and we say an edge in $\mathrm{H}$ is open if both of its endpoints are open. With $w(e)=\mathds{1}_{\{e\mbox{ is open}\}}$, the maximal flow as in the left-hand side of \eqref{maxflowmincuteq} is the maximal number of disjoint left-right open paths in H. Thus, by Lemma \ref{maxflowmincut}, either there exist $k/2$ disjoint left-right open paths, or there exists an up-down path in the dual graph with at most $k/2$ open edges (see Figure \ref{fig6} for an illustration). We will prove that the latter event is rare via a percolation argument.
		
		Let $\Gamma$ be the collection of self-avoiding paths in $\mathrm{H}$ with lengths at least $k$. We say a path $P \in \Gamma$ with length $L$ is bad if it has at most $L/2$ open edges. We claim that for $P\in \Gamma$ with length $L$ and some absolute constant $C_0$,
		\begin{equation}\label{eq-prob-P-bad}
			\mathbb P(P \mbox{ is bad}) \leq \binom{2L}{L/100}(2^{L/100}(2C_0p)^{L/200}+e^{-10L}).
		\end{equation}
		Provided with \eqref{eq-prob-P-bad}, we now complete the proof of the lemma. By a simple union bound, we have
		\begin{equation}
			\mathbb{P}(\mbox{there exists a bad } P\in \Gamma)\leq \sum_{L\geq k} 2^{2L}(2^{L/100}(2C_0p)^{L/200}+e^{-10L})\cdot 8k^23^L\leq Ce^{-ck},
		\end{equation}
		which holds for $p$ sufficiently small. Thus with probability at least $1-Ce^{-ck}$, we can find $k/2$ disjoint open paths, and at least $k/4$ of them have lengths at most $100k$. By Proposition \ref{series} and the parallel's law, on the aforementioned event we have that $R^{n}_{8k,k}\leq Cl^{(n)}(\phi,1-p)$ where $C = 10^{4}C_p^2$, completing the proof of the lemma.
		
		We now prove \eqref{eq-prob-P-bad}. For each $P \in \Gamma$ with length $L$, let $\mathfrak B(P)$ be the collection of sets of boxes $\mathbf B(P) = \{B_i(P): 1\leq i\leq \ell\}$ with $\ell = L/100$ such that
		(i) $B_i(P)$ is a box with side length 5 concentric with a square $S_i$, where $S_i$ has a common edge with $P$;
		(ii) the mutual distances of these boxes are at least 1. In addition, we let $B'_i(P)$ be the concentric box of $B_i(P)$ with side length 3. 
		
		\begin{figure}
			\centering
			\includegraphics[height=50mm,width=100mm]{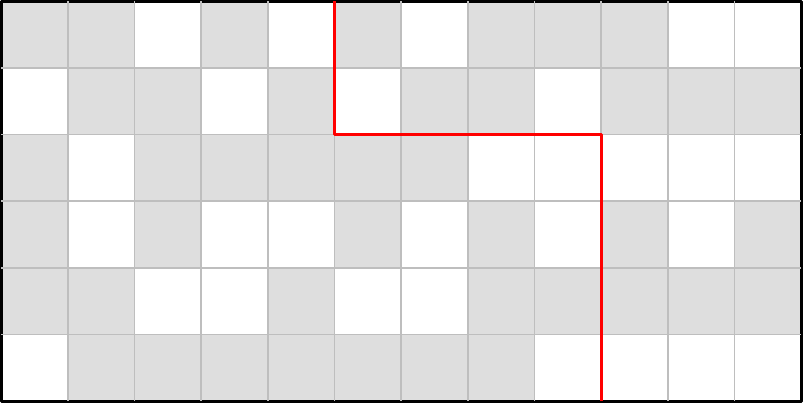}
			\caption{Illustration of the supercritical percolation. The gray squares are open, and the red path is one of the up-down paths we get from the Max-Flow-Min-Cut Theorem.}
			\label{fig6}
		\end{figure}
	
		If $P$ is bad, then there exists $\mathbf B(P) \in \mathfrak B(P)$ such that for $B_i(P)\in \mathbf B(P)$, we have $A_n(\phi_n,S_i)\geq C_p^2l^{(n)}(\phi,1-p)$ (in this case, we say that $\mathbf B(P)$ is a certificate for bad).	
		In light of this, we now proceed to prove \eqref{eq-prob-P-bad} by taking a union bound over $\mathbf B(P) \in \mathfrak B(P)$. The main ingredient is to bound the probability for a fixed $\mathbf B(P)$ to be a certificate for bad, where the key issue is to deal with the correlation among these boxes. To this end, we fix a $\mathbf B(P) = \{B_i(P): 1\leq i\leq \ell\}$ and we define
		\begin{equation}
			\phi^i_n(x)=\int_{2^{-2n}}^{1}\int_{B_i(P)}p_{\frac{t}{2}}(x,y)W(dy,dt) \And \phi^i_{\mathrm c}=\phi_{n}-\phi^i_n.
		\end{equation}		
		We define $A_n(\phi^i_n,S)$ as $A_n(\phi_n,S)$ but with $\phi^i_n$ in place of $\phi_n$. We see that, in order for a box $S_i$ to be not open, either of the following has to happen:		
		\begin{enumerate}
			\item[(1)] $\max_{x_i\in B_i^\prime(P)}\phi_{\mathrm c}^i(x_i)\geq \log C_p$;
			\item[(2)] $A_n(\phi^i_n,S_i)\geq C_pl^{(n)}(\phi,1-p)$.
		\end{enumerate}		
		Therefore, in order for $\mathbf B(P)$ to be a certificate for bad, either of the following has to happen:		
		\begin{enumerate}
			\item[(a)] at least half boxes satisfy $\max_{x_i\in B_i^\prime(P)}\phi_{\mathrm c}^i(x_i)\geq \log C_p$;
			\item[(b)] at least half boxes satisfy $A_n(\phi^i_n,S_i)\geq C_pl^{(n)}(\phi,1-p)$.
		\end{enumerate}
	
		We first bound the probability for (a). 
		We apply Lemma \ref{coarse} with $U_i=B_i(P)$, $V_i=B_i^\prime(P)$ and $a=1$, and we get that for $C_p$ larger than a certain absolute constant,
		\begin{equation}\label{tail12}
			\mathbb P(\mbox{event in (a)})\leq e^{-10^{3}\ell}\leq e^{-10L}.
		\end{equation}
		
		We now bound the probability for (b).
		Since $\phi_n=\phi^i_n+\phi^i_{\mathrm c}$, we have
		\begin{equation}\label{tail11}
			\mathbb{P}(A_n(\phi^i_n,S_i)\geq C_pl^{(n)}(\phi,1-p))\leq \mathbb{P}(A_n(\phi_n,S_i)\geq \sqrt{C_p}l^{(n)}(\phi,1-p))+\mathbb{P}(\|\phi^i_{\mathrm c}\|_{B_i^\prime(P)}\geq \frac{1}{2}\log C_p),
		\end{equation}
		where we used the fact that $\gamma<1$. By Proposition \ref{rsw}, for some absolute constant $C_0$, we may take $C_p$ sufficiently large, depending only on $p$, such that $\mathbb{P}(A_n(\phi_n,S_i)\geq \sqrt{C_p}l^{(n)}(\phi,1-p))<C_0p$. In addition, by \eqref{coar1} we may take $C_p$ sufficiently large, depending only on $p$, such that $\mathbb{P}(\|\phi^i_{\mathrm c}\|_{B_i^\prime(P)}\geq \frac{1}{2}\log C_p)<C_0p$. Combining \eqref{tail11} with a union bound over all choices of boxes, we have
		\begin{equation}\label{tail13}
			\mathbb P(\mbox{event in (b)} )\leq \binom{\ell}{\ell/2}(2C_0p)^{\ell/2}\leq 2^{L/100}(2C_0p)^{L/200}.
		\end{equation}
		Combining \eqref{tail12} and \eqref{tail13}, we complete the proof of \eqref{eq-prob-P-bad} by a union bound over all $\mathbf B(P)$. Thus, we complete the proof of the lemma.
	\end{proof}

	\begin{prop}\label{uptail}
		There exist absolute positive constants $C, c$ such that the following holds for all $n>1$, $p\leq p_0$ and $T>3$,
		\begin{equation}
			\mathbb{P}\left(R^{n}_{8,1}>T\Lambda_{n}(\phi,p)\right)\leq C\exp\left(-c\frac{(\log T)^2}{\log\log T}\right).
		\end{equation}
	\end{prop}
	\begin{proof}
		For $m<n$, we consider the rescaled lattice $\mathbb{Z}_{n-m}^2$. By Lemma \ref{taillem1}, we can choose $C, c, p$ such that the following holds:
		\begin{equation}\label{tail14}
			\mathbb{P}\left(R^{n-m}_{8\cdot2^{m},2^{m},\zeta_{n-m}}\geq C2^{2m}l^{(n-m)}(\phi,1-p)\right)\leq Ce^{-c2^m}.
		\end{equation}
		By Proposition \ref{mesh}, since $\zeta_n, \zeta_{n-m} \geq \sqrt{n-m}$, we get that (we may increase the value of $C$ and decrease the value of $c$ if necessary)
			\begin{equation}\label{tail1}
			\mathbb{P}\left(R^{n-m}_{8\cdot2^{m},2^{m},\zeta_{n}}\geq Ce^{C\sqrt{2^m}}2^{2m}l^{(n-m)}(\phi,1-p)\right)\leq C2^{2m}e^{-c2^m}.
		\end{equation}
		By the scaling property of $\phi$ we have $R^{m,n}_{8,1,\zeta_n}\overset{d}{=}R^{n-m}_{8\cdot2^{m},2^{m},\zeta_{n}}$. Combined with \eqref{tail1}, this yields that
		\begin{equation}
			\mathbb{P}\left(R^{m,n}_{8,1,\zeta_n}\geq e^{C\sqrt{2^m}}\cdot C2^{2m}l^{(n-m)}(\phi,1-p)\right)\leq C2^{2m}e^{-c2^m}.
		\end{equation}
		Notice that $\phi_n=\phi_m+\phi_{m,n}$, and thus we have
		\begin{equation}
			\begin{aligned}
				&\mathbb{P}\left(R^{n}_{8,1}\geq e^{s\sqrt{m}+C\sqrt{2^m}}l^{(n-m)}(\phi,1-p)\right)\\
				&\leq
				\mathbb{P}\left(\|\phi_{0,m}\|_{B(8,1)}\geq Cm+s\sqrt{m}\right)+\mathbb{P}\left(R^{m,n}_{8,1}\geq e^{C\sqrt{2^{m}}}l^{(n-m)}(\phi,1-p)\right)\\
				&\leq Ce^{-cs^2}+Ce^{-c2^m},
			\end{aligned}
		\end{equation}
		where in the first inequality we used the fact that $\gamma<1$, and in the second inequality we applied Lemma \ref{max} and may have further adjusted the constant $C$ to absorb other factors. In the case that $s\in(2,2^{n/2})$, take $m=\max\{i:2^i\leq s^2\}$ and we get that
		\begin{equation}\label{tail3}
			\mathbb{P}\left(R^{n}_{8,1}\geq Ce^{Cs\sqrt{\log s}}\Lambda_n(\phi,p)\right)\leq e^{-cs^2},
		\end{equation}	
		where we used Lemma \ref{prio} to substitute $l^{(n-m)}$ by $\Lambda_n(\phi,p)$. In the case that $s>2^{n/2}$, by Proposition $\ref{max}$ we have (recalling that we have assumed $\gamma < 1$)
		$$
			\mathbb{P}\left(R^{n}_{8,1}\geq Ce^s\Lambda(\phi,p)\right)\leq\mathbb{P}\left(\|\phi_{n}\|_{B(8,1)}\geq s\right)\leq C4^n\exp \left(-\frac{s^2n}{(n+C\sqrt{n})^2}\right)\leq C\exp\left(-c\frac{s^2}{\log s}\right).
		$$
		Combined with \eqref{tail3}, it completes the proof.
	\end{proof}
	
	The lower tail estimate can now be derived using duality.
	
	\begin{prop}\label{lowtail}
		There exist absolute positive constants $C, c$ such that the following holds for all $n>1$, $p\leq p_0$ and $T>3$,
		\begin{equation}\label{taillow}
			\mathbb{P}\left(R^{n}_{1,8}\leq T^{-1}\Lambda_{n}(\phi,p)^{-1}\right)\leq C\exp\left(-c\frac{(\log T)^2}{\log\log T}\right).
		\end{equation}
	\end{prop}
	\begin{proof}
		By the duality as in \eqref{dual1}, we can convert the lower tail of the resistance to its upper tail, where the dimensions of the rectangle will be changed from $1\times 8$ to $8\times 1$, modulo a slight modification on the boundary layer when taking the dual graph.
	\end{proof}
	
	\begin{cor}\label{tailpsi}
		There exist absolute positive constants $C,c$ such that for all $n>0$, $p\leq p_0$ and $T>3$, we have
		\begin{equation}\label{uptail2}
			\mathbb{P}\left(\tilde{R}^n_{8,1}>T\tilde \Lambda_{n}(\psi,p)\right)\leq C\exp\left(-c\frac{(\log T)^2}{\log\log T}\right),
		\end{equation}
		\begin{equation}\label{lowtail2}
			\mathbb{P}\left(\tilde{R}^n_{1,8}\leq T^{-1}\tilde \Lambda_{n}(\psi,p)^{-1}\right)\leq C\exp\left(-c\frac{(\log T)^2}{\log\log T}\right).
		\end{equation}
	\end{cor}
	\begin{proof}
		By Proposition \ref{normprop} and the observation that
		$$
		|\log R^{n}_{8,1}-\log \tilde{R}^n_{8,1}|\leq \gamma\|\phi_{n}-\psi_n\|_{B(8,1)}\leq \|\phi_{n}-\psi_n\|_{B(8,1)},
		$$
		we can deduce \eqref{uptail2} from Lemma \ref{comp} and Proposition \ref{uptail}. Similarly we can deduce \eqref{lowtail2} from Lemma \ref{comp} and Proposition \ref{lowtail}.
	\end{proof}
	
	We have the following bound on the tail probability of effective resistances around and across an annulus. For a square annulus, we denote its inner radius as a half of the side length of its inner boundary. Similarly, we denote its outer radius as a half of the side length of its outer boundary.
	\begin{prop}\label{taillem}
		For $n>i>0$ and $b>a>1$, consider an annulus $\mathcal{A}$ with inner radius $a$ and outer radius $b$. Then for all $0<q<1$, there exist constants $C_1=C_1(a,b)$ and $c_1=c_1(a,b)$, such that for all $n>0$, $p\leq p_0$ and $T>3$, 
		\begin{equation}\label{tailaround}
			\mathbb{P}\left(R^{i,n}(\mbox{around } \mathcal{A})\geq T\Lambda_{n-i}(\phi,p)\right)\leq C_1\exp\left(-c_1\frac{(\log T)^2}{\log\log T}\right),
		\end{equation}
		\begin{equation}\label{tailacross}
			\mathbb{P}\left(R^{i,n}(\mbox{across } \mathcal{A})\leq T^{-1}\Lambda_{n-i}(\phi,p)^{-1}\right)\leq C_1\exp\left(-c_1\frac{(\log T)^2}{\log\log T}\right).
		\end{equation}
	\end{prop}
	\begin{proof}
		By Proposition \ref{uptail}, for some constants $C_2=C_2(a,b)$ and $c_2=c_2(a,b)$ we have
		\begin{equation}
			\mathbb{P}\left(R^{n-i}_{2b,b-a,\zeta_{n-i}}\geq Tl^{(n-i)}(\phi,p)\right)\leq  C_2\exp\left(-c_2\frac{(\log T)^2}{\log\log T}\right).
		\end{equation}
		Thus by Proposition \ref{mesh}, we have
		\begin{equation}
			\mathbb{P}\left(R^{n-i}_{2b,b-a,\zeta_{n}}\geq Tl^{(n-i)}(\phi,p)\right)\leq  CC_2\exp\left(-c_2\frac{(\log T)^2}{\log\log T}\right).
		\end{equation}
		By the scaling property of $\phi$, we have $R^{n-i}_{2b,b-a,\zeta_{n}}\overset{d}{=}R^{i,n}_{2b\cdot 2^{-i},(b-a)\cdot 2^{-i},\zeta_{n}}$, and thus we get that
		\begin{equation}\label{tail15}
			\mathbb{P}\left(R^{i,n}_{2b\cdot 2^{-i},(b-a)\cdot 2^{-i},\zeta_{n}}\geq T\Lambda_{n-i}(\phi,p)\right)\leq CC_2\exp\left(-c_2\frac{(\log T)^2}{\log\log T}\right),
		\end{equation}
		where we applied Lemma \ref{prio} to substitute $l^{(n-i)}(\phi,p)$ by $\Lambda_{n-i}(\phi,p)$. Similarly, by Propositions \ref{lowtail} and \ref{mesh}, for $C_3=C_3(a,b)$ and $c_3=c_3(a,b)$ we get that
		\begin{equation}\label{tail16}
			\mathbb{P}\left(R^{i,n}_{(b-a)\cdot 2^{-i},2b\cdot 2^{-i},\zeta_{n}}\leq T^{-1}\Lambda_{n-i}(\phi,p)^{-1}\right)\leq C_3\exp\left(-c_3\frac{(\log T)^2}{\log\log T}\right).
		\end{equation}
		Notice that we can obtain contours around $\mathcal{A}$ by joining hard crossings through 4 rectangles of dimensions $2b\cdot 2^{-i}\times(b-a)\cdot 2^{-i}$ surrounding $\mathcal{A}$. 
		Then by Proposition \ref{series} and \eqref{tail15} we may obtain \eqref{tailaround}.
		Meanwhile, any crossing through the annulus $\mathcal{A}$ must cross at least one of these aforementioned rectangles in the easy direction. 
		Thus by Proposition \ref{parallel} and \eqref{tail16}, we obtain \eqref{tailacross}.
	\end{proof}
	
	\begin{cor}\label{rswcor}
		For $n>i>0$ and $b>a>1$, consider an annulus $\mathcal{A}$ with inner radius $a$ and outer radius $b$. Then for all $0<q<1$, there exist constants $C_4=C_4(p,a,b)$ and $c_4=c_4(p,a,b)$, such that for all $n>0$ and $p<p_0/2$ (where $p_0$ is the same as in Proposition \ref{rswlem1}), 
		\begin{equation}\label{tail5}
			\mathbb{P}\left(\tilde R^{i,n}(\mbox{around } \mathcal{A})\geq T\tilde \Lambda_{n-i}(\psi,p)\right)\leq C_4\exp\left(-c_4\frac{(\log T)^2}{\log\log T}\right),
		\end{equation}
		\begin{equation}\label{tail6}
			\mathbb{P}\left(\tilde R^{i,n}(\mbox{across } \mathcal{A})\leq T^{-1}\tilde \Lambda_{n-i}(\psi,p)^{-1}\right)\leq C_4\exp\left(-c_4\frac{(\log T)^2}{\log\log T}\right).
		\end{equation}
	\end{cor}
	\begin{proof}
		By Proposition \ref{quancom}, we have $\tilde \Lambda(\psi,p)\geq C_p^{-2}\Lambda(\phi,2p)$. We may then deduce \eqref{tail5} and \eqref{tail6} directly from Propositions \ref{normprop} and \ref{taillem} (applied with $2p$, which is fine since we assumed that $p<p_0/2$).
	\end{proof}
	
	\section{Concentration for resistance}\label{sec6}
	The goal of this section is to prove the following proposition. 
	
	\begin{prop}\label{mainprop}
		Recall the maximal quantile ratio $\Lambda_n(\phi,p)$ as in \eqref{quantile}. There exists a positive constant $\gamma_0$ such that for all $\gamma\in(0, \gamma_0)$ and $p\in (0,1/2)$,
		\begin{equation}\label{mainineq}
			\sup_{n\in\mathbb{N}}\Lambda_n(\phi,p)<\infty.
		\end{equation}
	\end{prop}
	In order to prove Proposition \ref{mainprop}, we will first prove it for quantiles with respect to the field $\psi$. Recall $p_0$ in Proposition \ref{taillem1} and fix $p<p_0/2$. We will inductively prove that
	\begin{equation}\label{induction}
		\tilde \Lambda_{n}(\psi,p)\leq 2.
	\end{equation}
	It is crucial that through our induction, we arrive at the same constant as assumed in the induction hypothesis. Our proof framework follows that for the tightness of the Liouville quantum gravity metric as in \cite{ding2019lioville,dubedat2020liouville,ding2020tightness,ding2023uniqueness}, which is based on the Efron-Stein inequality. We next set up the framework and describe the outline of the proof. Along the way, we will emphasize some major obstacle for dealing with the resistance and explain how this will be tackled in the following subsections. In what follows, we let $\kappa$ be an integer to be determined later. 
	
	\begin{defi}
		Recall $\epsilon_0=1/100$ from $\eqref{psidef}$. We define $\mathcal{S}_\kappa$ to be the collection of $2^{-\kappa}\times2^{-\kappa}$
		dyadic squares which are contained in the	$2^{-\kappa+1}\kappa^{\epsilon_0}$-neighborhood of $B(1)$.
	\end{defi}
	
	For $S\in \mathcal{S}_k$, we set (recall the definition of $\tilde p_{t}(x, y)$ as in \eqref{psidef})
	\begin{equation}
		\psi_{\kappa,n,S}(x)=\int_{2^{-2n}}^{2^{-2\kappa}}\int_{S}\tilde{p}_{\frac{t}{2}}(x,y)W(dy,dt).
	\end{equation}
	Since $\tilde{p}_{\frac{t}{2}}(x,\cdot)$ is supported on the $2^{-\kappa+1}\kappa^{\epsilon_0}$-neighborhood of $x$ for every $t\leq 2^{-2\kappa}$, it follows that $\psi_{\kappa,n,S}(\cdot)$ is supported on the $2^{-\kappa+1}\kappa^{\epsilon_0}$-neighborhood of $S$. In addition, it is obvious that
	\begin{equation}
		\psi_{\kappa,n}=\sum_{S\in \mathcal{S}_\kappa}\psi_{\kappa,n,S} \mbox{ on } B(1).
	\end{equation}
	
	In order to apply the Efron-Stein inequality, we resample $\psi_{\kappa,n,S}$ to get an independent copy $\hat{\psi}_{\kappa,n,S}$. Let $\psi_{n}^S=\psi_{n}-\psi_{\kappa,n,S}+\hat{\psi}_{\kappa,n,S}$. We also resample the top field (i.e., the coarse field) $\psi_{0,\kappa}$ to get an independent copy $\hat{\psi}_{0,\kappa}$ and let $\psi_{n}^{\mathrm{T}}=\psi_{n}-\psi_{0,\kappa}+\hat{\psi}_{0,\kappa}$. Then we write $\tilde{R}^{n,S}_{1,1}$ for the analogue of $\tilde R^n_{1, 1}$ where $\psi_n$ is replaced with $\psi_n^S$, and write $\tilde{R}^{n,\mathrm{T}}_{1,1}$ for the analogue of $\tilde R^n_{1, 1}$ where $\psi_n$ is replaced with $\psi_n^{\mathrm T}$.	
	Write
	$$
			X_S=\left(\log \tilde{R}^{n,S}_{1,1} -\log \tilde{R}^n_{1,1}\right)_+.
	$$
	Then, by the Efron-Stein inequality we have
	\begin{equation}\label{Efron}
			\mathrm{Var}\left[\log \tilde{R}^{n}_{1,1}\right] \leq\sum_{S\in\mathcal{S}_k}\mathbb{E}\left[X_S^2\right]+\mathbb{E}\left[\log \tilde{R}^{n,\mathrm{T}}_{1,1}-\log \tilde{R}^{n}_{1,1}\right]^2.
	\end{equation}
	
	Since $\log \tilde{R}^n_{1, 1}$ is a $\gamma$-Lipschitz function with respect to the $L^\infty$-norm and since the pointwise variance $\mathrm{Var}(\psi_\kappa(x)) \leq \kappa \log 2$, we apply the Gaussian concentration inequality (see \cite{Borell1975,Sudakov1978} and see also e.g. \cite[Theorem~3.25]{van2014probability}) and derive that
	\begin{equation}\label{Efron1}
			\mathbb{E}\left[\log \tilde{R}^{n,\mathrm{T}}_{1,1}-\log \tilde{R}^{n}_{1,1}\right]^2\leq C\gamma^2 \kappa.
	\end{equation}

	The key task is to bound the first term on the right-hand side of \eqref{Efron}. In the analogue of the LQG metric, this crucially relies on the fact that
	\begin{equation}\label{eq-LQG-Property}
		\mbox{ the crossing distance through a smaller box is smaller than that through a bigger box,}
	\end{equation} 
	which implies that the local perturbation of the field only results in a small perturbation of the metric. In our context, the effective resistance crossing through boxes of all sizes are all supposed to be of order 1, and as a result the analogous line of reasoning is invalid. In order to address this, an important observation is that the electric current flowing into a small box (respectively, the corresponding energy generated therein) should be insignificant compared to the strength of the global flow (respectively, the corresponding energy). This can be easily checked if each edge resistance is 1. In our setup, however, this is not at all trivial. A moment of thinking suggests that the key condition for the aforementioned flow/energy decay is that the resistance in neighboring scales are comparable. This raises two entangling issues which complicate our analysis: (i) when comparing resistances in two neighboring scales, we actually need that the crossing resistance in the hard direction is not so much larger than that in the easy direction; (ii) we need to show that \emph{simultaneously} in all spatial locations for a positive fraction of the scales the aforementioned comparison holds. While we can partially address Issue (i) by RSW estimates, in order for a uniformly bound to address Issue (2) it seems inevitable to employ some concentration bound on the resistance, which is exactly what we wish to prove to begin with. This naturally calls for an induction, as hinted at the beginning. We emphasize that our induction for the resistance is intrinsically more involved than that for the LQG metric, which is manifested by the fact that \eqref{eq-LQG-Property} is an a priori input for the LQG metric (which holds in fact throughout the entire regime as proved in \cite{ding2021distance}) instead of being a consequence of the induction. This is also why our proof can only work for small $\gamma$, since we need this to obtain the induction basis for the first $\kappa$ scales (for a fixed but large integer $\kappa$).
	
	In what follows, we describe the content for each later subsection, which should further elaborate our proof outline.
	
	In \textbf{Section \ref{subsec62}}, we decompose $X_S$ into two terms by Lemma \ref{resdif}. To be precise, we let $A_S$ be an annulus surrounding $S$ and we write $Y_S=2\tilde{R}^{n}(\mbox{around }\mathcal{A}_S)/\tilde{R}^{n}(\mbox{across }\mathcal{A}_S)+1$ and $Z_S=\mathcal{E}(\theta,\mathcal{A}_S)/\tilde R^n_{1,1}$, where $\theta$ is the electric current corresponding to $\tilde R^n_{1,1}$. Then by Lemma \ref{resdif} we have $X_S\leq Z_SY_S$ (see Lemma \ref{energy}) and furthermore,
	\begin{equation}
		\sum_{S\in \mathcal{S}_\kappa}X_S^2\leq \max_{S\in \mathcal{S}_\kappa}Y_S^2 \max_{S\in S_\kappa}Z_S \sum_{S\in \mathcal{S}_\kappa}Z_S.
	\end{equation}
	Since $\sum_{S\in \mathcal{S}_\kappa}Z_S$ can be easily bounded by noticing that $\tilde R^n_{1,1}=\sum_{S\in\mathcal{S}_\kappa}\mathcal{E}(\theta,S)$, it remains to bound $\max_{S\in \mathcal{S}_\kappa}Y_S^2$ and $\max_{S_\in \mathcal{S}_\kappa}Z_S$ separately.
	
	In \textbf{Section \ref{subsec63}}, we bound $\max_{S\in \mathcal{S}_\kappa}Z_S$, which is the hardcore bit of this proof. We will apply Lemma \ref{resdif} to bound the change of the resistance after resampling a local field, and the key task is to bound the right-hand side of \eqref{resdifeq}, which culminates at a \emph{simultaneous} exponential decay for the energy in a local box as in Lemma \ref{decaylem}. To this end, the important ingredient is to show that with probability that is fixed but closed to 1, in each neighboring scale the energy decays by some constant factor. We emphasize that we need this probability to be close to 1 such that we can apply Hoeffding's inequality later to show the following: in every spatial location (i.e., a box of size $2^{-\kappa}$), with probability except exponentially small (where the rate of the exponential is large, in alignment of the close to 1 probability in each scale) we have at least a positive fraction of the scales with energy decay---the fact that the rate in the exponential probability is large then allows us to afford a union bound over all boxes in $\mathcal S_{\kappa}$. It is in order to prove energy decay with probability close to 1 that we need concentration bounds on effective resistances coming from the induction. If, on the contrary, we only need energy decay with positive probability, we can then just apply the RSW estimate and then conceptually simplify our proof (as well as remove the assumption on small $\gamma$).
	
	In order to prove the aforementioned energy decay, we will prove in Lemma \ref{decaynumber} that with probability except exponentially small (with a large rate) there is a positive fraction of scales for which the neighboring annuli have comparable resistances. When proving Lemma \ref{decaynumber}, we will need to prove a few technical estimates including some approximations which facilitate ``independence'', and we need to prove some bounds on oscillations of the field, as incorporated in Lemmas \ref{n2l}, \ref{n3l} and \ref{n4l}.
	
	In \textbf{Section \ref{subsec64}}, we bound $\max_{S\in \mathcal{S}_\kappa}Y_S^2$ in a similar way as in \cite[Section~3.5]{ding2020tightness1}. By path gluing techniques for percolation, we can obtain tail estimates for effective resistances around and across $\mathcal{A}_S$ in Lemma \ref{unilem}. Then we can bound the tail of $\max_{S\in S_\kappa}Y_S$ by a union bound, which then also implies a sufficient upper bound on the second moment of $\max_{S\in \mathcal{S}_\kappa} Y_S^2$.
	
	In \textbf{Section \ref{subsec65}}, we combine the results in previous subsections to complete the induction. We reiterate that here we will crucially use our assumption that $\gamma$ is sufficiently small to prove our induction basis. Having proved Proposition \ref{mainprop}, we further use similar arguments as in Propositions \ref{uptail} and \ref{lowtail} to obtain tail estimates for effective resistances.
	\subsection{Efron-Stein bound in terms of electric energy}\label{subsec62}
	In this subsection we bound the Efron-Stein difference by the Dirichlet energy restricted to a local region where the resampling occurs. Our intuition is that if the effective resistances around and across an annulus are within a constant factor, then we would be able to reroute the electric current efficiently and derive the desired energy bound.
	
	Let $m_{\kappa}\in\mathbb{N}$ be chosen so that
	\begin{equation}
		2^{m_\kappa}\leq \kappa^{\epsilon_0}< 2^{m_\kappa+1}.
	\end{equation}
	For $S\in\mathcal{S}_{\kappa}$, we write $v_S$ as the center of $S$. We define $\mathcal{A}_S$ to be the square annulus centered at $v_S$ and with inner radius $2^{-\kappa+m_\kappa}$ and outer radius $2^{-\kappa+m_\kappa+1}$, and define $\mathcal{D}_S$ to be the box centered at $v_S$ with side length $2^{-\kappa+m_\kappa+1}$ (i.e., $\mathcal D_S$ is the ``hole'' in $\mathcal A_S$).
	
	\begin{lem}\label{energy}
		Let $\theta$ be the electric current corresponding to $\tilde R^{n}_{1,1}$. We have that 
		\begin{equation}
			X_S\leq\frac{\mathcal{E}(\theta,\mathcal{A}_S)}{\tilde{R}^{n}_{1,1}}\left(2\frac{\tilde{R}^{n}(\mbox{around }\mathcal{A}_S)}{\tilde{R}^{n}(\mbox{across }\mathcal{A}_S)}+1\right).
		\end{equation}
	\end{lem}
	\begin{proof}
		Write $\tilde{R}^{n,\setminus S}_{1,1}$ as the analogue for $\tilde{R}^n_{1, 1}$ for the modified electric network where we set edge resistances in $\mathcal{D}_S$ to infinity. Applying Lemma \ref{resdif} with $\mathrm D=\mathcal{D}_S$, $\mathrm H=\mathcal{A}_S$ and $\mathcal P^\prime$ being the collection of contours separating the inner and outer boundaries of $\mathcal A_S$, we get that
		\begin{equation}\label{energyeq1}
			\log \tilde{R}^{n,S}_{1,1}
			\leq\log \tilde{R}^{n,\setminus S}_{1,1}
			\leq \log \left(\tilde{R}^{n}_{1,1}+\mathcal{E}(\theta,\mathcal{A}_S)+2\theta(\mathcal{D}_S)^2\tilde{R}^{n}(\mbox{around }\mathcal{A}_S)\right).
		\end{equation}
		Note that
		\begin{equation}
			\mathcal{E}(\theta,\mathcal{A}_S)\geq\theta(\mathcal{D}_S)^2\tilde{R}^{n}(\mbox{across }\mathcal{A}_S).
		\end{equation}
		Combined with \eqref{energyeq1}, it implies that (by some straightforward analysis)
		\begin{equation}
			\begin{aligned}
				\log \tilde{R}^{n,S}_{1,1} 
				&\leq
				\log \tilde{R}^{n}_{1,1}+\frac{1}{\tilde{R}^{n}_{1,1}}\left(\mathcal{E}(\theta,\mathcal{A}_S)+2\theta(\mathcal{D}_S)^2\tilde{R}^{n}(\mbox{around }\mathcal{A}_S)\right)\\
				&\leq
				\log \tilde{R}^{n}_{1,1}+\frac{\mathcal{E}(\theta,\mathcal{A}_S)}{\tilde{R}^{n}_{1,1}}\left(1+2\frac{\tilde{R}^{n}(\mbox{around } \mathcal{A}_S)}{\tilde{R}^{n}(\mbox{across }\mathcal{A}_S)}\right),
			\end{aligned}
		\end{equation}
		completing the proof of the lemma.
	\end{proof}
	
	\begin{lem}\label{sum}
		Let $\theta$ be the electric current corresponding to $\tilde{R}^{n}_{1,1}$. We have that
		\begin{equation}
			\sum_{S\in\mathcal{S}_\kappa}\mathcal{E}(\theta,\mathcal{A}_S)\leq \kappa\tilde{R}^{n}_{1,1}.
		\end{equation}
	\end{lem}
	\begin{proof}
		For each edge $e$ in the network, there are at most $2^{2m_{\kappa}+4}$ squares $S\in\mathcal{S}_\kappa$ such that $\mathcal{A}_S$ contains $e$. Thus we have
		\begin{equation}
			\sum_{S\in\mathcal{S}_\kappa}\mathcal{E}(\theta,\mathcal{A}_S)=\sum_{S\in\mathcal{S}_\kappa}\sum_{e\in S}\theta(e)^2r_e
			\leq \kappa\sum_{e\in B(1)\cap \mathbb Z_n^2}\theta(e)^2r_e \leq \kappa\tilde{R}^{n}_{1,1},
		\end{equation}
		where we used the fact that $2m_\kappa+4\leq2\log_2 \kappa^{\epsilon_0}+4\leq \log_2 \kappa$.
	\end{proof}
	
	We next derive a useful bound on the first term of the Efron-Stein bound as in \eqref{Efron}.
	
	\begin{lem}\label{intervar}
		Let $\theta$ be the electric current corresponding to $\tilde{R}^{n}_{1,1}$. We have
		\begin{equation}
			\sum_{S\in\mathcal{S}_\kappa}\mathbb{E}\left[X_S^2\right]\leq
			\kappa\left(\mathbb{E}\left[\max_{S\in\mathcal{S}_\kappa}\left(\frac{\mathcal{E}(\theta,\mathcal{A}_S)}{\tilde{R}^{n}_{1,1}}\right)^2\right]\right)^{\frac{1}{2}}
			\left(\mathbb{E}\left[\max_{S\in\mathcal{S}_\kappa}\left(2\frac{\tilde{R}^{n}(\mbox{around }\mathcal{A}_S)}{\tilde{R}^{n}(\mbox{across }\mathcal{A}_S)}+1\right)^4\right]\right)^{\frac{1}{2}}.
		\end{equation}
	\end{lem}
	\begin{proof}
		By Lemma \ref{energy}, we have
		\begin{equation}
			\begin{aligned}
				\sum_{S\in\mathcal{S}_\kappa}X_S^2
				&\leq
				\sum_{S\in\mathcal{S}_\kappa}\left(\frac{\mathcal{E}(\theta,\mathcal{A}_S)}{\tilde{R}^{n}_{1,1}}\right)^2\left(2\frac{\tilde{R}^{n}(\mbox{around }\mathcal{A}_S)}{\tilde{R}^{n}(\mbox{across }\mathcal{A}_S)}+1\right)^2\\
				&\leq
				\max_{S\in\mathcal{S}_\kappa}\frac{\mathcal{E}(\theta,\mathcal{A}_S)}{\tilde{R}^{n}_{1,1}}\max_{S\in\mathcal{S}_\kappa}\left(2\frac{\tilde{R}^{n}(\mbox{around }\mathcal{A}_S)}{\tilde{R}^{n}(\mbox{across }\mathcal{A}_S)}+1\right)^2\sum_{S\in\mathcal{S}_\kappa}\frac{\mathcal{E}(\theta,\mathcal{A}_S)}{\tilde{R}^{n}_{1,1}}\\
				&\leq
				\kappa\max_{S\in\mathcal{S}_\kappa}\frac{\mathcal{E}(\theta,\mathcal{A}_S)}{\tilde{R}^{n}_{1,1}}\max_{S\in\mathcal{S}_\kappa}\left(2\frac{\tilde{R}^{n}(\mbox{around }\mathcal{A}_S)}{\tilde{R}^{n}(\mbox{across }\mathcal{A}_S)}+1\right)^2,
			\end{aligned}
		\end{equation}
		where the last inequality follows from Lemma \ref{sum}. Taking expectation over both sides and applying Cauchy-Schwarz, we complete the proof of the lemma.
	\end{proof}
	
	\subsection{Exponential decay of electric energy}\label{subsec63}
	In this subsection, we prove the energy decay as incorporated in the next lemma.
	
	\begin{lem}\label{decaylem}
		Suppose that \eqref{induction} holds for $n-1$. Let $\theta$ be the electric current that corresponds to $\tilde{R}^{n}_{1,1}$. Then there exists an absolute constant $c_1$ such that
		\begin{equation}
			\mathbb{E}\left[\max_{S\in\mathcal{S}_\kappa}\left(\frac{\mathcal{E}(\theta,\mathcal{A}_S)}{\tilde{R}^{n}_{1,1}}\right)^2\right]\leq \exp\left(-c_1\kappa\right).
		\end{equation}
	\end{lem}
	
	As described earlier, the idea of proving Lemma \ref{decaylem} is to combine Lemma \ref{resdif} with resistance bounds. We next carry out the proof details. For $S\in\mathcal{S}_\kappa$, recall that $v_S$ is the center of $S$. For $\kappa/3\leq i\leq 2\kappa/3$, define $\mathcal{A}_S^i$ to be the square annulus centered at $v_S$ and with inner radius $2^{-i+7/3}$ and outer radius $2^{-i+8/3}$, and define $\mathcal{B}_S^i$ to be the square annulus centered at $v_S$ and with inner radius $2^{-i+2}$ and outer radius $2^{-i+3}$ (so in particular $\mathcal A^i_S \subset \mathcal B^i_S$).
	
	\begin{lem}\label{decaynumber}
		Suppose that \eqref{induction} holds for $n-1$. For any $\mathsf C,\mathsf c>0$, there exists a constant $C_1=C_1(\mathsf C,\mathsf c)$ such that $\mathbb{P}(N_1 \geq \mathsf c \kappa) \leq e^{-\mathsf C \kappa}$ where
		\begin{equation}
			N_1= |\{i\in [\kappa/3, 2\kappa/3] \cap \mathbb{Z}: \tilde{R}^{n}(\mbox{around } \mathcal{A}_S^i)\geq C_1\tilde{R}^{n}(\mbox{across } \mathcal{A}_S^i)\}|.
		\end{equation}
	\end{lem}
	The key to prove Lemma \ref{decaynumber} is to get approximate independence among different scales. While some general lemmas on Gaussian free fields are available along this line (see \cite[Lemma~3.1]{gwynne2020local}), the underlying field considered there is slightly different from ours. As a result, we write a proof in our specific setup for completeness. To this end, we decompose
	\begin{equation}
		\psi_{n}=\psi_{0,i}+\psi_{i,n}^{\mathrm{f}}+\psi_{i,n}^{\mathrm{c}},
	\end{equation}
	where (the superscripts $\mathrm{f}$ and $\mathrm{c}$ stand for ``fine'' and ``coarse'' respectively)
	\begin{equation}
		\psi_{i,n}^\mathrm{f}(x)=\int_{2^{-2n}}^{2^{-2\kappa}}\int_{\mathcal{B}_S^i}\tilde{p}_{\frac{t}{2}}(x,y)W(dy,dt) \And\psi_{i,n}^\mathrm{c}(x)=\int_{2^{-2n}}^{2^{-2\kappa}}\int_{\mathbb{R}^2\setminus\mathcal{B}_S^i}\tilde{p}_{\frac{t}{2}}(x,y)W(dy,dt).
	\end{equation}
	In the next three lemmas, we will control the oscillation of $\psi_{0, i}$, the maximum of $\psi^\mathrm{c}_{i, n}$ and the ratio between the resistances around and across $\mathcal{A}^i_S$ with respect to $\psi_{i, n}^\mathrm{f}$.
	
	\begin{lem}\label{n2l}
		For all $\mathsf C,\mathsf c>0$, there exists a constant $C_2=C_2(\mathsf C,\mathsf c)$ such that $\mathbb{P}(N_2 \geq \mathsf c \kappa) \leq e^{-\mathsf C \kappa}$ where
		\begin{equation}
			N_2=|\{i\in [\kappa/3,2\kappa/3]\cap \mathbb{Z}:\max_{u_i,v_i\in \mathcal{B}^i_{S}}(\psi_{0,i}(u_i)-\psi_{0,i}(v_i))\geq C_2\}|.
		\end{equation}
	\end{lem}
	\begin{proof}
		By Proposition \ref{distance} and Fernique's inequality \cite{fernique} (see also \cite[Theorem~4.1]{adler1990introduction}), we get that for each $i$,
		\begin{equation}
			\mathbb{E}\left[\max_{u_i,v_i\in \mathcal{B}^i_{S}}|\psi_{0,i}(u_i)-\psi_{0,i}(v_i)|\right]\leq C.
		\end{equation}
		Summing this over $\kappa/3 \leq i\leq 2\kappa/3$ yields that
		\begin{equation}\label{mean3}
			\mathbb{E}\left[\max_{u_i,v_i\in \mathcal{B}^i_{S}}\sum_{\kappa/3\leq i\leq 2\kappa/3} (\psi_{0,i}(u_i)-\psi_{0,i}(v_i))\right]\leq C \kappa.
		\end{equation}
		We next bound the variance for the sum in \eqref{mean3}. Note that for $j\geq i$, by the independence of $\psi_{0,i}$ and $\psi_{i,j}$ (recall that $\psi_{i, j} = \psi_{0,j} - \psi_{0, i}$),
		\begin{equation}
			\begin{aligned}
				\mathbb{E}(\psi_{0,i}(u_i)-\psi_{0,i}(v_i))(\psi_{0,j}(u_j)-\psi_{0,j}(v_j))
				=&\mathbb{E}(\psi_{0,i}(u_i)-\psi_{0,i}(v_i))(\psi_{0,i}(u_j)-\psi_{0,i}(v_j))\\
				\leq&\sqrt{\mathbb{E}(\psi_{0,i}(u_i)-\psi_{0,i}(v_i))^2\mathbb{E}(\psi_{0,i}(u_j)-\psi_{0,i}(v_j))^2}\\
				\leq&C\sqrt{\frac{|u_i-v_i|}{2^{-i}}\frac{|u_j-v_j|}{2^{-i}}}\leq C 2^{-\frac{(j-i)}{2}},
			\end{aligned}
		\end{equation}
		where the first inequality follows from Cauchy-Schwarz and the second inequality follows from Proposition \ref{distance}. Thus we have
		\begin{equation}\label{var3}
			\begin{aligned}
				&\mathrm{Var}\left(\sum_{\kappa/3\leq i\leq 2\kappa/3}(\psi_{0,i}(u_i)-\psi_{0,i}(v_i))\right)\\
				&\leq 2\sum_{\kappa/3\leq i\leq j\leq 2\kappa/3}\mathbb{E}(\psi_{0,i}(u_i)-\psi_{0,i}(v_i))(\psi_{0,j}(u_j)-\psi_{0,j}(v_j))\\
				&\leq C\sum_{\kappa/3\leq i\leq j\leq 2\kappa/3}2^{-\frac{(j-i)}{2}}
				\leq C \kappa.
			\end{aligned}
		\end{equation}
		By \eqref{mean3} and \eqref{var3}, we can apply the Borell-TIS inequality \cite{Borell1975-lg,Sudakov1978} (see, e.g., \cite[Theorem~2.1.1]{adler2009random} and \cite[Theorem~3.25]{van2014probability}) and derive that there exists a constant $C_2=C_2(\mathsf C,\mathsf c)$ such that
		\begin{equation}
			\mathbb{P}\left(\sum_{\kappa/3\leq i\leq 2\kappa/3}\max_{u_i,v_i\in \mathcal{B}^i_{S}}(\psi_{0,i}(u_i)-\psi_{0,i}(v_i))\geq C_2\mathsf c \kappa\right)\leq e^{-\mathsf C\kappa},
		\end{equation}
		completing the proof of the lemma.
	\end{proof}
	
	\begin{lem}\label{n3l}
		For any $\mathsf C,\mathsf c>0$, there exists a constant $C_3 = C_3(\mathsf C,\mathsf c)$ such that $\mathbb{P}(N_3\geq\mathsf c \kappa)\leq e^{-\mathsf C\kappa}$ where
		\begin{equation}
			N_3=|\{i\in[\kappa/3,2\kappa/3]\cap \mathbb{Z}:\max_{u_i\in\mathcal{A}_S^i}|\psi^\mathrm{c}_{i,n}(u_i)|\geq C_3\}|.
		\end{equation}
	\end{lem}
	\begin{proof}
		Similar to the proof of Lemma \ref{n2l}, we first bound the expectation and then enhance this by the concentration inequality. For $\kappa/3 \leq i\leq 2\kappa/3$, we fix $s_i\in \{-1,1\}. $ It is clear that we can write $\mathcal{A}^i_S$ as the union of four rectangles. For $x,x^{\prime}$ contained in the same rectangle, notice that for any $y\in \mathbb{R}^2\setminus \mathcal{B}^i_S$,
		\begin{equation}
			|\tilde{p}_{\frac{t}{2}}(x-y)-\tilde{p}_{\frac{t}{2}}(x^{\prime}-y)|=|\int_0^1\nabla_x\tilde{p}_{\frac{t}{2}}(x-y+s(x^{\prime}-x))\cdot (x^{\prime}-x)ds|\leq C\frac{|x-y|}{t}e^{-\frac{|x-y|^2}{2t}}|x-x^{\prime}|.
		\end{equation}
		Hence we have
		\begin{equation}
			\begin{aligned}
				\mathrm{Var}(\psi_{i,n}^\mathrm{c}(x)-\psi_{i,n}^\mathrm{c}(x^{\prime}))
				&=\pi\int^{2^{-2i}}_{2^{-2n}}\int_{\mathbb{R}^2\setminus\mathcal{B}^i_S}|\tilde{p}_{\frac{t}{2}}(x-y)-\tilde{p}_{\frac{t}{2}}(x^{\prime}-y)|^2dydt\leq C |x-x^{\prime}|.
			\end{aligned}
		\end{equation}
		Then by Fernique's inequality \cite{fernique} (see also \cite[Theorem~4.1]{adler1990introduction}) and summing over the four rectangles in $\mathcal{A}^i_S$, we have
		\begin{equation}\label{mean4.1}
			\mathbb{E}\left[\max_{u_i\in \mathcal{B}^i_{S}}s_i\psi_{i,n}^\mathrm{c}(u_i) \right]\leq C.
		\end{equation}
		Summing over all $\kappa/3 \leq i\leq 2\kappa/3$, we have
		\begin{equation}\label{mean4}
			\mathbb{E}\left[\max_{u_i\in \mathcal{B}^i_{S}}\sum_{\kappa/3\leq i\leq 2\kappa/3}s_i\psi_{i,n}^\mathrm{c}(u_i) \right]\leq C \kappa.
		\end{equation}
		We now need to bound the variance of the sum in \eqref{mean4}. To this end, we first get that
		\begin{equation}\label{var4.1}
				\mathrm{Var}(s_i\psi^\mathrm{c}_{i,n}(u_i))
				\leq\pi\int_{2^{-2n}}^{2^{-2i}}\int_{\mathbb{R}^2\setminus\mathcal{B}_S^i}[p_{\frac{t}{2}}(u_i-y)]^2dydt
				\leq C \int_{2^{-2n}}^{2^{-2i}}\frac{1}{2\pi t}e^{-\frac{2^{-2i}}{100t}}dt
				\leq C.
		\end{equation}
		For the covariance between different scales, we have for $j>i$,
		\begin{equation}\label{var4.2}
			\begin{aligned}
				\mathbb{E}(s_i\psi_{i,n}^\mathrm{c}(u_i)s_j\psi_{j,n}^\mathrm{c}(u_j))
				&\leq C\int_{2^{-2n}}^{2^{-2j}}\int_{\mathbb{R}^2}p_{\frac{t}{2}}(u_i-y)p_{\frac{t}{2}}(u_j-y)dydt\\
				&\leq  C\int_{2^{-2n}}^{2^{-2j}}\frac{1}{t}e^{-\frac{2^{-2i}}{100t}}dt=C\int_{2^{2j-2i}}^{2^{2n-2i}}\frac{1}{s}e^{-s/100}ds\leq C2^{-2(j-i)}.
			\end{aligned}
		\end{equation}
		Combining \eqref{var4.1} and \eqref{var4.2}, we have
		$$
			\mathrm{Var}\left(\sum_{\kappa/3\leq i\leq 2\kappa/3}s_i\psi_{i,n}^\mathrm{c}(u_i)\right)\leq \sum_{\kappa/3\leq i\leq 2\kappa/3}\mathrm{Var}(\psi_{i,n}^{\mathrm c}(u_i))+2\sum_{\kappa/3\leq i<j\leq 2\kappa/3}\mathbb{E}(\psi_{i,n}^{\mathrm c}(u_i)\psi_{j,n}^{\mathrm c}(u_j))\leq C \kappa.
		$$
		Combined with \eqref{mean4}, it yields by the Borell-TIS inequality \cite{Borell1975-lg,Sudakov1978} (see, e.g., \cite[Theorem~2.1.1]{adler2009random} and \cite[Theorem~3.25]{van2014probability}) that there exists a constant $C_3=C_3(\mathsf C,\mathsf c)$ such that
		\begin{equation}
			\mathbb{P}\left(\sum_{\kappa/3\leq i\leq 2\kappa/3}\max_{u_i\in \mathcal{B}^i_{S}}s_i\psi_{i,n}^\mathrm{c}(u_i)\geq C_3\mathsf c \kappa\right)\leq e^{-\mathsf C\kappa}.
		\end{equation}
		Taking a union bound over $s_i\in \{-1,1\}$ for $\kappa/2\leq i\leq 2\kappa/3$, we have
		\begin{equation}
			\begin{aligned}
				\mathbb P(N_3\geq \mathsf c \kappa)
				&\leq \mathbb{P}\left(\sum_{\kappa/3\leq i\leq 2\kappa/3}\max_{u_i\in \mathcal{B}^i_{S}}|\psi_{i,n}^\mathrm{c}(u_i)|\geq C_3\mathsf c \kappa\right)\\
				&\leq \sum_{s_i\in\{-1,1\}}\mathbb{P}\left(\sum_{\kappa/3\leq i\leq 2\kappa/3}\max_{u_i\in \mathcal{B}^i_{S}}s_i\psi_{i,n}^\mathrm{c}(u_i)\geq C_3\mathsf c \kappa\right)\leq 2^\kappa e^{-\mathsf C\kappa},
				\end{aligned}
		\end{equation}
		which completes the proof of the lemma (by properly adjusting $\mathsf C$).
	\end{proof}
	
	\begin{lem}\label{n4l}
		Suppose that \eqref{induction} holds for $n-1$. Write $Y_i$ and $Z_i$ as the effective resistances around and across $\mathcal{A}_S^i$ with respect to the field $\psi_{i,n}^{\mathrm{f}}$. Then for all $\mathsf C,\mathsf c>0$, there exists a constant $C_4=C_4(\mathsf C,\mathsf c)$, such that $\mathbb{P}(N_4\geq\mathsf c \kappa)\leq e^{-\mathsf C\kappa}$ where
		\begin{equation}
			N_4=|\{i\in [\kappa/3,2\kappa/3]\cap \mathbb{Z}:Y_i\geq C_4Z_i\}|.
		\end{equation}
	\end{lem}
	\begin{proof}
		By Corollary \ref{rswcor} and our assumption that \eqref{induction} holds for $n-i$ (since it holds for $n-1$), for every $q\in(0,1)$ there exists $C_5=C_5(q)$ such that
		\begin{equation}\label{n41}
			\mathbb{P}(\tilde{R}^{i,n}(\mbox{around }\mathcal{A}_S^i)\leq C_5\tilde{R}^{i,n}(\mbox{across }\mathcal{A}_S^i))\geq 1-q/2.
		\end{equation}
		Notice that by \eqref{mean4.1}, we also have
		\begin{equation}\label{n42}
			\mathbb{E}[\max_{u_i\in \mathcal{B}_S^i}\psi^{\mathrm c}_{i,n}(u_i)]\leq C.
		\end{equation}
		In addition, note that
		\begin{equation}\label{n43}
			Y_i\leq e^{\max_{u_i\in \mathcal{A}_S^i}\psi^{\mathrm c}_{i,n}(u_i)}\tilde{R}^{i,n}(\mbox{around }\mathcal{A}_S^i) \And Z_i\geq e^{-\max_{u_i\in \mathcal{A}_S^i}\psi^{\mathrm c}_{i,n}(u_i)}\tilde{R}^{i,n}(\mbox{across }\mathcal{A}_S^i).
		\end{equation}
		Combining \eqref{n41}, \eqref{n42} and \eqref{n43}, by Markov's inequality and Lemma \ref{comp}, there exists $C_4=C_4(C_5,q)$ such that
		\begin{equation}
			\mathbb{P}(Y_i\leq C_4Z_i)\geq 1-q.
		\end{equation}
		Since the events $\{Y_i\leq C_4 Z_i\}$ are independent, by Hoeffding's inequality and taking $q$ sufficiently close to $0$, it completes the proof of the lemma.
	\end{proof}
	\begin{proof}[Proof of Lemma \ref{decaynumber}]
		For  $\mathsf C,\mathsf c>0$, let $C_2, C_3, C_4$ be chosen as in Lemmas \ref{n2l}, \ref{n3l} and \ref{n4l}. Notice that on the event
		\begin{equation}
			\left\{\max_{u_i,v_i\in \mathcal{B}^i_{S}}(\psi_{0,i}(u_i)-\psi_{0,i}(v_i))\leq C_2\right\}
			\cap\left\{\max_{u_i\in\mathcal{A}_S^i}|\psi^{\mathrm c}_{i,n}(u_i)|\leq C_3\right\}
			\cap\left\{Y_i\leq C_4Z_i\right\},
		\end{equation}
		it holds that
		\begin{equation}
			\tilde{R}^{n}(\mbox{around }\mathcal{A}_S^i)\leq C_4e^{\gamma C_2+2\gamma C_3}\tilde{R}^{n}(\mbox{across }\mathcal{A}_S^i).
		\end{equation}
		So there exists $C_1 = C_1(C_2, C_3, C_4)$ such that
		\begin{equation}
			\mathbb{P}(N_1\geq 3\mathsf c \kappa)\leq \mathbb{P}(N_2\geq \mathsf c\kappa)+\mathbb{P}(N_3\geq \mathsf c\kappa)+\mathbb{P}(N_4\geq \mathsf c\kappa)\leq 3e^{-\mathsf C\kappa},
		\end{equation}
		which completes the proof of the lemma (by properly adjusting $\mathsf C$ and $\mathsf c$).
	\end{proof}
	
	\begin{proof}[Proof of Lemma \ref{decaylem}]
		Let $D_S^i$ be the box centered at $v_S$ with side length $2^{-i+3}$ (i.e., the boundary of $D_S^i$ is the same as the inner boundary of $\mathcal{B}_S^i$). Applying Lemma \ref{resdif} with $\mathrm H=\mathcal{A}^i_S$, $\mathrm D=D^i_S$ and $\mathcal P^\prime$ being the collection of contours separating the inner and outer boundaries of $\mathcal A^i_S$, we get that
		\begin{equation}
			0\leq\mathcal{E}(\theta, \mathcal{A}_S^i)+2\theta(D_S^i)^2\tilde{R}^{n}(\mbox{around }\mathcal{A}_S^i)-\mathcal{E}(\theta,D^i_S).
		\end{equation}
		Since $\theta(D^i_S)^2\tilde{R}^{n}(\mbox{across } \mathcal{A}_S^i)\leq \mathcal{E}(\theta, \mathcal{A}_S^i)$, we get on the event $\{\tilde{R}^{n}(\mbox{around }\mathcal{A}_S^i)\leq C_1\tilde{R}^{n}(\mbox{across }\mathcal{A}_S^i)\}$,
		\begin{equation}
			\mathcal{E}(\theta,D^i_S)\leq\mathcal{E}(\theta,\mathcal{A}_S^i)+2C_1\theta(D^i_S)^2\tilde{R}^{n}(\mbox{across }\mathcal{A}_S^i)\leq(2C_1+1)\mathcal{E}(\theta,\mathcal{A}_S^i).
		\end{equation}
		Noting that $\mathcal{E}(\theta,\mathcal{A}_S^i)\leq\mathcal{E}(\theta,D^{i-1}_S)-\mathcal{E}(\theta,D^i_S)$ since $D^{i}_S\cup \mathcal A^i_S\subset D^{i-1}_S$, we have
		\begin{equation}
			\mathcal{E}(\theta,D^i_S)\leq\frac{2C_1+1}{2C_1+2}\mathcal{E}(\theta,D^{i-1}_S).
		\end{equation}	
		Therefore, on the event that $\{N_1 < \mathsf c \kappa\}$ (recalling $N_1$ from Lemma \ref{decaynumber}), we have that for some $c_2=c_2(C_1)$,
		$$\frac{\mathcal{E}(\theta,\mathcal{A}_S)}{\tilde{R}^{n}_{1,1}}\leq \exp\left(-c_2\left(\frac{\kappa}{3}-\mathsf c \kappa\right)\right).$$
		Since in addition we have $\mathcal{E}(\theta,\mathcal{A}_S)\leq \tilde{R}^{n}_{1,1}$ for every realization, we apply the preceding inequality and Lemma \ref{decaynumber} for all $S\in \mathcal{S}_\kappa$, and get that
		$$\mathbb{E}\left[\max_{S\in\mathcal{S}_\kappa}\left(\frac{\mathcal{E}(\theta,\mathcal{A}_S)}{\tilde{R}^{n}_{1,1}}\right)^2\right]\leq\exp\left(-2c_2\left(\frac{\kappa}{3}-\mathsf c \kappa\right)\right)+2^{2\kappa}\exp\left(-\mathsf C\kappa\right),$$
		completing the proof of the lemma by choosing $\mathsf c < 1/3 $ and $\mathsf C>2$.
	\end{proof}
	
	\subsection{High moments for resistances}\label{subsec64}
	In this subsection we prove the following lemma.
	\begin{lem}\label{unilem1}
		Suppose that \eqref{induction} holds for $n-1$. Then for some absolute constant $C$ we have
		\begin{equation}
			\mathbb{E}\left[\max_{S\in\mathcal{S}_\kappa}\left(\frac{\tilde{R}^{n}(\mbox{around }\mathcal{A}_S)}{\tilde{R}^{n}(\mbox{across }\mathcal{A}_S)}\right)^4\right]\leq C\exp\left(3\kappa^{2/3}\right).
		\end{equation}
	\end{lem}
	In order to prove Lemma \ref{unilem1}, it suffices to prove the following tail estimates.
	\begin{lem}\label{unilem}
		Suppose that \eqref{induction} holds for $n-1$. Then there exist absolute constants $C,c$ such that for $T>2^{\kappa^{1/2+3\epsilon_0}}$
		\begin{equation}
			\mathbb{P}\left(\max_{S\in\mathcal{S}_\kappa}\frac{\tilde{R}^{n}(\mbox{around }\mathcal{A}_S)}{\tilde{R}^{n}(\mbox{across }\mathcal{A}_S)}>T\right)\leq C5^\kappa\exp\left(-c\frac{(\log T)^2}{\kappa^{3\epsilon_0}\log\log T}\right).
		\end{equation}
	\end{lem}
	\begin{proof}
		Let $\mathcal{A}$ be an annulus with inner radius $2^{-\kappa}$ and outer radius $2^{-\kappa+1}$. Since \eqref{induction} holds for $n-\kappa$, by Corollary \ref{rswcor} for some absolute constants $C,c$ we have
		\begin{equation}\label{uni3}
			\mathbb{P}\left(\tilde R^{\kappa,n}(\mbox{around } \mathcal{A})\geq T\right)\leq C\exp\left(-c\frac{(\log T)^2}{\log\log T}\right),
		\end{equation}
		\begin{equation}\label{uni4}
			\mathbb{P}\left(\tilde R^{\kappa,n}(\mbox{across } \mathcal{A})\leq T^{-1}\right)\leq C\exp\left(-c\frac{(\log T)^2}{\log\log T}\right).
		\end{equation}
		For each $\mathcal{A}_S$ where $S\in S_\kappa$, we can find at most $C2^{m_\kappa}$ copies of $\mathcal{A}$ such that the union of the paths around these copies must contain a contour around $\mathcal{A}_S$. Thus by Proposition \ref{series}, $R^{\kappa,n}(\mbox{around } \mathcal{A}_S)$ can be upper-bounded by the sum of effective resistances around these copies. Then by \eqref{uni3} and a union bound we have,
		\begin{equation}\label{uptail1}
			\mathbb{P}\left(\tilde R^{\kappa,n}(\mbox{around } \mathcal{A}_S)\right)\geq 2^{2m_\kappa}T)\leq C2^{m_\kappa}\exp\left(-c\frac{(\log T)^2}{\log\log T}\right).
		\end{equation}
		Similarly, for each $\mathcal{A}_S$ where $S\in S_\kappa$, we can find at most $C2^{m_\kappa}$ copies of $\mathcal{A}$ such that each path crossing $\mathcal{A}_S$ must cross one of these copies. By Proposition \ref{parallel} and \eqref{uni4}, we get that
		\begin{equation}\label{lowtail1}
			\mathbb{P}\left(\tilde R^{\kappa,n}(\mbox{across } \mathcal{A}_S)\leq 2^{-2m_\kappa}T^{-1}\right)\leq C2^{m_\kappa}\exp\left(-c\frac{(\log T)^2}{\log\log T}\right).
		\end{equation}
		By \eqref{uptail1}, \eqref{lowtail1} and a union bound, for some absolute constants $C,c$ we have for all $T>\exp(\kappa^{1/2+\epsilon_0})$ (we assume this in order to absorb the term $2^{2m_\kappa}$) that
		\begin{equation}\label{uni1}
			\mathbb{P}\left(\max_{S\in\mathcal{S}_\kappa}\frac{\tilde{R}^{\kappa,n}(\mbox{around }\mathcal{A}_S)}{\tilde{R}^{\kappa,n}(\mbox{across }\mathcal{A}_S)}>T\right)\leq C5^\kappa\exp\left(-c\frac{(\log T)^2}{\log\log T}\right).
		\end{equation}
		In addition, we have
		\begin{equation}\label{uni2}
			\begin{aligned}
				&\mathbb{P}\left(\max_{u,v\in \mathcal{A}_S,S\in \mathcal{S}_\kappa}(\psi_{0,\kappa}(u)-\psi_{0,\kappa}(v))>x\right)\\
				&\leq
				\mathbb{P}\left(\|\phi_\kappa-\psi_\kappa\|_{B(1)}>x/4\right)+\mathbb P\left(2^{-\kappa+m_\kappa}\|\nabla \phi_{\kappa}\|_{B(1)}\geq x/4\right)\\			
				&\leq Ce^{-cx^2}+C4^\kappa \exp\left(-c\frac{x^2}{2^{2m_\kappa}}\right),
			\end{aligned}
		\end{equation}
		where the first inequality follows from a simple union bound, and the second inequality follows from Propositions \ref{gradient} and \ref{normprop}. Combining \eqref{uni1} and \eqref{uni2} we derive that for $T>2^{\kappa^{1/2+3\epsilon_0}}$ and $\kappa$ sufficiently large, and for some absolute constants $C,c$ we have
		\begin{equation}
			\begin{aligned}
				&\mathbb{P}\left(\max_{S\in\mathcal{S}_\kappa}\frac{\tilde{R}^{n}(\mbox{around }\mathcal{A}_S)}{\tilde{R}^{n}(\mbox{across }\mathcal{A}_S)}>T\right)\\
				&\leq
				\mathbb{P}\left(\max_{S\in\mathcal{S}_\kappa}\frac{\tilde{R}^{\kappa,n}(\mbox{around }\mathcal{A}_S)}{\tilde{R}^{\kappa,n}(\mbox{across }\mathcal{A}_S)}>\sqrt{T}\right)
				+\mathbb{P}\left(\max_{u,v\in \mathcal{A}_S,S\in \mathcal{S}_{\kappa}}(\psi_{0,\kappa}(u)-\psi_{0,\kappa}(v))>\frac{1}{2}\log T\right)\\
				&\leq C5^\kappa\exp\left(-c\frac{(\log \sqrt{T})^2}{\log\log \sqrt{T}}\right)+C4^\kappa\exp\left(-c\frac{(\log T)^2}{2^{2m_\kappa}}\right)\\
				&\leq C5^\kappa\exp\left(-c\frac{(\log T)^2}{\kappa^{3\epsilon_0}\log\log T}\right),
			\end{aligned}
		\end{equation}
		where in the first inequality we used $\gamma\leq\gamma_0<1$. Thus we complete the proof.
	\end{proof}

	\begin{proof}[Proof of Lemma \ref{unilem1}]
		By Lemma \ref{unilem} and integrating the tail probability, we get that
		\begin{equation}
			\begin{aligned}
				\mathbb{E}\left[\max_{S\in\mathcal{S}_\kappa}\left(\frac{\tilde{R}^{n}(\mbox{around }\mathcal{A}_S)}{\tilde{R}^{n}(\mbox{across }\mathcal{A}_S)}\right)^4\right]
				&=\int_{0}^{\infty}4T^3\mathbb{P}\left(\max_{S\in\mathcal{S}_\kappa}\frac{\tilde{R}^{n}(\mbox{around }\mathcal{A}_S)}{\tilde{R}^{n}(\mbox{across }\mathcal{A}_S)}>T\right)dT\\
				&\preceq \exp\left(3\kappa^{2/3}\right)+5^\kappa\int_{\exp\left(\kappa^{2/3}\right)}^{\infty}4T^3\exp\left(-c_3\frac{(\log T)^2}{\kappa^{3\epsilon_0}\log\log T}\right)dT\\
				&\preceq \exp\left(3\kappa^{2/3}\right)+5^\kappa\exp\left(-c\frac{\kappa^{4/3-3\epsilon_0}}{\log\log \kappa}\right)\\
				&\preceq \exp\left(3\kappa^{2/3}\right),
			\end{aligned}
		\end{equation}
		where the implicit constant in $\preceq$ are absolute constants.
	\end{proof}
	
	\subsection{Conclusion}\label{subsec65}
	In this subsection we complete the proof of our induction. We first relate the quantile and the variance of a random variable.
	\begin{lem}\label{quanvar}
		Let $X$ be a positive and continuous random variable, and define $l(p)=\inf\{x:\mathbb{P}(X\geq x)\leq p\}.$ Then for $0<p<1/2$ we have
		\begin{equation}
			\frac{l(1-p)}{l(p)}\leq \exp\left(\frac{\sqrt{2}}{p}\sqrt{\mathrm{Var}\left[\log X\right]}\right).
		\end{equation}
	\end{lem}
	\begin{proof}
		If $X^*$ is an independent copy of $X$, then we have
		\begin{equation}
			\begin{aligned}
				2\mathrm{Var}\left[\log X\right]
				&=\mathbb{E}\left[(\log X^*-\log X)^2\right]
				\geq \mathbb{E}\left[(\log X^*-\log X)^2\mathds{1}_{\{X^*\geq l(1-p)\}}\mathds{1}_{\{X\leq l(p)\}}\right]\\
				&\geq p^2\left(\log\frac{l(1-p)}{l(p)}\right)^2,
			\end{aligned}
		\end{equation}
		which completes the proof.
	\end{proof}

	\begin{proof}[Proof of Proposition \ref{mainprop}]
		Our proof is by induction. We first verify the induction basis, that is, we show that \eqref{induction} holds for all $n \leq \kappa$ by choosing $\gamma_0$ sufficiently small depending on $\kappa$.	
		Note that $\mathrm{Var}(\psi_{n}(x))\leq n\log 2$, and that $\log \tilde{R}^{n}_{1,1}$ is a $\gamma$-Lipschitz function with respect to the $L^{\infty}$-norm. By the Gaussian concentration inequality (see \cite{Borell1975,Sudakov1978} and see also e.g. \cite[Theorem~3.25]{van2014probability}) we have
		\begin{equation}
			\mathrm{Var}(\log \tilde{R}^{n}_{1,1})\preceq \gamma^2n.
		\end{equation}
		Thus by Lemma \ref{quanvar} there exists an absolute constant $C$ such that for $n \leq \kappa$
		\begin{equation}\label{quaneq}
			\tilde \Lambda_n(\psi,p)\leq \exp\left(C\gamma\sqrt{\kappa}\right).
		\end{equation}
		Setting $\gamma_0$ sufficiently small such that the right-hand side of \eqref{quaneq} is bounded by 2 for all $\gamma < \gamma_0$ completes the verification of the induction basis.
		
		We now continue with the inductive step. Supposing that \eqref{induction} holds for $n-1$, we will show that it also holds for $n$. Recall \eqref{Efron} and \eqref{Efron1}. By Lemmas \ref{intervar}, \ref{decaylem} and \ref{unilem1}, there exists an absolute constant $C,c$, such that
		\begin{equation}
			\mathrm{Var}(\log \tilde{R}^{n}_{1,1})\leq C\kappa\exp\left(3\kappa^{2/3}-c\kappa\right)+C\gamma^2 \kappa.
		\end{equation}
		Therefore, by Lemma \ref{quanvar}, we have that
		$$\tilde \Lambda_n(\psi, p) \leq \exp\left(\frac{\sqrt{2}}{p}\sqrt{C\kappa\exp\left(3\kappa^{2/3}-c\kappa\right)+C\gamma^2\kappa}\right).$$
		Choosing $\kappa$ fixed and sufficiently large and then possibly decreasing the value of $\gamma_0$ (both depending $p$), we can then get that the right-hand side above is bounded by 2, completing the inductive step. In addition, noting that our choice of $\gamma_0$ eventually only depends on $p$. Thus, recalling Proposition \ref{quancom} we have $$\sup_{n}\Lambda_n(\phi,2p)<\infty.$$
		Then, we my complete the proof of the proposition by Propositions \ref{uptail} and \ref{lowtail}.
	\end{proof}
	
	\begin{prop}\label{tail}
		For any $q>0$, there exist constants $C_6=C_6(q)$ and $c_1=c_1(q)$, such that for all $T>3$ and $n,a\geq 1$,
		\begin{equation}\label{tail9}
			\mathbb{P}\left(R^{n}_{a,qa}>T\right)\leq C_6\exp\left(-c_1\frac{(\log T)^2}{\log\log T}\right),
		\end{equation}
		\begin{equation}\label{tail10}
			\mathbb{P}\left(R^{n}_{a,qa}<T^{-1}\right)\leq C_6\exp\left(-c_1\frac{(\log T)^2}{\log\log T}\right).
		\end{equation}
	\end{prop}
	
	\begin{proof}
		It suffices to prove for $q=8$, $a=2^k$ and for all $s>2$ that
		\begin{equation}\label{con3}
			\mathbb{P}\left(R^{n}_{8\cdot2^k,2^k}\geq Ce^{Cs\sqrt{\log s}}\right)\leq e^{-cs^2}
		\end{equation}
		holds for some absolute constant $C,c$. For general $(a,q)$, \eqref{tail9} and \eqref{tail10} follow from Proposition \ref{series} and the duality as in \eqref{dual1}.
		
		We now prove \eqref{con3}. We divide our proof into three cases. 
		
		\noindent {\bf Case 1: $s\in (2,2^{k/2})$.} \eqref{con3} follows immediately from Lemma \ref{taillem1} and Proposition \ref{mainprop}.
		
		\noindent {\bf Case 2: $s\in(2^{k/2},2^{(n+k)/2})$.}
		We consider the rescaled lattice $\mathbb{Z}_{n-m}^2$ for some $m<n$. By Lemma \ref{taillem1} and Proposition \ref{mainprop},
		\begin{equation}\label{tail7}
			\mathbb{P}\left(R^{n-m}_{8\cdot2^{m+k},2^{m+k},\zeta_{n-m}}\geq C2^{2(m+k)}\right)\leq Ce^{-c2^{m+k}}.
		\end{equation}
		By Proposition \ref{mesh}, since $\zeta_n, \zeta_{n-m}\geq \sqrt{n-m}$, we have
		\begin{equation}
			\mathbb{P}\left(R^{n-m}_{8\cdot2^{m+k},2^{m+k},\zeta_n}\geq e^{C\sqrt{2^{m+k}}}\cdot C2^{2(m+k)}\right)
			\leq C2^{2(m+k)}e^{-c2^{m+k}}.
		\end{equation}
		By the scaling property of $\phi$, we know that $R^{n-m}_{8\cdot2^{m+k},2^{m+k},\zeta_n}\overset{d}{=}R^{m,n}_{8\cdot 2^k,2^k,\zeta_n}$. Thus, by possibly adjusting the values of $c, C$, we have
		\begin{equation}\label{tail8}
			\mathbb{P}\left(R^{m,n}_{8\cdot2^{k},2^{k}}\geq e^{C\sqrt{2^{m+k}}}\right)
			\leq Ce^{-c2^{m+k}}.
		\end{equation}
		Combining \eqref{tail8} and Proposition \ref{max}, we have
		\begin{equation}
			\begin{aligned}
				&\mathbb{P}\left(R^{n}_{8\cdot 2^k,2^k}\geq e^{s\sqrt{m}}e^{2C\sqrt{2^{m+k}}}\right)\\
				&\leq
				\mathbb{P}\left(\|\phi_{0,m}\|_{B(8\cdot 2^k,2^k)}\geq Cm+s\sqrt{m}\right)+\mathbb{P}\left(R^{m,n}_{8\cdot2^{k},2^{k}}\geq e^{C\sqrt{2^{m+k}}}\right)
				\leq C2^{2k}e^{-cs^2}+Ce^{-c2^{m+k}},
			\end{aligned}
		\end{equation}
		where in the first inequality we used the fact that $C \sqrt{2^{m+k}} \geq Cm$. Taking $m=\max\{i:2^i\leq s^2\}-k$, it holds that (after adjusting $c$ to absorb the term $2^{2k}$ and adjusting $C$ to absorb the term $e^{2Cs}$)
		\begin{equation}
			\mathbb{P}\left(R^{n}_{8\cdot 2^k,2^k}\geq Ce^{Cs\sqrt{\log s}}\right)\leq e^{-cs^2}.
		\end{equation}
	
		\noindent {\bf Case 3: $s\geq2^{(n+k)/2}$.} By Proposition $\ref{max}$ and a union bound, we have (recall $\gamma \leq 1$)
		$$
			\mathbb{P}\left(R^{n}_{8\cdot 2^k,2^k}\geq 8e^s\right)\leq 2^{2k+4}\mathbb{P}\left(\max_{B(1)}|\phi_{n}|\geq s\right)\leq C4^{n+k}\exp\left(-\frac{s^2n}{(n+C\sqrt{n})^2}\right)\leq C\exp\left(-c\frac{s^2}{\log s}\right),
		$$
		completing the proof of the proposition.
	\end{proof}
	
	\begin{proof}[Proof of Theorem \ref{thm1}]
		Fix $q>0$. For any $a>1$ and $n\geq 1$, and any $\eta>0$, by Proposition \ref{tail} there exists $C_{\eta,q}=C_{\eta,q}(\eta,q)$ such that $\mathbb{P}(C_{\eta,q}^{-1}\leq R^{n}_{a,qa,\zeta_n}\leq C_{\eta,q})\geq 1-\eta$. Thus, the family of distributions of $\{R^{n}_{a,qa,\zeta_n}\}_{n\in \mathbb{Z}^+,a\geq 1}$ is tight for a fixed sequence $\{\zeta_n\}_{n\geq 1}$ with $\zeta_n\geq \sqrt{n}$. We may then complete the proof by Proposition \ref{mesh}.
	\end{proof}

	\section{Tightness for the trace of the random walk}\label{sec7}
	Having established Theorem \ref{thm1}, we next turn to proving tightness for the random walk in the remaining three sections. In this section, we prove tightness for the trace of the random walk.
	
	In \textbf{Section \ref{subsec71}}, we relate exiting distributions of random walks to effective resistances. While it is well-known that exiting distributions can be characterized by resistances, the main novelty of our estimate, as incorporated in Lemma \ref{returnlem}, is that it allows to derive exiting distribution estimates with up-to-constant estimates for resistances. On the one hand, this is a desirable feature since we can only hope to control resistances up to constant factors; on the other hand, we cannot expect this to be true in full generality in light of examples in \cite{Benjamini1991,ding2013sensitivity}. As a result, our lemma is specialized to $\mathbb{Z}^2$ where planarity plays an important role in the proof.
	
	In \textbf{Section \ref{subsec72}}, we consider the exiting measure when the random walk exits from a box for the first time and we prove that the exiting measure is continuous with respect to the starting point. To this end, we will apply a lemma from \cite{gwynne2022invariance}, and we will need to verify its input that with high probability the random walk trace (before exiting the box) disconnects the boundary of the box from (a small box around) the starting point.
	
	In \textbf{Section \ref{subsec73}}, we prove Theorem \ref{thm2}, which implies the existence of sub-sequential limits for the random walk traces.
	
	The lemmas proved in Sections \ref{subsec72} and \ref{subsec73} will also play an important role in Section \ref{sec9}.
	
	\subsection{Exiting distributions via resistances}\label{subsec71}
	In this subsection, we consider deterministic networks except for Lemma \ref{exitlem2}. For notation consistency, we continue to use $P$ and $E$ to denote the probability and the expectation for the random walk.
	\begin{lem}\label{returnlem4}
		Consider a network with vertex set $\{1, 2, 3\}$. For $1\leq i<j\leq 3$, let $c_{i, j}$ be the conductance on the edge $(i, j)$ and let $R_{i,j}$ be the effective resistance between $i$ and $j$. Then,
		\begin{equation}\label{return1}
			\frac{c_{12}}{c_{12}+c_{13}}=\frac{R_{13}+R_{23}-R_{12}}{2R_{23}}.
		\end{equation}
	\end{lem}
	\begin{proof}
		For $i\neq j$, let $k$ be the vertex different from $i$ and $j$. Applying the series law and the parallel law, we get that
		\begin{equation}\label{return2}
			R_{ij}=\frac{1}{c_{ij}+(c_{ik}^{-1}+c_{jk}^{-1})^{-1}}.
		\end{equation}
		Plugging \eqref{return2} into the right-hand side of \eqref{return1}, we complete the proof.
	\end{proof}
	By Lemma \ref{returnlem4} and the network reduction principle (see, e.g., \cite[Page~18]{ding2011cover}), we see that in any network containing vertex $v$ and vertex sets $A$ and $Z$ where $v\not\in A\cup Z$, we have that
	\begin{equation}\label{return6}
		P^v(\tau_A < \tau_Z) = \frac{R(v,Z)+R(A,Z)-R(v,A)}{2R(A,Z)}.
	\end{equation}
	
	Next we prove several lemmas for networks with $\mathbb Z^2$ as the underlying graph. We view $\mathbb{Z}^2$ as a network where each edge is associated with an edge conductance, and we consider random walks on the \emph{network} $\mathbb{Z}^2$. Since this subsection is all about $\mathbb Z^2$, we slightly abuse the notation by denoting $B(n)=[-n,n]^2\cap\mathbb{Z}^2$.
	\begin{lem}\label{returnlem}
		Let $0\leq r_1<r_2<r_3<r_4$ be four integers. Denote $\mathcal{A}=B(r_3)\setminus B(r_2)$, $\mathcal{B}=B(r_4)\setminus B(r_3)$, $A=\partial B(r_1)$ and $Z=\partial B(r_4)$. If for some constants $\mathsf C, \mathsf c>0$ we have
		\begin{equation}\label{returncon}
			\mathsf c R(A,Z)\leq R(\mbox{around }\mathcal{A})\leq \mathsf C R(\mbox{across }\mathcal{B}),
		\end{equation}
		then there exists $q=q(\mathsf c,\mathsf C)\in (0,1)$ such that for any $v\in B(r_2) \setminus B(r_1)$,
		\begin{equation}
			P^v(\tau_A<\tau_Z)\geq q.
		\end{equation}
	\end{lem}

	In order to prove Lemma \ref{returnlem}, the key ingredient is Lemma \ref{returnlem3}, for which we need the following lemmas for preparation. The proof of Lemma \ref{plainlem} will use Lemma \ref{plainlem1}.

	In what follows, by a planar graph we mean a finite graph that is embedded into $\mathbb R^2$ such that edges do not cross except at vertices, and we define the outer boundary of the planar graph to be the boundary of the face that contains infinity.
	\begin{lem}\label{plainlem}
		Let $\mathrm{G}$ be a network associated with a planar graph. Let $Z$ be the outer boundary of $\mathrm{G}$ and let $A\notin Z$ be an inner point. Consider the electric current $\theta$ from $A$ to $Z$ in the network $\mathrm{G}$. Then there exist $\alpha_1,\cdots,\alpha_k\in \mathbb{R}^+$ and paths $P_1,\cdots,P_k$ from $A$ to $Z$ such that, $P_i$'s are ranked in the clockwise (or counterclockwise) order and they do not cross each other (but may have common edges), and $\theta$ can be decomposed as the sum of flows along $P_i$ with strength $\alpha_i$.
	\end{lem}

	\begin{lem}\label{plainlem1}
		Let $\mathrm{G}$ be a network associated with a planar graph. Let $Z$ be the outer boundary of $\mathrm{G}$ and let $A\notin Z$ be an inner point. Consider the unit electric current $\theta$ from $A$ to $Z$ in the network $\mathrm{G}$, and let $f$ be the voltage with respect to $\theta$ (where $f(Z)$ is set to 0). For each vertex $x\notin A\cup Z$, the neighbors of $x$ can be arranged in the clockwise order as $y_1,\cdots y_n$, such that for some $m\leq n$ we have $f(y_i)\geq f(x)$ for $1\leq i\leq m$ and $f(y_i)\leq f(x)$ for $m+1\leq i\leq n$.
	\end{lem}
	\begin{figure}			
		\centering			
		\includegraphics[width=80mm]{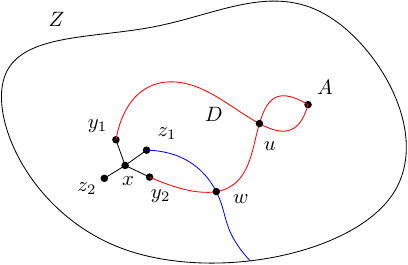}
		\caption{Illustration for the proof of Lemma \ref{plainlem1}}
		\label{fig7}
	\end{figure}
	\begin{proof}
		We prove by contradiction. Suppose that for some $x\notin A\cup Z$ the claim does not hold. Then we consider all neighbors of $x$ with voltages different from $x$. For those with voltages higher than $x$, they are not consecutive (due to our assumption that the claim does not hold for $x$), so there exist $y_1,z_1,y_2,z_2$ arranged in the clockwise order such that $f(y_i)> f(x)$ and $f(z_i)<f(x)$ for $i\in \{1,2\}$. Then for each $y_i$, there exists a self-avoiding path from $y_i$ to $A$, denoted by $\{v^i_k\}_{k\leq n_i}$, such that $v^i_0=y_i$, $v^i_{n_i}=A$ and $f(v^i_k)> f(v^i_{k-1})$ for all $1\leq k\leq n_i$ (the path exists since except at $A$ and $Z$ the total incoming flow to a vertex is equal to the total outgoing flow). Let $u$ be the first point (i.e., with minimal voltage) where these two paths intersect. Consider the domain surrounded by sub-paths $y_1u$, $y_2u$ and edges $y_1x$ and $y_2x$, and we denote this domain by $D$. Without loss of generality we assume that $z_1$ lies in this sub-domain. Since $f(x)>f(z_1)$, there exists a self-avoiding path from $z_i$ to $Z$, denoted by $\{u_k\}_{k\leq n_3}$, such that $u_0=z_1$, $u_{n_3}\in Z$ and $f(u_k)< f(u_{k-1})$ for $1\leq k\leq n_3$. This path must intersect with the boundary of $D$ at a vertex $w$ (see Figure \ref{fig7} for an illustration). Thus for some $i\in \{1,2\}$ (in particular for $i = 2$ as illustrated in Figure \ref{fig7}) we have $f(w)\leq f(z_1)<f(x)<f(y_i)\leq f(w)$, which leads to a contradiction.
	\end{proof}

	\begin{proof}[Proof of Lemma \ref{plainlem}]
		We identify $A$, and respectively $Z$, as a single point. Let $f$ be the voltage function corresponding to $\theta$ with $f(Z) = 0$. We list all vertices in $V(\mathrm{G})$ as $\{v_k\}_{0\leq k\leq n}$ such that $v_0=A$, $v_n=Z$ and $f(v_{k-1}) \geq f(v_k)$ for all $1\leq k\leq n$. We inductively define $\mathcal{P}_k$ to be sets of weighted paths satisfying the following claims: 
		\begin{enumerate}
			\item\label{claim1} the sum of weights equals to 1;
			\item\label{claim2} the paths in $\mathcal{P}_k$ do not cross with each other;
			\item\label{claim3} the voltages of the vertices along each path in $\mathcal{P}_k$ are decreasing, and the end of each path is in $\{v_k,\cdots,v_n\}$;
			\item\label{claim4} the sum of weights of paths ending in $v_k$ equals to $\sum_{v:f(v)>f(v_k)}\theta(v,v_k)$.
		\end{enumerate}
		
		For $k=1$, let $y_1^1,\cdots,y_{n_1}^1$ be all neighbors of $v_0$ with voltages less than $f(v_0)$, arranged in the clockwise order, and we set $\mathcal P_1=\{P_1^1,\cdots, P_{n_1}^1\}$, where $P_i^1$ consists of a single edge $v_0y_i$ with weight $\theta(v_0,y_i)$. It is straightforward to check that $\mathcal P_1$ satisfies all conditions above.
		
		Now for $k\geq 1$ suppose that we have defined $\mathcal{P}_k$, and we consider the vertex $v_k$. By Lemma \ref{plainlem1}, the neighbors of $v_k$ can be arranged in the clockwise order as $y^k_1,\cdots,y^k_{n_k}$, such that $f(y^k_i)\geq f(v_{k})$ for $1\leq i\leq m_k$ and $f(y^k_i)\leq f(v_k)$ for $m_k+1\leq i\leq n_k$. Note that a path with weight $w$ is equivalent to several copies of this path with a total weight of $w$. For the paths in $\mathcal{P}_{k}$ ending at $v_{k}$, since by induction hypothesis the sum of their weights equals to $$\sum_{1\leq i\leq m_k}\theta(y_i^k,v_k)=\sum_{m_k+1\leq i\leq n_k}\theta(v_k,y_i^k),$$ they can be decomposed to paths $\cup_{m_k+1\leq i\leq n_k}\mathcal{Q}_i^k$ that do not cross each other, with weights $\{w_Q\}$ such that $\sum_{Q\in \mathcal Q_i^k} w_Q=\theta (v_k,y_i^k)$. We define
		$$\mathcal{P}_{k+1}=\{P:P\in \mathcal{P}_k \mbox{ that does not end at } v_k\}\cup \cup_{i=m_k+1}^{n_k}\{Q \oplus (v_k, y_i^k): Q\in \mathcal Q_k^i\},$$
		where $\oplus$ above means the concatenation of paths. It can be checked directly that $\mathcal{P}_{k+1}$ satisfies Claims \ref{claim1}, \ref{claim2} and \ref{claim3}. For Claim \ref{claim4}, we see that the paths in $\mathcal P_{k+1}$ ending at $v_{k+1}$ can be decomposed as a union of $\mathbf P_{i, k+1}$ over $1\leq i\leq k$. Here $\mathbf{P}_{i, k+1}$ is a collection of paths of the form $P_{i, k+1}^l \oplus (v_{i}, v_{k+1})$, where $P_{i, k+1}^l$ is a path in $\mathcal P_{i}$ ending at $v_{i}$ and the sum of the weights for $P_{i, k+1}^l$ over $l$ is $\theta(v_{i}, v_{k+1})$ (and naturally we think $\mathbf P_{i, k+1} = \emptyset$ if $\theta(v_i, v_{k+1}) = 0$). Thus $\mathcal{P}_{k+1}$ satisfies Claim \ref{claim4}, completing the proof of the induction. Finally, arranging the paths in $\mathcal{P}_n$ in the clockwise (or counterclockwise) order, we complete the proof.
	\end{proof}
	
	\begin{lem}\label{returnlem3}
		Let $0\leq r_1<r_2<r_3<r_4$ be four integers. Denote $\mathcal{A}=B(r_3)\setminus B(r_2)$, $\mathcal{B}=B(r_4)\setminus B(r_3)$, $A=\partial B(r_1)$ and $Z=\partial B(r_4)$. If $R(\mbox{around }\mathcal{A})\leq \mathsf C R(\mbox{across } \mathcal{B})$ for some $\mathsf C>0$, then there exists $r=r(\mathsf C)$ such that
		\begin{equation}
			R(v,Z)+R(A,Z)-R(v,A)\geq rR(\mbox{around }\mathcal{A}).
		\end{equation}
	\end{lem}

	In order to prove Lemma \ref{returnlem3}, we first consider some approximation and reduction to simplify the setup. We treat $A$ and $Z$ as two nodes (where we identify all points in $Z$, and respectively in $A$, as a single point as before). Let $\theta_v$ and $\theta_A$ denote the unit electric current from $v$ to $Z$ and from $A$ to $Z$, respectively. Also let $\rho$ be the flow with strength $1$ that corresponds to $R(\mbox{around } \mathcal{A})$. We apply Lemma \ref{plainlem} to $\theta_v$ and $\theta_A$ and obtain the following: $\theta_v$ can be decomposed as the sum of flows along $P_k$ with strength $\alpha_k$ (for $1\leq k\leq m$), where $P_1, \ldots, P_m$ (for some $m\geq 1$) are paths from $v$ to $Z$ arranged in the counterclockwise order; $\theta_A$ can be decomposed as the sum of flows along $Q_l$ with strength $\beta_l$ (for $1\leq k\leq n$), where $Q_1, \ldots, Q_n$  (for some $n\geq 1$) are paths from $A$ to $Z$ arranged in the clockwise order. In addition, $\rho$ can be decomposed as the sum of flows along $C_t$ with strength $\eta_t$ (for $1\leq t\leq d$), where $C_1,\ldots ,C_d$ are cycles around $\mathcal{A}$ (for some $d\geq 1$). By considering an equivalent network we may assume that $C_t$'s are edge-disjoint.
	
	We fix an integer $N>0$ and $a>0$. We approximate each $P_k$, $Q_l$ and $C_t$ by a flow with integral multiple of $a$ as its strength, so that we obtain $\bar \theta_v, \bar \theta_A$ and $\bar \rho$ with strength greater than $1-\beta a$, where $\beta =\beta (r_4)$. Since a flow with strength $w$ can be equivalently regarded as flows on several copies of this path with a total strength of $w$, we may assume that $\alpha_k=\beta_l=N!a$, $\eta_t=Na$ for all $1\leq k\leq m$, $1\leq l\leq n$ and $1\leq t\leq d$. For each $1\leq k\leq m$, we define $P_k^{1},\ldots ,P_k^{N!}$ to be $N!$ copies of $P_k$, where each $P_k^i$ is associated with a flow of strength $\alpha_k^i=a$, and further we define
	$r_{e,P_k,i}=\mathds{1}_{\{e\in P_k^i\}}\frac{N!\bar \theta_v(e)}{\alpha_k}r_e+\mathds{1}_{\{e\notin P_k^i\}}\cdot \infty.$
	
	\begin{lem}\label{returnlem7}
		We continue to use notations as described in the preceding two paragraphs. For an integer $M$ with $0<M\leq N$, there exists a path set $\mathrm{TrP}\subset \{P^i_k: 1\leq k\leq m, 1\leq i\leq N!\}$ with $\sum_{(k,i):P_k^i\in \mathrm{TrP}} \alpha_k^i \geq M/2N$ such that the following holds. We write $\hat{P}_k^i$ as the sub-path of $P_k^i$ after its last exit of $B(r_3)$. If $\beta a<1/2$, it holds that
		\begin{equation}\label{re3}
			(1-\beta a)^2R(v,A)\leq R(v,Z)+R(A,Z)+(M/N)^2R(\mbox{around } \mathcal A)-\sum_{(k,i):P_k^i\in \mathrm{TrP}}\sum_{e\in \hat P_k^i}(\alpha_k^i)^2r_{e,P_k,i}.
		\end{equation}
	\end{lem}
	\begin{proof}
		We remark that, although it is a natural attempt to apply Max-Flow-Min-Cut theorem to deduce Lemma \ref{returnlem7}, we do not see how this can be implemented since we have additional and rather non-conventional requirements, as incorporated in \eqref{re3}. Instead, it seems to us that a rather non-trivial proof, such as the one we carry out, is necessary.
		
		Note that $M\bar \rho/N$ can be decomposed as the sum of flows along cycles $C_1, \ldots, C_d$, each with strength $Ma$. We may then decompose each $P_k$ and $Q_l$ as the sum of flows on $N!/M$ copies of the same path, each with strength $Ma$.  With a slight abuse of notation, we will continue to use $P_k, Q_l$ to denote the copies of these paths. (We may now assume that each $P_k$ is a combination of $M$ copies in $\{P_k^1,\ldots,P_k^{N!}\}$, and for different $k$ these copies are disjoint.)
		
		The core of the proof relies on an inductive construction of a collection of triples $\mathcal T_M=\{(\tilde P_k, \tilde Q_l, \tilde C_t)\}$, where $\tilde P_{k}$ is a subset of $P_{k}$, $\tilde Q_{l}$ is a subset of $Q_{l}$, $\tilde C_t$ is a subset of $C_t$. We require that $\tilde{P}_k \cup \tilde{Q}_l \cup \tilde{C}_t$ forms a path from $v$ to $A$ (we allow $\tilde{C}_t$ to be either the empty set or a single vertex). We will occasionally slightly abuse the notation by viewing $\mathcal{T}_M$ as a collection of paths, where each path is given by the union of a triple in $\mathcal{T}_M$. For brevity, we will write $P_k$, $Q_l$ or $C_t$ is in $\mathcal T_M$ if $\tilde P_k$, $\tilde Q_l$ or $\tilde C_t$ occurs in some triple in $\mathcal T_M$. In what follows, we first describe our algorithm leading to the inductive construction, and then we analyze the algorithm and construct a flow out of $\mathcal T_M$ whose energy we then control.
		
		\noindent{\bf Description of the algorithm.}		
		We initiate by setting $\mathcal T_M=\emptyset$. In our construction, we will truncate paths, which is how we obtain $\tilde P_k$ and $\tilde Q_l$ from $P_k$ and $Q_l$ respectively. For clarity of checking, we announce up front that we will stick to the following convention: once a segment of $Q_l$ is truncated, this segment will be considered as removed throughout the algorithm (i.e., it will never be added back); in addition, once a segment of $P_k$ is added into $\mathcal{T}_M$, it will never be removed. (We will assume this as part of our inductive hypotheses and we will verify it inductively in the analysis of the algorithm.) At each step, take a path $P_k$ that is not in $\mathcal T_M$ and we consider the following {\bf procedure} with $P_k$. For convenience of description, we will describe the \emph{list} at $(x, P_k)$ for a point $x$ on some path $P_k$ as follows. Let $e$ be the edge on $P_k$ incident to $x$ whose other end point is before $x$ in the order along $P_k$. The list at $(x, P_k)$ contains all cycles (i.e., all $C_t$'s) containing $x$ with an ordering such that those cycles containing $e$ are listed before the other cycles.
		
		We walk along $P_k$ starting from $v$ (by ``walk along'' we mean to check each vertex/edge on $P_k$ in the order specified by the path). There are two scenarios as follows.
		
		\begin{figure}
			\centering
			\subfigure[Case 1.1]
			{
				\begin{minipage}[c]{0.3\textwidth}
					\centering
					\includegraphics[height=4.5cm,width=4.5cm]{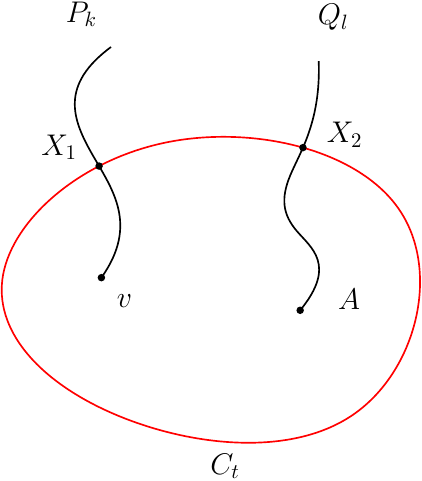}
					\label{Case11}
				\end{minipage}
			}
			\subfigure[Case 1.2(1)]
			{
				\begin{minipage}[c]{0.3\textwidth}	
					\centering			
					\includegraphics[height=4.5cm,width=4.5cm]{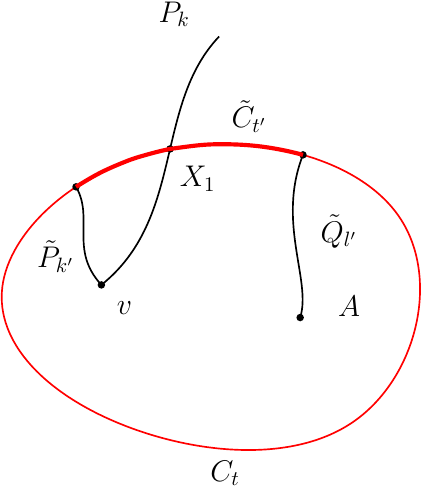}
					\label{Case12a}
				\end{minipage}
			}
			\subfigure[Case 1.2(2)]
			{
				\begin{minipage}[c]{0.3\textwidth}	
					\centering			
					\includegraphics[height=4.5cm,width=4.5cm]{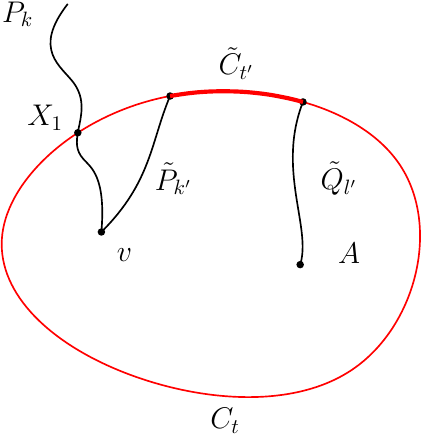}
					\label{Case12b}
				\end{minipage}
			}
			\subfigure[Case 1.3]
			{
				\begin{minipage}[c]{0.3\textwidth}
					\centering
					\includegraphics[height=4.5cm,width=4.5cm]{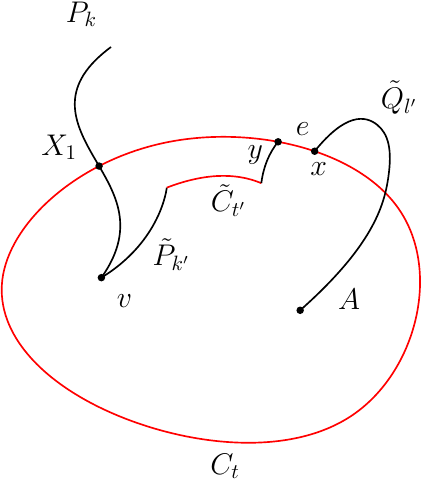}
					\label{Case13}
				\end{minipage}
			}
			\subfigure[Case 2.1]
			{
				\begin{minipage}[c]{0.3\textwidth}
					\centering
					\includegraphics{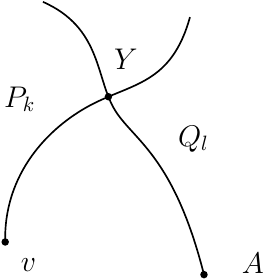}
					\label{Case21}
				\end{minipage}
			}
			\subfigure[Case 2.2]
			{
				\begin{minipage}[c]{0.3\textwidth}
					\centering
					\includegraphics{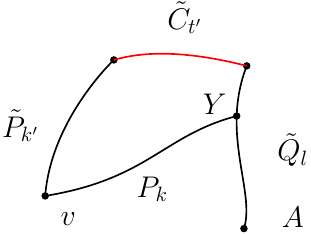}
					\label{Case22}
				\end{minipage}
			}
			\caption{Illustration for the five sub-cases in the procedure.}	
			\label{fig1}
		\end{figure}
		
		\begin{enumerate}
			\item\label{sen1} Whenever it intersects some cycle $C_t$ at some vertex $X_1$ (when with multiple choices for cycles, we will choose $C_t$ to be the first cycle in the list at $(X_1, P_k)$), we walk along $C_t$ starting from $X_1$ in the clockwise order until it first intersects some $Q_l$ that is not in $\mathcal T_M$ at some vertex $X_2$ (we allow $X_1=X_2$). We consider the arc of $C_t$ between $X_1$ and $X_2$ (along which we just walked), and divide it into the following three sub-cases.
			
			\begin{enumerate}
				\item[1.1]\label{al11} If this arc contains no common edge with all paths in $\mathcal{T}_M$, we truncate $P_k$, $Q_l$ and $C_t$ to get a path from $v$ to $A$ in the following way. We define $\tilde P_k$ to be the sub-path of $P_k$ from $v$ to $X_1$, define $\tilde C_t$ to be this arc of $C_t$ (with $X_1$ and $X_2$ as its endpoints), and define $\tilde Q_l$ to be the sub-path of $Q_l$ from $X_2$ to $A$. We then add the triple $(\tilde P_k, \tilde Q_l, \tilde C_t)$ into $\mathcal{T}_M$. (See Figure \ref{Case11} for an illustration.)
				
				\item[1.2]\label{al12} If this arc has a common edge with $\tilde C_{t^{\prime}}$ for some triple $(\tilde P_{k^\prime}, \tilde Q_{l^\prime}, \tilde C_{t^{\prime}})\in \mathcal{T}_M$, since we have assumed that $C_t$'s have no common edge, it holds that $t=t^{\prime}$. In this case, if $X_1$ is an interior point of $\tilde C_{t^{\prime}}$ (i.e., a point on $\tilde C_{t^\prime}$ different from its two endpoints), we replace the triple $(\tilde P_{k^\prime}, \tilde Q_{l^\prime}, \tilde C_{t^{\prime}})$ in $\mathcal T_M$ by $(\tilde P_{k}, \tilde Q_{l^\prime}, \tilde C_{t^{\prime}})$ (where $\tilde P_k$ denotes the sub-path of $P_k$ from $v$ to $X_1$) and then we stop our procedure with $P_k$ and continue our procedure with $P_{k^{\prime}}$ (that is, we walk along $P_{k^\prime}$ starting from $v$ and check which scenario will occur). (See Figure \ref{Case12a} for an illustration.) Otherwise (i.e,. $X_1$ is not an interior point of $\tilde C_{t^\prime}$), we consider the list at $(P_k,X_1)$: if in the list there exists $C_{t^{\prime\prime}}$ after $C_t$, we will walk along the next $C_{t^{\prime\prime}}$; if $C_t$ is the last cycle in the list, we keep walking along $P_k$ from $X_1$. (See Figure \ref{Case12b} for an illustration.)
				
				\item[1.3]\label{al13} If Case 1.2 does not occur and if this arc has a common edge $e$ with $\tilde Q_{l^{\prime}}$ for some triple $(\tilde P_{k^\prime}, \tilde Q_{l^\prime}, \tilde C_{t^{\prime}})\in \mathcal{T}_M$ (when there are multiple common edges, we choose $e$ as the first such edge in the order from $X_1$ to $X_2$), we let $x,y$ be two endpoints of $e$ such that $x$ lies between $y$ and $A$ on $\tilde Q_{l^\prime}$. We replace the triple $(\tilde P_{k^\prime}, \tilde Q_{l^\prime}, \tilde C_{t^{\prime}})$ in $\mathcal T_M$ by $(\tilde P_{k}, \tilde Q_{l^\prime}^\prime, \tilde C_{t})$, where $\tilde P_k$ denotes the sub-path of $P_k$ from $v$ to $X_1$, $\tilde Q_{l^\prime}^\prime$ denotes the sub-path of $\tilde Q_{l^\prime}$ from $x$ to $A$, and $\tilde C_t$ denote the sub-arc of $C_t$ with $X_1$ and $x$ as its endpoints. Then we stop our procedure with $P_k$ and continue our procedure with $P_{k^\prime}$. We will inductively prove later that during the procedure with $P_{k'}$ the procedure will be continued at least until we encounter $X_1$; note that $X_1$ may be included by multiple cycles and as a convention (in case of multiple cycles) we will walk along the cycle $C_{t'}$. (See Figure \ref{Case13} for an illustration.)
			\end{enumerate}	
			
			\item\label{sen2} If scenario 1 does not happen, whenever it intersects some $Q_l$ at some vertex $Y$, we consider the following two sub-cases.
			\begin{enumerate}
				\item[2.1]\label{al21} If $Q_l$ is not in $\mathcal{T}_M$, we truncate $P_k$ and $Q_l$ to get a path $\tilde P_k \cup \tilde Q_l$, where $\tilde P_k$ is the sub-path of $P_k$ from $v$ to $Y$ and $\tilde Q_l$ is the sub-path of $Q_l$ from $Y$ to $A$. We then add $(\tilde P_k,\tilde Q_l,\emptyset)$ into $\mathcal T_M$. (See Figure \ref{Case21} for an illustration.)
				
				\item[2.2]\label{al22} If $Q_l$ is in $\mathcal T_M$, we suppose that $(\tilde P_{k^\prime}, \tilde Q_{l}, \tilde C_{t^{\prime}})\in \mathcal{T}_M$. If $Y$ is not an interior point of $\tilde Q_l$, we keep walking along $P_k$. If $Y$ is an interior point of $\tilde Q_l$, we replace the triple $(\tilde P_{k^\prime}, \tilde Q_{l}, \tilde C_{t^{\prime}})$ in $\mathcal T_M$ by $(\tilde P_{k}, \tilde Q^{\prime}_{l}, \emptyset)$ (where $\tilde P_k$ denotes the sub-path of $P_k$ from $v$ to $Y$, and $\tilde Q^{\prime}_{l}$ denotes the sub-path of $\tilde Q_{l}$ from $Y$ to $A$). Then we stop our procedure with $P_k$ and continue our procedure with  $P_{k^\prime}$. We will inductively prove later that during the procedure with $P_{k'}$ the procedure will be continued at least until we encounter $X_1$; note that $X_1$ may be included by multiple cycles and as a convention (in case of multiple cycles) we will walk along the cycle $C_{t'}$. (See Figure \ref{Case22} for an illustration.)
			\end{enumerate}
		\end{enumerate}
		
		\begin{figure}
			\centering
			\includegraphics{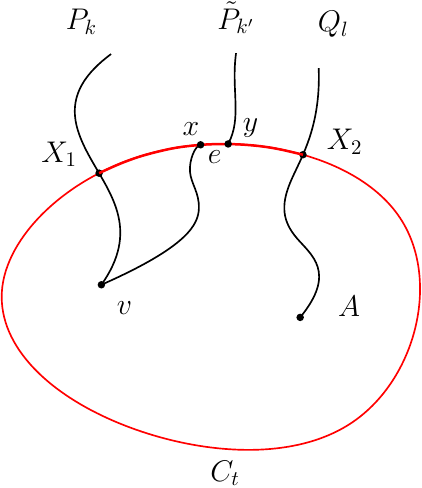}
			\caption{The case that the sub-arc of $C_t$ from $X_1$ to $X_2$ have some common edge with some $\tilde P_{k^\prime}$.}
			\label{fig10}
		\end{figure}
		
		\noindent{\bf Analysis of the algorithm.}
		We first prove that our description of the algorithm is sound. To this end, we need to show that sub-cases in each scenario capture all possibilities. Since this is obvious for the second scenario, we next focus on the first scenario. To this end, it suffices to consider the case that in the first scenario (at time $\tau_0$) when walking along some $C_t$, the arc we obtained have a common edge with some $\tilde P_{k^\prime}$ in $\mathcal T_M$. We will show that we must enter Case 1.2. We consider the last edge $e$ (with endpoints $x,y$ in the clockwise order) on this arc we have walked through that belongs to some $\tilde P_{k^\prime}$. Since $\tilde P_{k^\prime}$ might have grown during the running of the algorithm, we consider the first time (denoted by $\tau$, where $\tau<\tau_0$) when a segment of $\tilde P_{k^\prime}$ containing $e$ was added to $\mathcal T_M$. (See Figure \ref{fig10} for an illustration.) Since we assumed that $\tilde P_j$ never decreases, the sub-arc of $C_t$ from $y$ to $X_2$ contains no common edge with all $\tilde P_j$'s in $\mathcal{T}_M$ at time $\tau$. As a result, a part of the sub-arc of $C_t$ from $y$ to $X_2$ was in $\mathcal{T}_M$ at time $\tau$ (since otherwise $\tilde P_{k^\prime}$ will be matched to this part of $C_t$). Since $Q_l$ (recall that $Q_l$ is the first path that is not in $\mathcal T_M$ and intersects with $C_t$ while we are walking along $C_t$) is not in $\mathcal T_M$ at any time in $[\tau,\tau_0]$, any possible extension of this sub-arc in Case 1.2 must end before $X_2$. In addition, by the rules in Cases 1.3 and 2.2, whenever the triple that contains a sub-arc of $C_t$ between $y$ and $X_2$ is replaced, we will walk along $C_t$ in the following procedure with some path $P_{k^\prime}$ (see the convention as specified at the end of Cases 1.3 and 2.2, where $t^\prime$ corresponds to the $t$ here). In conclusion, there is always a part of the sub-arc of $C_t$ from $y$ to $X_2$ that stays in $\mathcal{T}_M$ during $[\tau,\tau_0]$. Therefore, when we are considering $P_k$ at time $\tau_0$, we must enter Case 1.2 (and more specifically, the case as in Figure \ref{Case12b} happens).
		
		We then need to check the inductive hypothesis that each $\tilde P_k$ never decreases, which then also verifies the inductive claim at the end of Cases 1.3 and 2.2. It suffices to check this property for $\tilde P_{k^\prime}$ when we decided to continue our procedure with $\tilde P_{k^\prime}$ (i.e., we show that when walking along $P_{k^\prime}$ in the procedure, we must walk past $\tilde P_{k^\prime}$). If $\tilde P_{k^\prime}$ was in a triple of $\mathcal T_M$ before the modification to $\mathcal T_M$ in Cases 1.2, 1.3 and 2.2, we claim that, after this modification during our procedure with $P_{k^\prime}$, we must walk past $\tilde P_{k^\prime}$ before ending the procedure with $P_{k^\prime}$ (and we denote this time as $\tau_1$), thereby ensuring the desired monotonicity for $\tilde P_{k^\prime}$. In order to verify the claim, suppose that for some $\tilde P^{\prime}_{k^\prime}\subsetneqq\tilde P_{k^\prime}$, $(\tilde P^{\prime}_{k^\prime}, \tilde Q_{l},\tilde C_t)$ was added into $\mathcal{T}_M$ at time $\tau_1$. Consider a previous time $\tau_2<\tau_1$ that we walked past $\tilde P^{\prime}_{k^\prime}$ to get $\tilde P_{k^\prime}$. Then at time $\tau_2$, $\tilde C_t$ did not have a common edge with all paths in $\mathcal{T}_M$, and $\tilde Q_l$ was not truncated (note that it is obvious that for any $\tilde C_t$ in $\mathcal T_M$ it has no common edge with any $\tilde Q_l \in \mathcal T_M$). This then leads to a contradiction: in all relevant cases, we would have added a triple with $\tilde P^\prime_{k^\prime}$ instead of with $\tilde P_{k^\prime}$ to $\mathcal T_M$. Thus we complete the proof of the desired claim.
		This observation ensures that $\tilde P_k$ never decreases. In addition, we mention by passing that it is obvious that $\tilde Q_l$ is decreasing.
		
		We now show that at any time of our algorithm $\tilde C_t$'s in $\mathcal{T}_M$ can not have a common edge with $\tilde{P}_k$'s in $\mathcal T_M$. In fact, when $\tilde C_t$ is first added into $\mathcal T_M$ in Case 1.1, it holds directly by the rules of Case 1.1. After $\tilde C_t$ have been added in $\mathcal T_M$, during the procedure with some $P_k$ whenever a common edge is found between $\tilde C_t$ and some $\tilde P_{k}$, as incorporated in Case 1.2 we will truncate $\tilde C_t$ and keep this claim true.
		
		We then claim that the algorithm must terminate. Note that if each step terminates, at each step a new $P_k$ will be added into $\mathcal T_M$, and thus the algorithm must terminate. Now it suffices to prove that each step must terminate. In fact, at each time when Case 1.3 or Case 2.2 occurs, $\cup \tilde Q_l$ strictly decreases. It suffices to consider each time when Case 1.2 occurs, if we keep walking along $P_k$ then $\cup \tilde P_k$ will strictly increase; if we switch our procedure with $P_k$ to the procedure with $P_{k^\prime}$, we see that $\cup \tilde C_t$ strictly decrease, so we must eventually enter Case 1.1, 1.3, 2.1 or 2.2. Therefore, either Case 1.1 or Case 2.1 eventually occurs and this step will terminate at this time. As a result, all $P_k$'s must be eventually contained in $\mathcal{T}_M$.
		
		We claim that there are a few $P_k$'s with flows of total strength $M/2N$ which have been truncated before their last exits of $B(r_3)$, and we define $\mathrm{TrP}$ to be the collection of all $P_k^i$'s contained in these paths. In the case that all $C_t$'s occur at least once in $\mathcal{T}_M$ the claim holds immediately since $M(1-\beta a)/N>M/2N$. In the case that some $C_t$ does not occur in $\mathcal{T}_M$, we claim that all $P_k$'s have been truncated before their first exit of $B(r_3)$ so that in this case $\mathrm{TrP}$ can be any collection of $P_k$'s with flows of total strength $M/2N$. If the claim does not hold, we consider the last path $P_k$ with which our procedure encounter a vertex in $C_t$ when walking along $P_k$. For Case 1.3, $P_{k^\prime}$ will encounter $C_t$ again since $P_j$'s never decrease (we use the same notation in Cases 1.2 and 1.3), and for Cases 1.1 and 1.2, $C_t$ will be added in $\mathcal{T}_M$; either of these cases leads to a contradiction (to ``last encounter'' or to ``$C_t$ not added'').
		
		\noindent{\bf Construction of the flow and the analysis of its energy.}		
		Now we send a flow $\xi$ from $v$ to $A$, with strength $a/N$ on each path in $\mathcal{T}_M$. Then the total strength is greater than $1-\beta a$. For $P_k^i\in \mathrm{TrP}$, recall that $\hat{P}_k^i$ denotes the sub-path of $P_k^i$ after its last exit of $B(r_3)$ (and we set $\hat P_k^i = \emptyset$ if $P_k^i \not\in \mathrm{TrP}$ for convenience). We write (recall that $\alpha_k^i = a$)
		$$
		\hat \theta_v(e)=\bar \theta_v(e)-\sum_{P_k^i\in \mathrm{TrP}}\alpha_k^i\cdot \mbox{sign}(\bar \theta_v(e)) \mathds{1}_{\{e\in \hat P_k^i\}}.
		$$
		We claim that 
		\begin{equation}
			|\xi(e)|\leq |\hat{\theta}_v(e)|\vee |\bar \theta_A(e)|\vee |M/N\cdot \bar \rho(e)|.
		\end{equation}
		Since the edges in $\tilde C_t$ are disjoint from edges in $\tilde P_k$'s and $\tilde Q_l$'s, it suffices to consider the edges that are contained in both $\tilde P_k$ and $\tilde Q_l$ for some $k$ and $l$. For such an edge $e$, we will in fact show that $|\xi(e)| \leq |\bar \theta_A(e)|$. To this end, we first note that until $e$ was contained in some $\tilde Q_l$ in $\mathcal T_M$, we cannot walk past $e$ when walking along any $P_k$ (this is because we will run into the first scenario or Case 2.1 before walking past $e$). Now we assume that $e$ was already contained by some $\tilde Q_l$ when implementing the procedure with some $P_k$. Then, $\tilde P_k$ can walk through $e$ only in Case 2.2, where in addition $Y$ (using the notation in Case 2.2) is not an interior point of $\tilde Q_{l}$; this is exactly when the edge $e$ is truncated from $\tilde Q_l$ (and as shown above such truncation is forever).
		So in this case it holds that $|\xi(e)|\leq |\bar \theta_A(e)|$.
		
		By \eqref{resistancedef} we have
		\begin{equation}\label{re1}
			(1-\beta a)^2R(v,A)\leq \mathcal{E}(\hat \theta_v)+\mathcal{E}(\bar \theta_A)+(M/N)^2R(\mbox{around } \mathcal{A}).
		\end{equation}
		We consider the following equivalent network. Recalling that
		$r_{e,P_k,i}=\mathds{1}_{\{e\in P_k^i\}}\frac{N!\bar \theta_v(e)}{\alpha_k}r_e+\mathds{1}_{\{e\notin P_k^i\}}\cdot \infty$, we substitute the edge $e$ by a parallel of edges with resistance $r_{e,P_k,i}$. If we send a flow from $v$ to $Z$ by assigning a flow with strength $\alpha_k^i$ on each $P_k^i$ (where the edge resistance is given by $r_{e, P_k, i}$), then the Dirichlet energy is
		$$\sum_{k,i}\sum_{e\in P_k^i}(\alpha_k^i)^2r_{e,P_k,i}=\sum_{e}|\bar \theta_v(e)|r_e\sum_{(k,i):e\in P_k^i}\alpha_k^i=\sum_{e}|\bar\theta_v(e)|^2r_e.$$
		So the above flow is indeed the electric current. Furthermore,
		\begin{equation}\label{re2}
			\begin{aligned}
				\mathcal{E}(\hat \theta_v)
				&=\sum_{e}\left(|\bar\theta_v(e)|-\sum_{(k,i):e\in \hat P_k^i}\alpha_k^i\right)^2r_e\\
				&=\sum_e |\bar\theta_v(e)|^2r_e-2\sum_{e}\left(\sum_{(k,i):e\in \hat P_k^i}\alpha_k^i\right)|\bar\theta_v(e)|r_e+\sum_e\left(\sum_{(k,i):e\in \hat P_k^i}\alpha_k^i\right)^2r_e\\
				&\leq \sum_e \bar\theta_v(e)^2r_e-\sum_{e}\left(\sum_{(k,i):e\in \hat P_k^i}\alpha_k^i\right)|\bar\theta_v(e)|r_e\\
				&=\mathcal{E}(\bar\theta_v)-\sum_{(k,i):P_k^i\in \mathrm{TrP}}\sum_{e\in \hat P_k}(\alpha_k^i)^2r_{e,P_k,i}.
			\end{aligned}
		\end{equation}
		Combining \eqref{re1} and \eqref{re2} with the fact that $\mathcal E(\bar\theta_v)\leq R(v,Z)$, $\mathcal E(\bar \theta_A)\leq R(A,Z)$, we complete the proof.
	\end{proof}
	
	\begin{proof}[Proof of Lemma \ref{returnlem3}]
		Now we inductively define disjoint index sets $I_s$ such that $\sum_{(k,i)\in I_s}\alpha_k^i=2^{s-1}/2N$. We take $\mathrm{TrP}_{2^s}$ from Lemma \ref{returnlem7} with $M=2^s$ (for $0\leq s\leq \log N$). We set $I_1=\{(k,i):P_k^i\in \mathrm{TrP}_1\}$. Suppose now we have defined $I_{s-1}$, since the $P_k^i$'s in $\mathrm{TrP}_{2^s}$ have a total strength $2^{s}/2N$, there exists an index set $I_s$ disjoint from $I_1,\ldots,I_{s-1}$ that satisfies our assumption. For all $(k,i)\in \cup_{1\leq s\leq K} I_s$ we see that $P_k^i$'s are disjoint paths crossing $\mathcal{B}$. For $K$ with $2^{K}<N$ we have
		\begin{equation}
			R(\mbox{across }\mathcal B)\leq \sum_{s=1}^{K} \left(\sum_{(k,i)\in I_s}\sum_{e\in \hat P_k^i}(\alpha_k^i)^2r_{e,P_k,i}\right)/(2^sK/2N)^2.
		\end{equation}
		So there exists $s<K$ such that 
		\begin{equation}\label{re4}
			\sum_{(k,i)\in I_s}\sum_{e\in \hat P_k^i}(\alpha_k^i)^2r_{e,P_k,i}\geq K\cdot (2^s/2N)^2R(\mbox{across }\mathcal B).
		\end{equation}
		By taking $K>4\mathsf C$ and $N>2^{K}$ (note that this implies that $N$ depends only on $\mathsf C$), combining \eqref{re3} (where we take $M=2^s$) and \eqref{re4} we have
		\begin{equation}
			(1-\beta a)^2R(v,A)\leq R(v,Z)+R(A,Z)+(2^s/N)^2R(\mbox{around } \mathcal A)-2\mathsf C\cdot(2^s/N)^2R(\mbox{across }\mathcal B).
		\end{equation}
		Since $R(\mbox{around }\mathcal{A})\leq \mathsf C R(\mbox{across } \mathcal{B})$, this completes our proof by taking $a\to 0$.
	\end{proof}
	
	\begin{proof}[Proof of Lemma \ref{returnlem}]
		Since $v\notin A\cup Z$, applying \eqref{return6}	and then applying Lemma \ref{returnlem3}, we have
		\begin{equation}
			P^v(\tau_A<\tau_Z)
			=\frac{R(v,Z)+R(A,Z)-R(v,A)}{2R(A,Z)}
			\geq\frac{rR(\mbox{around } \mathcal{A})}{2R(A,Z)}\geq q,
		\end{equation}
		for some $r=r(\mathsf C )$ and for some $q=q(\mathsf C ,\mathsf c )>0$.
	\end{proof}
	
	The following analogue of Lemma \ref{returnlem} will be useful later too.
	\begin{lem}\label{returnlem2}
		Let $0\leq r_1<r_2<r_3<r_4$ be four integers. Denote $\mathcal{A}=B(r_3)\setminus B(r_2)$, $\mathcal{B}=B(r_2)\setminus B(r_1)$, $A=\partial B(r_1)$ and $Z=\partial B(r_4)$. If for some constants $\mathsf c, \mathsf C>0$ we have
		\begin{equation}
			\mathsf c R(A,Z)\leq R(\mbox{around }\mathcal{A})\leq \mathsf C R(\mbox{across }\mathcal{B}),
		\end{equation}
		then there exists $q=q(\mathsf C, \mathsf c)\in (0,1)$ such that for any $v\in B(r_4) \setminus B(r_3)$,
		\begin{equation}
			P^v(\tau_A>\tau_Z)\geq q.
		\end{equation}
	\end{lem}
	\begin{proof}
		In light of \eqref{return6}, we just need to upper-bound the right-hand side of \eqref{return6}, whose proof is highly similar to that for Lemma \ref{returnlem3} (we just need to switch the roles of $A$ and $Z$). Thus, we omit further details.
	\end{proof}

	In light of Lemma \ref{returnlem2}, we have the following bound for the trace of the random walk. For $x\in \mathbb R^2$, let $ S_k(x)$ be the dyadic square in $\mathcal S_k$ which contains $x$, and let $\hat S_k(x)$ be the square with the same center as $S(x)$ and with side length $3\times2^{-k}$. For $n>0$, let $\tau^n_0=0$ and inductively let
	\begin{equation}\label{stoppingtime}
		\tau^n_i=
		\begin{cases}
			\inf\{s>\tau^n_{i-1}:X^{(n,x)}_s\in \partial \hat{S}(X^{(n)}_{\tau^n_{i-1}})\} & \mbox{if }X^{(n)}_{\tau^n_{i-1}}\in B(1), \\
			\tau^n_{i-1} & \mbox{if }X^{(n)}_{\tau^n_{i-1}}\notin B(1).
		\end{cases}
	\end{equation}
	\begin{lem}\label{exitlem2}
		For all $n>0$ and $\delta\in (0,1)$, there exists an integer $t=t(k,\delta)$ such that
		\begin{equation}
			\mathbb{P}(P(\tau^n_t=\tau^n_{t-1})\geq 1-\delta)\geq 1-\delta.
		\end{equation}
	\end{lem}
	\begin{proof}
		For each $S\in\mathcal{S}_k$ and $v\in \hat S$, there exist rational numbers $0\leq r_1<r_2<r_3<r_4$ such that $S=v_S+B(r_1)$, $\partial \hat S\subset v_S+B(r_4)\setminus B(r_3)$, and $B(1)\subset v_S+B(r_4)\setminus B(r_3)$. We define $A(S)=\partial S$, $Z(S)=\partial (v_S+B(r_4))$, $\mathcal{A}(S)=v_S+B(r_3)\setminus B(r_2)$ and $\mathcal{B}(S)=v_S+B(r_2)\setminus B(r_1)$. Then by Theorem \ref{thm1}, for any $q\in(0,1)$, there exist $C_1=C_1(q)$ and $c_1=c_1(q)$ such that
		$$\mathbb{P}(c_1R(A(S),Z(S))\leq R(\mbox{around }\mathcal{A}(S))\leq C_1R(\mbox{across }\mathcal{B}(S)))\geq 1-q.$$
		By a union bound, we have
		\begin{equation}
			\mathbb P(cR(A(S),Z(S))\leq R(\mbox{around }\mathcal{A}(S))\leq C_1R(\mbox{across }\mathcal{B}(S))\mbox{ for each } S \in \mathcal S_k)\geq 1-2^{2k+2}q.
		\end{equation}
		
		In what follows we assume that the event in the preceding inequality holds. By Lemma \ref{returnlem2}, for some $r =r (c_1,C_1)$ (thus depending only on $k,q$) and for all $v\in \partial \hat S\subset v_S+B(r_4)\setminus B(r_3)$, we have $P^v(\tau_{A(S)}>\tau_{Z(S)})\geq r$.
		By the strong Markov property, we have
		$$P\left(\exists j>i:X^{(n)}_{\tau_j^n}\in S| X^{(n)}_{\tau_i^n}\in S\right)=E\left[P^{X^{(n)}_{\tau_{i+1}^n}}(\tau_{A(S)}<\tau_{Z(S)})\right]\leq 1-r.$$
		Thus we have
		$$P(|\{i:X^{(n)}_{\tau^n_i}\in S\}|\geq l+1)\leq (1-r)^l.$$
		Taking $t=2^{2k+2}(l+1)$, we have $$P(\tau^n_t\neq \tau^n_{t-1})\leq \sum_{S\in \mathcal S_k}P(|\{i:X^{(n)}_{\tau^n_i}\in S\}|\geq l+1)\leq 2^{2k+2}(1-r)^l.$$
		Taking $q=\delta/2^{2k+2}$ and taking $l$ sufficiently large depending only on $k,r$ (thus depending only on $k,\delta$), we complete the proof from the preceding inequality.
	\end{proof}
	
	\subsection{Coupling the exit measure}\label{subsec72}
	In this subsection, we prove that the exiting measure is continuous with respect to the starting point of the random walk, as in the next lemma. We will write $d_{\mathrm{TV}}$ for the total variation distance between two random variables.
	
	\begin{lem}\label{conti}
		Let $K\subset B(1)$ be a compact set and let $\tau^x$ be the first time when $X^{(n,x)}$ hits $\partial B(1)$. For each $\mathsf C ,\epsilon >0$, there exist $\delta = \delta(K, \mathsf C , \epsilon )>0$ and $C_2=C_2(K) >0$, such that for sufficiently large $n$, it holds with $\mathbb{P}$-probability at least $1-C_2\epsilon^\mathsf C $ that 
		\begin{equation}\label{contieq}
			d_{\mathrm{TV}}(X^{(n,x)}_{\tau^x},X^{(n,y)}_{\tau^y})\leq \epsilon \mbox{ for all } x, y\in K \mbox{ with } |x-y|\leq \delta.
		\end{equation}
	\end{lem}

	In order to prove Lemma \ref{conti}, we will apply the following lemma from \cite{gwynne2022invariance}.
	
	\begin{lem}\label{exitlem1}
		\textup{(\hspace{-0.15mm}\cite[Lemma~3.12]{gwynne2022invariance})} Let $G$ be a connected graph and let $A\subset V(G)$ be a set such that the random walk started from any vertex of $G$ a.s. hits $A$ in finite time. For $x\in V(G)$, let $X^x$ be the 
		random walk started from $x$ and let $\tau^x$ be the first time when $X^x$ hits $A$. For $x,y\in V(G)\setminus A$,
		\begin{equation}\label{exit1}
			d_{\mathrm{TV}}(X^x_{\tau^x},X^y_{\tau^y})\leq 1-P(X^x \mbox{ disconnects } y \mbox{ from } A \mbox{ before time } \tau^x),
		\end{equation}
		where the disconnecting event is that any path from $y$ to $A$ must intersect with $X^x_t$ for some $t\leq\tau^x$.
	\end{lem}
	
	In light of Lemma \ref{exitlem1}, it suffices to bound the disconnecting probability in \eqref{exit1} in order to prove Lemma \ref{conti}; this is proved in the next lemma. For a $2^{-k}\times 2^{-k}$ dyadic square $S$, recall that $\hat{S}$ is the square with the same center as $S$ with side length $3\times2^{-k}$.
	\begin{lem}\label{exitlem}
		For integers $n>l>k$ and for any $\mathsf C>0$, there exists a constant $\alpha=\alpha(\mathsf C)>0$ such that the following holds with $\mathbb P$-probability at least $1 - \exp(-\mathsf C(l-k))$. For any two points $x, y\in S\cap\mathbb Z_n^2$ with $|x-y| \leq 2^{-l}$, letting $X^x$ be the random walk started from $x$, we then have that 
		\begin{equation}\label{return7}
			P(X^x \mbox{ disconnects } y \mbox{ from } \partial \hat S \mbox{ before exiting } \hat S)\geq 1-\exp\left(-\alpha(l-k)\right).
		\end{equation}
	\end{lem}

	\begin{figure}		
		\centering		
		\begin{minipage}[c]{0.45\textwidth}
			\centering
			\includegraphics[height=5.5cm,width=5.5cm]{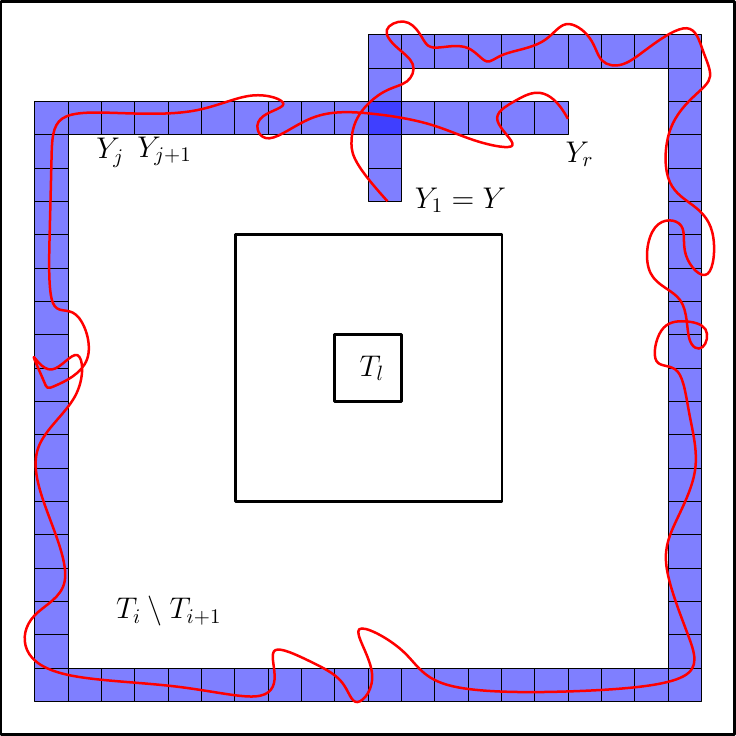}
		\end{minipage}
		\begin{minipage}[c]{0.45\textwidth}			
			\centering			
			\includegraphics[height=5.5cm]{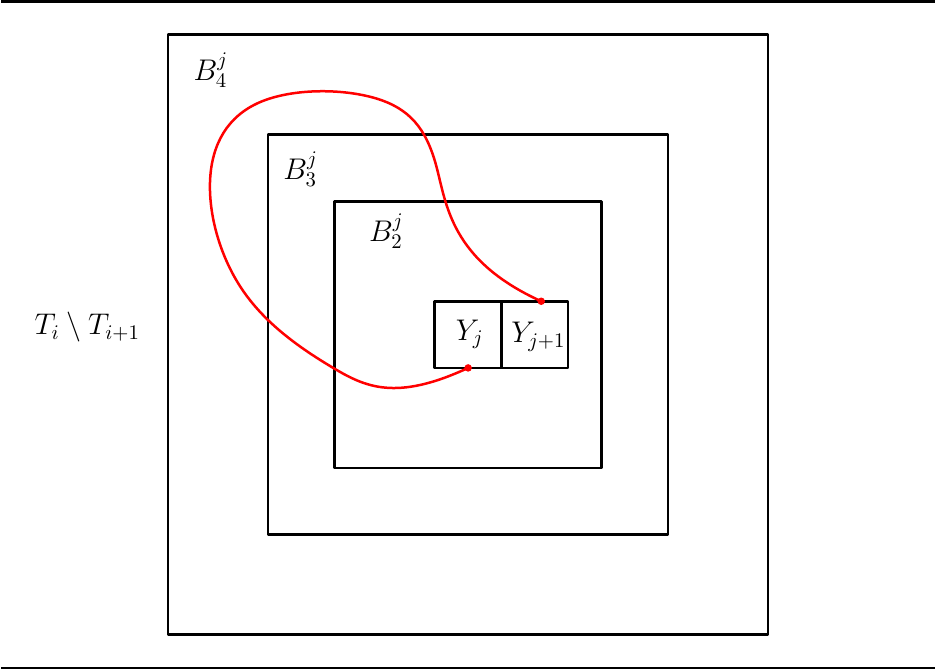}
		\end{minipage}
		\caption{Illustration of Lemma \ref{exitlem}}	
		\label{fig8}
	\end{figure}
	\begin{proof}
		For any $|x-y|\leq 2^{-l}$, we can find a $2^{-l}\times2^{-l}$ square $T_l$ (whose corners are multiples of $(2^{-l-1}, 2^{-l-1})$) such that $x,y\in T_l$. For $k\leq i\leq l$, let $T_i$ be the the concentric square of $T_l$ with side length $2^{-i}$.
		
		In the annulus $T_i\setminus T_{i+1}$, we define $\mathcal Y$ to be the set of dyadic squares with side length $2^{-i-100}$ that have $L^{\infty}$-distance at least $2^{-i-10}$ from $\partial (T_i\setminus T_{i-1})$, and define $\mathcal Y^\prime$ to be the set of dyadic squares with side length $2^{-i-100}$ that have $L^{\infty}$-distance $2^{-i-1}$ from $\partial T_i$. For $Y\in \mathcal Y$, let $B_1(Y),B_2(Y),B_3(Y),B_4(Y)$ be boxes with the same center, such that $B_1(Y)=Y$ and the side length of $B_{j+1}(Y)$ is 4 times that of $B_j(Y)$ for $j=1,2,3$. Define $E_i$ to be the event that for all $Y\in \mathcal Y$
		\begin{equation}
			c_2R(\partial B_1(Y),\partial B_4(Y))\leq R(\mbox{around } \partial B_3(Y)\setminus \partial B_2(Y))\leq C_3R(\partial B_4(Y),\partial B_3(Y)).
		\end{equation}
		Now we consider each $Y\in \mathcal{Y}^\prime$. We may choose $Y=Y_1,\ldots,Y_r$ with $r$ less than an absolute constant such that these $Y_j$'s form a partition for a ``tube'' where $Y_j$ shares an edge with $Y_{j+1}$, and that any path with Hausdorff distance at most $64\cdot 2^{-i-100}$ from the tube must disconnect $\partial T_i$ from $\partial T_{i+1}$. As a result, if the random walk visits $Y_j$ after $Y_{j+1}$ before leaving $B_4(Y_j)$, then its trace must disconnect $\partial T_i$ and $\partial T_{i+1}$. (See Figure \ref{fig8} for an illustration.) On the event $E_i$, for $1\leq j\leq r$ and for $v\in Y_{j+1}$, we apply Lemma \ref{returnlem} with $A=\partial B_1(Y_j)$ and $Z=\partial B_4(Y_j)$, and derive that
		\begin{equation}\label{exit2}
			P^v(X^v \mbox{ hits } Y_{j} \mbox{ before exiting } B_4(Y_j))\geq q,
		\end{equation}
		where $q=q(C_3,c_2)>0$. Thus by the strong Markov property we have for any $v\in \cup_{1\leq j\leq r}Y_j$,
		\begin{equation}\label{return9}
			P^v(X^v \mbox{ disconnects } T_i \mbox{ from } T_{i+1} \mbox{ before exiting } T_{i+1}\setminus T_i)\geq q^r.
		\end{equation}
		
		For any $p\in (0, 1)$, by Theorem \ref{thm1}, there exist constants $c_2=c_2(p)$ and $C_3=C_3(p)$ such that $\mathbb{P}(E_i) \geq p$ (note that here we applied a union bound over all $1\leq j\leq r$).
		Since we can choose $p$ close to 1 (depending on $\mathsf C$), we can adapt the proof of Lemma \ref{decaynumber} verbatim and get that
		\begin{equation}\label{return10}
			\mathbb{P}(\#\{k\leq i\leq l:E_i^{c} \mbox{ occurs}\}\geq (l-k)/2)\leq \exp(-\mathsf C(l-k)).
		\end{equation}
		On the event in \eqref{return10}, each time we visit $T_i\setminus T_{i+1}$ where $E_i$ occurs, by \eqref{return9} the disconnecting event in \eqref{return7} occurs with $P$-probability greater least $q^r$. Thus by the strong Markov property we have 
		$$P(X^x \mbox{ disconnects } y \mbox{ from } \partial B(1) \mbox{ before exiting } \hat S)\geq 1-(1-q^r)^{(l-k)/2}.$$
		Combined with \eqref{return10} and a union bound over $T_l$, it completes the proof.
	\end{proof}

	\begin{proof}[Proof of Lemma \ref{conti}]
		Fix $k$ such that $K\subset B(1-2^{-k})$ and assume without loss of generality that $\mathsf C \geq 10$ and $\alpha<1$. For sufficiently large $n$, applying Lemmas \ref{exitlem1} and \ref{exitlem} with $l$ such that $\exp(-\alpha(l-k))<\epsilon$, we have for each $S\in \mathcal S_k$,
		$$d_{\mathrm{TV}}(X^{(n,x)}_{\tau^x},X^{(n,y)}_{\tau^y})\leq \epsilon \mbox{ for all } x, y\in S \mbox{ with } |x-y|\leq 2^{-l}$$
		holds with $\mathbb P$-probability at least $1-\epsilon^\mathsf C$. Thus by a union bound we have: with $\mathbb{P}$-probability at least $1-2^{2k+2}\epsilon^\mathsf C$, \eqref{contieq} holds, which completes the proof by taking $\delta=2^{-l}$.
	\end{proof}
	
	\subsection{Tightness of the random walk traces}\label{subsec73}
	In this subsection, we complete the proof of Theorem \ref{thm2}, i.e., we prove tightness for the random walk traces. Recall the definition of $d_{\mathrm{CMP}}$ from Remark \ref{rmk1}. It was shown in \cite[Lemma~2.1]{aizenman1998holder} that $d_{\mathrm{CMP}}$ induces a complete metric on the set of curves viewed modulo time parameterization.
	
	For any dyadic point $x$ and any dyadic square $S$ that contains $x$, we define the random measure $\mu_n(x,S)$ to be the measure on the exit point when the random walk on $\mathbb Z_n^2$ started at $x$ exits $S$ for the first time (i.e., this is the harmonic measure/exit measure on $\partial S$). Via a diagonal argument, for any subsequence of ${X^{(n)}}$, there exists a further subsequence ${X^{(n_m)}}$ such that
	\begin{equation}\label{weakcon}
		\mu_{n_m}(x,S) \mbox{ converges weakly for all } x \mbox{ and } S.
	\end{equation}
	To lighten the notation, we write $n$ for $n_m$ in the rest of this subsection.
	
	\begin{figure}
		\centering
		\includegraphics[width=60mm]{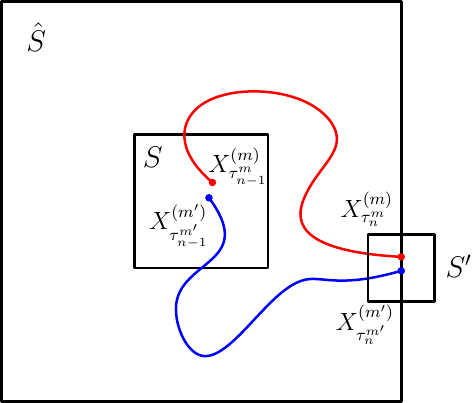}
		\caption{Coupling random walks in different scales}
		\label{fig3}
	\end{figure}
	
	\begin{proof}[Proof of Theorem \ref{thm2}]
		We fix a small $\delta$ and let $k=\lfloor\log \delta^{-1} \rfloor+1$.  Recall the definition of $\tau^n_i$ in \eqref{stoppingtime}. By Lemma \ref{exitlem2} there exists $t=t(\delta)$ such that with $\mathbb{P}$-probability at least $1-\delta/2$, 
		\begin{equation}\label{tight1}
			P(\tau^n_t=\tau^n_{t-1})\geq 1-\frac{\delta}{2}.
		\end{equation}
		
		Now we fix $\epsilon>0$ and an integer $l$ to be determined later. By \eqref{weakcon} there exists some $N=N(\epsilon,l)$ such that the following holds for all $n, n' \geq N$: for all $x\in 2^{-l}\mathbb{Z}^2$, with $\mathbb{P}$-probability at least $1-\delta/2$, it holds that
		\begin{equation}\label{tight3}
			\sum_{S^{\prime}\in \mathcal S_l}|P(X^{(n)}_{x,S}\in S^{\prime})-P(X^{(n^{\prime})}_{x,S}\in S^{\prime})|\leq \epsilon,
		\end{equation}
		where $X^{(n)}_{x,S}$ denotes the the point from which the random walk exits $S$.
		
		By Lemmas \ref{exitlem1} and \ref{exitlem}, for a square $S^{\prime}\in \mathcal S_l$ and for all $x,y\in S^{\prime}$, with $\mathbb{P}$-probability at least $1-\exp(-100(l-k))$ we have
		\begin{equation}
			d_{\mathrm{TV}}(X^{(n)}_{x,S},X^{(n)}_{y,S})\leq\exp(-\alpha(l-k))
		\end{equation}
		for some absolute constant $\alpha$. Applying a union bound over $S^\prime\in \mathcal S_l$ and three shifts of $\mathcal S_l$ by $(0,2^{-l-1})$, $(2^{-l-1},0)$, and $(2^{-l-1},2^{-l-1})$, we see that with $\mathbb P$-probability at least $1-2^{2l+4}\exp(-100(l-k))$, for all $x,y$ such that $|x-y|<2^{-l}$,
		\begin{equation}\label{tight2}
			d_{\mathrm{TV}}(X^{(n)}_{x,S},X^{(n)}_{y,S})\leq\exp(-\alpha(l-k)).
		\end{equation}
		
		Now we construct a coupling between $X^{(n)}$ and $X^{(n^{\prime})}$ (both started from the origin) inductively and we denote by $Q$ their joint measure under the coupling. With $\mathbb{P}$-probability at least $1-\delta-2^{2l+4}\exp(-100(l-k))$, we see that \eqref{tight1}, \eqref{tight3} and \eqref{tight2} hold, which we assume in what follows. We will show by induction that we can couple the two random walks up to $\tau^n_s$ and $\tau^{n^\prime}_s$ respectively such that
		\begin{equation}\label{tight5}
			Q(E_s)\geq 1-s\epsilon, \mbox{ where } E_s=\left\{|X^{(n)}_{\tau^n_{j}}-X^{(n^{\prime})}_{\tau^{n^{\prime}}_{j}}|\leq 2^{-l},\mbox{ for all } j\leq s\right\}.
		\end{equation}
		Suppose \eqref{tight5} holds for $s-1$ in place of $s$ (i.e., suppose our induction hypothesis holds), and we will next show that \eqref{tight5} holds (for $s$). By \eqref{tight3} and \eqref{tight2} we can couple the random walk on $\mathbb Z_n^2$ between time $\tau^n_{s-1}$ and $\tau^n_{s}$ and the random walk on $\mathbb Z_{n^{\prime}}^2$ between time $\tau^{n^{\prime}}_{s-1}$ and $\tau^{n^{\prime}}_{s}$ such that, 
		\begin{equation}
			Q\left(|X^{(n)}_{\tau^n_{s}}-X^{(n^{\prime})}_{\tau^{n^{\prime}}_{s}}|>2^{-l};|X^{(n)}_{\tau^n_{j}}-X^{(n^{\prime})}_{\tau^{n^{\prime}}_{j}}|\leq 2^{-l},\mbox{ for all } j<s\right)\leq \epsilon.
		\end{equation}
		Combined with our induction hypothesis that \eqref{tight5} holds for $s-1$, it follows that \eqref{tight5} holds for $s$, completing the induction procedure. Combining \eqref{tight5} with \eqref{tight1}, we get the following: if we write $\beta^{(n)}$ as the trace of $X^{(n)}$, then
		\begin{equation}
			Q\left(d_{\mathrm{CMP}}(\beta^{(n)},\beta^{(n^{\prime})})\leq 3\cdot2^{-k}\right)\geq Q\left(\left\{\tau^n_t=\tau^n_{t-1}\right\}\cap \{\tau^{n^{\prime}}_t=\tau^{n^{\prime}}_{t-1}\}\cap E_t\right)\geq 1-{\delta}-t\epsilon.
		\end{equation}
		
		Now we set $\epsilon<\delta/t$ and set $l$ sufficiently large depending on $\delta$, so that $N$ depends only on $\delta$ (recall \eqref{tight3}). Summarizing the above discussions, we arrive at the following: for each $n,n^{\prime}>N$, with $\mathbb{P}$-probability at least $1-2\delta$, there exist a coupling $Q$ such that
		\begin{equation}
			Q\left(d_{\mathrm{CMP}}(\beta^{(n)},\beta^{(n^{\prime})})\leq 3\delta\right)\geq1-2\delta.
		\end{equation}
		So for each $n,n^{\prime}>N$, if we consider the Prokhorov distance with respect to the metric $d_{\mathrm{CMP}}$, we have
		\begin{equation}
			\mathbb{P}\left(d_{\mathrm{prok}}(X^{(n)},X^{(n^{\prime})})\leq 3\delta\right)\geq 1-2\delta.
		\end{equation}
		By the completeness of the measure space on curves (as shown in \cite{aizenman1998holder}), we know that $X^{(n)}$ converges in probability (note that here $n = n_m$ is already a subsequence). So there is a further subsequence that converges $\mathbb P$-a.s. with respect to $d_{\mathrm{prok}}$, and by Prokhorov's Theorem it converges weakly $\mathbb P$-a.s.
	\end{proof}
	
	\section{Tightness of the expected exit time}\label{sec8}
	In this section we give the proof of Theorem \ref{exptime} taking Theorem \ref{thm1} as a crucial input. The key is to estimate the expected exit time, for which the main ingredient is to compute the probability of hitting a point before exiting. In light of \eqref{return6}, Lemma \ref{returnlem3} will be crucial in establishing a lower bound, and the (easier) upper bound employs Proposition \ref{series}.
	
	\subsection{Tightness of the LQG measure}
	In this subsection we will prove that, for small $\gamma$ the ``LQG measure'' (defined in \eqref{lqg} below) normalized by its expectation is typically of order 1. In fact, much more beyond the tightness (including tail behaviors) has been understood for the LQG measure thanks to works including \cite{kahane1985chaos,robert2010gaussian,rhodes2014gaussian} (see also \cite{berestycki2024gaussian}). We include a proof merely for completeness since our field, defined on a rescaled lattice, is slightly different from the continuum case in the literature. Note that our computation follows that in the cited literature.
	
	For a subset $B\subset \mathbb Z_n^2$, we define the LQG measure of $B$ by
	\begin{equation}\label{lqg}
		\pi_n(B)=\sum_{y\in B} \pi_n(y),
	\end{equation}
	where $\pi_n(y) = \sum_{e: e\sim y} e^{-\gamma \phi_n(m_e)}$. For notation convenience, we define
	\begin{equation}
		\eta_n(B)=\sum_{y\in B} e^{\gamma\phi_{n}(y)}.
	\end{equation}A direct calculation shows that 
	\begin{equation}\label{lqgex}
		\mathbb{E}[\eta_n(B)]=2^{\frac{\gamma^2n}{2}}|B|.
	\end{equation}
	For brevity, we write $\tilde{\pi}_n(B)=\pi_n(B)/\mathbb{E}[\pi_n(B)]$ and $\tilde{\eta}_n(B)=\eta_n(B)/\mathbb{E}[\eta_n(B)]$. For small $\gamma$, we can give a uniform bound on the second moment of the LQG measure in any box $B$.	
	\begin{lem}\label{lqgpo}
		For $\gamma<\sqrt{2}$, $m\in \mathbb Z$ and for every box $B$ with side length $2^{m}$, there exists a constant $C_1=C_1(\gamma)$ such that
		\begin{equation}\label{lqgpoeq}
			\mathbb{E}[\tilde{\pi}_n(B)]^2\leq C_1(2^{-\gamma^2m}\vee 1).
		\end{equation}
	\end{lem}
	\begin{proof}
		A direct calculation shows that
		\begin{equation}
			\begin{aligned}\label{lqgne3}
				\mathbb{E}[\sum_{y\in B}\exp\left(\gamma\phi_{n}(y)\right)]^2
				&=\sum_{y,y^{\prime}\in B}\exp\left(\frac{\gamma^2}{2}\mathrm{Var}(\phi_{n}(y)+\phi_{n}(y^{\prime}))\right)\\
				&=\sum_{y,y^{\prime}\in B}\exp\left(\frac{\gamma^2}{2}(2n\log 2+2\mathbb{E}\phi_n(y)\phi_n(y^\prime))\right)\\
				&\preceq 2^{\gamma^2 n}\sum_{y,y^{\prime}\in B}\exp\left(-\gamma^2(\log|y-y^{\prime}|\vee 1)\right)\\
				&\preceq 2^{\gamma^2 n}\sum_{k\leq m+1} |B|^2(2^{-\gamma^2k}\vee 1)2^{-2(m-k)}\preceq (\mathbb{E}[\eta_n(B)])^2(2^{-\gamma^2m}\vee 1).
			\end{aligned}
		\end{equation}
		This implies that $\mathbb{E}[\tilde{\eta}(B)]^2\preceq 2^{-\gamma^2m}\vee 1$, with the implicit constant in $\preceq$ depending only on $\gamma$. Thus, by considering several shifts of the box $B$ (to address the slight difference between the left-hand sides of \eqref{lqgpoeq} and \eqref{lqgne3}) we may obtain \eqref{lqgpoeq}. 
	\end{proof}
	Next we control the negative moment.
	
	\begin{lem}\label{lqgne}
		For $\gamma<\sqrt{2}$, $m\geq 0$ and for $p$ sufficiently small, there exists a constant $C_2=C_2(\gamma,p)$ such that for every box $B$ with side length $2^{m}$,
		\begin{equation}\label{lqgneeq}
			\mathbb{E}[\tilde{\pi}_n(B)]^{-p}\leq C_2.
		\end{equation}
	\end{lem}
	
	In order to prove Lemma \ref{lqgne}, we need a version of Girsanov's lemma. In the following lemmas, we keep using the notation for $B$ in Lemma \ref{lqgne}. Let $\mathbb{P}_n$ denote the law of $\phi_{n}$ and we let $K_n(x,y)=\mathbb{E}[\phi_{n}(x)\phi_{n}(y)]$ be the covariance function of $\phi_{n}$.
	\begin{lem}\label{gir}
		Let $x$ be sampled independently and uniformly in $B$ and let $d\mathbb{Q}_n=\tilde{\eta}_n(B)d\mathbb{P}_n$. Under $\mathbb{Q}_n$, the law of $\phi_{n}$ is the same as the law of $\phi_{n}+\gamma K_n(x,\cdot)$ under $\mathbb{P}_n$.
	\end{lem}
	\begin{proof}
		Define $\mathbb{Q}_n^{z}$ by $d\mathbb{Q}_n^z=\frac{1}{Z}\exp(\gamma\phi_{n}(z))d\mathbb{P}_n$, where $Z=2^{\frac{\gamma^2n}{2}}$ is a normalizing constant. By Girsanov's lemma, the law of $\phi_n$ under $\mathbb Q^z_n$ is the same as the law of $\phi_{n}+\gamma K_n(z,\cdot)$ under $\mathbb{P}_n$. Since $\mathbb Q_n^z = \tilde \eta(\{z\}) d\mathbb P_n$, $\mathbb Q_n=\frac{1}{|B|}\sum_{z\in B}\mathbb Q_n^z$ and $x$ is uniformly and independently distributed on $B$, we complete the proof.
	\end{proof}
	
	\begin{lem}\label{lqglem}
		Let $x$ be sampled independently and uniformly in $B$. Consider the event
		\begin{equation}\label{Es}
			\mathcal{E}_s=\left\{\sum_{y\in \mathbb Z_n^2:|y-x|\leq s^{-M}}e^{\gamma^2 K_n(x,y)}e^{\gamma\phi_{n}(y)}\leq\frac{1}{s}\mathbb{E}[\eta_n(B)]\right\}.
		\end{equation}
		There exist $M=M(\gamma),s_0=s_0(\gamma)>0$ such that for all $s_0<s<2^{n/M}$, it holds that $\mathbb{P}(\mathcal{E}_s)\geq \frac{1}{2}$.
	\end{lem}
	\begin{proof}
		By \eqref{cov1}, for $|x-y|\leq s^{-M}$ and $\delta=2^{-n}$ we have $$K_n(x,y)\geq \int_{s^{-2M}/2}^{s^{-2M}/2\delta^2}\frac{1}{2 s}e^{-s}ds\geq  e^{-1}\int_{s^{-2M}/2}^{1/2}\frac{1}{2 s}ds=e^{-1}\log (s^M),$$ where we used $s^{-M}\geq \delta$. As a result, we have that
		\begin{equation}\label{lqgne1}
			\sum_{y\in \mathbb Z_n^2:|y-x|\leq s^{-M}}e^{\gamma^2 K_n(x,y)}e^{\gamma\phi_{n}(y)}\leq Cs^{-2}\sum_{y\in \mathbb Z_n^2:|y-x|\leq s^{-M}}e^{(\gamma^2+\frac{2e}{M}) K_n(x,y)}e^{\gamma\phi_{n}(y)}.
		\end{equation}
		By Lemma \ref{gir} with parameter $\beta=\gamma+\frac{2e}{M\gamma}$,
		\begin{equation}\label{lqgne2}
			\begin{aligned}
				\mathbb{E}\left[\sum_{y\in B}e^{(\gamma^2+\frac{2e}{M}) K_n(x,y)}e^{\gamma\phi_{n}(y)}\right]
				&=\mathbb{E}\left[\eta_n(B)\frac{\sum_{y\in B}\exp\left(\beta\phi_{n}(y)\right)}{Z}\right]\\
				&\leq\left(\mathbb{E}[\eta_n(B)^2]\right)^{\frac{1}{2}}\left(\mathbb{E}\left[\frac{\sum_{y\in B}\exp\left(\beta\phi_{n}(y)\right)}{Z}\right]^{2}\right)^{\frac{1}{2}},
			\end{aligned}
		\end{equation}
		where $Z=\mathbb{E}\sum_{y\in B}\exp\left(\beta\phi_{n}(y)\right)$ is the normalizing constant, and the inequality follows from Cauchy-Schwartz. For $M$ sufficiently large, depending only on $\gamma$, we have $\beta<\sqrt{2}$. Applying Lemma \ref{lqgpo} with $\beta$ in place of $\gamma$, we get that
		\begin{equation}
			\mathbb{E}\left[\sum_{y\in\mathbb Z_n^2:|y-x|\leq s^{-M}}e^{(\gamma^2+\frac{2e}{M}) K_n(x,y)}e^{\gamma\phi_{n}(y)}\right]\preceq\left(\mathbb{E}[\eta_n(B)^2]\right)^{\frac{1}{2}}= O(\mathbb E(\eta_n(B))),
		\end{equation}
		where the second transition also follows from Lemma \ref{lqgpo} (see \eqref{lqgne3}). Then by Lemma \ref{lqgpo} and Markov's inequality, there exists $s_0$ such that for all $s_0<s<2^{n/M}$, we have $\mathbb{P}(\mathcal{E}_s)\geq \frac{1}{2}$.		
	\end{proof}
	
	\begin{proof}[Proof of Lemma \ref{lqgne}]
		By Lemma \ref{gir}, if we let $x$ be sampled uniformly in $B$, it holds that
		\begin{equation}\label{lqgne4}
			\mathbb{E}\left[\exp\left(-s\frac{\sum_{y\in B}e^{\gamma^2K_n(x,y)}e^{\gamma\phi_{n}(y)}}{\mathbb{E}[\eta_n(B)]}\right)\right]=\mathbb{E}_{\mathbb{Q}_n}\left[\exp\left(-s\tilde{\eta}_n(B)\right)\right]=\mathbb{E}\left[\tilde{\eta}_n(B)\exp\left(-s\tilde{\eta}_n(B)\right)\right].
		\end{equation}
		Since $\sum_{y\in B} e^{\gamma^2 K_n(x, y)} e^{\gamma \phi_n(y)} \leq \sum_{y\in\mathbb Z_n^2:|y-x|\leq s^{-M}} e^{\gamma^2 K_n(x, y)} e^{\gamma \phi_n(y)}+ C s^{M \gamma^2} \eta_n(B)$ for some absolute constant $C\geq 1$, we get that on the event $\mathcal E_s$ (recall \eqref{Es})
		\begin{equation}
			s\sum_{y\in B}e^{\gamma^2 K_n(x,y)}e^{\gamma\phi_{n}(y)}\leq \mathbb{E}[\eta_n(B)]+Cs^{1+M\gamma^2}\eta_n(B).
		\end{equation}
		Then for $s_0<s<2^{n/M}$ we have
		\begin{equation}\label{lqgne5}
			\begin{aligned}
				\mathbb{E}\left[\exp\left(-s\frac{\sum_{y\in B}e^{\gamma^2K_n(x,y)}e^{\gamma\phi_{n}(y)}}{\mathbb{E}[\eta_n(B)]}\right)\right]
				&\geq
				\mathbb{E}\left[\exp\left(-s\frac{\sum_{y\in B}e^{\gamma^2K_n(x,y)}e^{\gamma\phi_{n}(y)}}{\mathbb{E}[\eta_n(B)]}\right)\mathds{1}_{\mathcal{E}_s}\right]\\
				&\succeq \mathbb{E}\left[\exp\left(-Cs^{1+M\gamma^2}\tilde{\eta}_n(B)\right)\mathds{1}_{\mathcal{E}_s}\right]\\
				&\succeq
				\mathbb{E}\left[\exp\left(-Cs^{1+M\gamma^2}\tilde{\eta}_n(B)\right)\right],
			\end{aligned}
		\end{equation}
		where the last inequality used the FKG inequality and Lemma \ref{lqglem} (note that $\mathcal E_s$ is a decreasing event). For $s\geq 2^{n/M}$ we have $$K_n(x,y)\leq (\mathbb{E}[\phi_n(x)^2]\mathbb{E}[\phi_n(y)^2])^{1/2}=n\log 2\leq \log (s^M).$$
		This implies that for $s \geq 2^{n/M}$
		\begin{equation}\label{lqgne6}
			\mathbb{E}\left[\exp\left(-s\frac{\sum_{y\in B}e^{\gamma^2K_n(x,y)}e^{\gamma\phi_{n}(y)}}{\mathbb{E}[\eta_n(B)]}\right)\right]\geq \mathbb{E}\left[\exp\left(-s^{1+M\gamma^2}\tilde \eta_n(B)\right)\right].
		\end{equation}
		Combining \eqref{lqgne5} and \eqref{lqgne6} we get that \eqref{lqgne5} in fact holds for all $s > s_0$ (recall $C\geq 1$).
		Combined with \eqref{lqgne4}, it yields that 
		$$\mathbb{E}\left[\exp\left(-Cs^{1+M\gamma^2}\tilde \eta_n(B)\right)\right]\preceq \mathbb{E}\left[\tilde{\eta}_n(B)\exp\left(-s\tilde{\eta}_n(B)\right)\right].$$
		In light of the preceding inequality, we can derive from the inequality $qe^{-\lambda q} \leq e^{-1}\lambda^{-1}$ (applied with $q = \tilde \eta_n(B)$ and $\lambda=s$) that
		\begin{equation}
			\mathbb{E}\left[\exp\left(-Cs^{1+M\gamma^2}\tilde{\eta}_n(B)\right)\right]\leq\frac{C}{s}
		\end{equation}
		holds for some absolute constant $C$ and for all $s>s_0$. Therefore, for $p<\frac{1}{1+ M\gamma^2}$, we get that
		\begin{equation}
			\mathbb{E}[\tilde{\eta}_n(B)]^{-p}\leq C\int_0^{\infty}t^{p-1}\mathbb{E}\exp(-t\tilde{\eta}_n(B))dt\preceq1+\int_{t_0}^{\infty}t^{p-1-\frac{1}{1+M\gamma^2}} dt\preceq 1,
		\end{equation}
		where $t_0$ is a constant depending on $s_0$ (thus only on $\gamma$). Thus, by considering several shifts of the box $B$ we may obtain \eqref{lqgneeq}. 
	\end{proof}
	
	\subsection{Tightness of the expected exit time}
	In this subsection we will prove the tightness of the expected exit time after normalization. Recall that $\tau_n$ denotes the hitting time of $\partial B(1)$ (that is, the exit time of $B(1)$) by the random walk $X^{(n)}$ on $\mathbb Z_n^2$. We will need the following notation.
	\begin{defi}
		For a network $\mathrm{G}$, a vertex set $V\subset V(\mathrm{G})$ and two vertices $x,y\in V(\mathrm{G})$, we define the Green's function of the random walk stopped upon exiting $V$ by
		\begin{equation}
			G_V(x,y)=E^x[\mbox{number of times that } X^x \mbox{ hits } y\mbox{ before exiting }V].
		\end{equation}
	\end{defi}
	By \eqref{return6} and \cite[Proposition~2.20]{Lyons_Peres_2017}, if we write $Z=V(\mathrm{G})\setminus V$, we have
	\begin{equation}\label{greencha}
		G_V(x,y)=R(x,Z)\pi(y)P^y(\tau_x<\tau_Z)=\frac{1}{2}\pi(y)(R(x,Z)+R(y,Z)-R(x,y)),
	\end{equation}
	where $\pi(y)$ is the sum of conductances over all edges incident to $y$. In particular, we write $G_n(x,y)=G_{B(1)}(x,y)$ with respect to the random walk $X^{(n)}$ on $\mathbb Z_n^2$. Note that $G_n(x,y)=0$ if either $x$ or $y$ is not in $B(1)$. In this subsection, $G_B$ is with respect to the network on $\mathbb Z_n^2$ and we denote $|B| = |B\cap \mathbb Z^2_n|$.
	
	\begin{lem}\label{green}
		For all $\gamma<\gamma_0$ the following holds: for all $q \in (0,1)$, $b>a>0$ and $m \geq 0$, there exists a constant $C_5=C_5(q,a,b)$ such that, for any square annulus $B$ with inner radius $a\cdot 2^m$ and with outer radius $b\cdot 2^m$,
		\begin{equation}
			\mathbb{P}\left(\sum_{y\in B}G_B(x,y)\leq C_5 2^{\frac{\gamma^2n}{2}}|B|,\mbox{ for all } x\in B\right)\geq 1-q.
		\end{equation}
		Furthermore, one can choose $C_5=o(q^{-0.9})$ as $q\to 0$.
	\end{lem}

	In order to prove Lemma \ref{green}, we first define some typical events. We fix an absolute constant $C_0>\log 100(\log 4+C^\prime)$, where $C^\prime$ is the absolute constant in Proposition \ref{max}. We write $\mathcal{R}_k(B)$ as the collection of all dyadic rectangles with dimensions $2^{m-k+1}\times2^{m-k}$ or $2^{m-k}\times2^{m-k+1}$ that intersect with $B$. For $S\in \mathcal{R}_k(B)$, we define $R^n(S, \mathrm{hard})$ to be the effective resistance between the shorter sides of $S$ in $\mathbb Z_n^2$. For $C_6>0$, we define
	\begin{equation}\label{green5}
		E_k^{n}(C_6,B)=\left\{R^n(S, \mathrm{hard})\leq C_6e^{C_0\gamma k}\exp(k^{2/3}), \mbox{ for all } S\in \mathcal{R}_k(B)\right\}.
	\end{equation}
	Also define $\mathcal D_k(B)$ to be the set of all dyadic boxes in $B$ with side length $2^{m-k}$, and for $C_7 > 0$ define
	\begin{equation}\label{green7}
		F^n_k(C_7,B)=\left\{\pi_n(B_k)\leq C_72^{\frac{3k}{2}}2^{\frac{\gamma^2n}{2}}|B_k|,\mbox{ for all } B_k\in \mathcal{D}_k(B)\right\}.
	\end{equation}
	Then we define 
	\begin{equation}\label{green10}
		\mathcal U_n(C_6,C_7,B)=\bigcap_{k\geq 1}(E^n_k(C_6,B)\cap F^n_k(C_7,B)).
	\end{equation}
	We prove in the following lemma that $\mathcal U_n(C_6,C_7,B)$ is a typical event for suitable choice of $C_6,C_7$.
	
	\begin{lem}\label{greenlem1}
		For each $m\geq 0$ and $q\in (0, 1)$ and for $a, b>0$ there exist constants (possibly depending on $a, b$) $C_6=O(q^{-0.2})$, $C_7=O(q^{-2/3})$ such that $\mathbb P(\mathcal U_n(C_6,C_7,B))\geq 1-q$ for all $n$ and for all square annulus $B$ with inner radius $a2^m$ and with outer radius $b2^m$. 
	\end{lem}
	\begin{proof}
		We first control the event  $E_k^n(C_6, B)$ by dividing $k$ into three cases. For $0<k\leq m$, by Proposition \ref{tail} and a union bound over $\mathcal{R}_k(B)$ we have 
		\begin{equation}\label{green1}
			\mathbb{P}((E^n_k(C_6,B))^c)\leq C2^{2k}\exp\left(-c\frac{(\log C_6+k^{2/3})^2}{\log(\log C_6+k^{2/3})}\right).
		\end{equation}
		For $m< k\leq m+n$, using the fact $\gamma<1$ we have
		\begin{equation}\label{green2}
			\begin{aligned}
				(E^n_k(C_6,B))^c&\subset \{\max_{x\in B}\phi_{k-m}(x)\geq C_0k+\frac{1}{2}\log C_6\} \\
				&\cup \{\exists S\in \mathcal{R}_k(B):R^{k-m,n}(S, \mathrm{hard})\leq C_6^{1/2}\exp(k^{2/3})\},
			\end{aligned}
		\end{equation}
		where $R^{k-m,n}(S, \mathrm{hard})$ denotes the effective resistance between the shorter sides of $S$ but with $\phi_{k-m,n}$ instead of $\phi_n$. Recall the absolute constant $C^\prime$ in Propositions \ref{max}. By a union bound, we get the following by noting that $C_0>10(\log 4+C^\prime)$:
		\begin{equation}\label{gre3}
			\begin{aligned}
				\mathbb{P}\left(\max_{x\in B}\phi_{k-m}(x)\geq C_0k+\frac{1}{2}\log C_6\right)
				&\leq (b+1)^22^{2m}\mathbb{P}\left(\max_{x\in B(1)}\phi_{k-m}(x)\geq C_0k+\frac{1}{2}\log C_6\right)\\
				&\leq (b+1)^2C2^{2m}4^{k-m}\exp \left(-\frac{(C_0k+\frac{1}{2}\log C_6)^2}{\log 4\cdot(\sqrt{k-m}+C^\prime)^2}\right)\\
				&\leq C(b+1)^2C_6^{-50}4^{-k},
			\end{aligned}
		\end{equation}
		where in the last inequality we used the fact that $(C_0k+\frac{1}{2}\log C_6)^2\geq (100C_0k+50\log C_6)C_0k/100\geq \log 4(100k+50\log C_6)(\sqrt{k-m}+C^\prime)^2.$
		In addition, by Lemma \ref{tail} and the scaling property of $\phi_n$, we have
		\begin{equation}\label{gre4}
			\mathbb{P}\left(R^{k-m,n}_{2,1}\leq C_6^{1/2}\exp(k^{2/3}) \right)\leq C\exp\left(-c\frac{(\log C_6^{1/2}+k^{2/3})^2}{\log(\log C_6^{1/2}+k^{2/3})}\right).
		\end{equation}
		Combining \eqref{green2}, \eqref{gre3}, \eqref{gre4} with a union bound over $\mathcal{R}_k(B)$ we get that
		\begin{equation}\label{green3}
			\mathbb{P}((E^n_k(C_6,B))^c)\leq C(b+1)^2C_6^{-50}4^{-k}+C(b+1)^2\exp\left(-c\frac{(\log C_6+k^{2/3})^2}{\log(\log C_6+k^{2/3})}\right).
		\end{equation}
		And for $k>m+n$, by Proposition \ref{max} we have similarly that
		\begin{equation}\label{green4}
			\begin{aligned}
				\mathbb{P}((E^n_k(C_6,B))^c)&\leq \mathbb{P}\left(\max_{x\in B}\phi_{n}(x)\geq C_0k+\log C_6\right)\\
				&\leq C2^{2m}(b+1)^2\mathbb{P}\left(\max_{x\in B(1)}\phi_{n}(x)\geq 10(n+m)+\log C_6\right)\\
				&\leq 
				C2^{2m}(b+1)^2 \exp \left(-\frac{(C_0k+\log C_6)^2}{\log 4\cdot (\sqrt{n}+C^\prime)^2}\right)\leq C(b+1)^2C_6^{-50}4^{-k}.
			\end{aligned}
		\end{equation}
		Next we control the probability of the event $F_k^n(C_7, B)$. By Lemma \ref{lqgpo}, we apply Markov's inequality and a union bound and we get that
		\begin{equation}\label{gre2}
			\begin{aligned}
				\mathbb{P}((F^n_k(C_7,B))^c)&\leq 2^{2k}(b+1)^2\frac{1}{C_7^22^{3k}}\mathbb{E}(\tilde{\pi}_n(B_k))^2\\
				&\leq C(b+1)^2C_7^{-2}2^{-k}\cdot(2^{-\gamma^2(m-k)}\vee 1)\leq C(b+1)^2C_7^{-2}2^{-(1-\gamma^2)k},
			\end{aligned}
		\end{equation}
		where we used the fact $m\geq 0$ in the last inequality. By taking $C_6,C_7$ sufficiently large (depending only on $q$) such that $C_6=O(q^{-0.2})$ and $C_7=O(q^{-2/3})$, we get from \eqref{green1}, \eqref{green3}, \eqref{green4} and \eqref{gre2} that
		$$
			\begin{aligned}
				&\mathbb{P}\left(\mathcal U_n(C_6,C_7,B)\right)\\
				&\geq 1-(b+1)^2\sum_{k\geq1}\left(C_6^{-50}4^{-k}+C2^{2k}\exp\left(-c\frac{(\log C_6+k^{2/3})^2}{\log(\log C_6+k^{2/3})}\right)+CC_7^{-2}2^{-(1-\gamma^2)k}\right) \geq 1-q,
			\end{aligned}
		$$
		which completes the proof.
	\end{proof}

	\begin{lem}\label{greenlem2}
		For all $b>a>0$, $m\in \mathbb Z$ and any square annulus $B$ with inner radius $a\cdot 2^m$ and with outer radius $b\cdot 2^m$, there exists $\alpha=\alpha(a,b)$ such that, on the event $\mathcal U_n(C_6,C_7,B)$ we have
		\begin{equation}
			\sum_{y\in B}G_{B}(x,y)\leq C\alpha \cdot C_6C_7 2^{\frac{\gamma^2n}{2}}|B|, \mbox{ for all } x\in B.
		\end{equation}
	\end{lem}

	\begin{proof}
		Write $Z=\partial B$. For all $x,y\in B$, by \eqref{greencha} we have
		\begin{equation}
			G_{B}(x,y)=R^n(x,Z)\pi_n(y)P^y(\tau_x<\tau_Z)=\frac{1}{2}\pi_n(y)(R^n(x,Z)+R^n(y,Z)-R^n(x,y)).
		\end{equation}
		For each $x,y\in B$, consider the smallest integer $t$ such that $\|x-y\|_{\infty}\geq 2^{m-t+1}$. Then we can find two disjoint dyadic boxes $B_x$ and $B_y$ with side length $2^{m-t+1}$ such that $x\in B_x$ and $y\in B_y$. Thus we have
		\begin{equation}\label{green12}
			R^n(x,y)\geq R^n(x,\partial B_x)+R^n(y,\partial B_y).
		\end{equation}
		Now, note that there exist dyadic rectangles $S^{k,i}_x\in \mathcal{R}_k(B)$ for $1\leq k \leq t$ and $1\leq i\leq \lfloor b\rfloor+4$, such that the union of (arbitrary) hard crossings through all these dyadic rectangles as well as path from $x$ to $\partial B_x$ contains an arbitrary path from $x$ to $Z$ (see Figure \ref{fig4} for an illustration). Similar definitions also apply to $y$. Thus, on the event $\mathcal U_n(C_6,C_7,B)$, for some absolute constant $C$ we have that
		\begin{equation}
				R^n(x,Z)-R^n(x,\partial B_x)\leq \sum_{1\leq k\leq t, 1\leq i\leq \lfloor b\rfloor+4}R^n(S_x^k, \mathrm{hard})
				\leq C(b+1)\cdot C_6e^{C_0\gamma t}\exp(t^{2/3}),
		\end{equation}
		where the first inequality follows from  Proposition \ref{series}, and the second inequality follows from \eqref{green5}. 
		Combined with the fact that (recall \eqref{green12})
		\begin{equation}
			R^n(x,Z)+R^n(y,Z)-R^n(x,y)\leq R^n(x,Z)-R^n(x,\partial B_x)+R^n(y,Z)-R^n(y,\partial B_y),
		\end{equation}
		it yields that (recall \ref{green7})
		\begin{equation}
			\begin{aligned}
				\sum_{y\in \mathbb Z_n^2:2^{m-t}\leq \|x-y\|_{\infty}\leq (b+1)2^{m-t+1}} G_B(x,y)
				&\leq
				C(b+1)^3C_6e^{C_0\gamma t}\exp(t^{2/3})C_7 2^{\frac{3t}{2}}2^{\frac{\gamma^2n}{2}}2^{2(m-t)}2^{2n}\zeta_n^2\\
				&\leq 
				C(b+1)^3C_6C_7 2^{\frac{\gamma^2n}{2}-(C_0\gamma-\frac{1}{4})t}2^{2m+2n}\zeta_n^2,
			\end{aligned}
		\end{equation}
		where in the last inequality we used the fact that $\exp(t^{2/3})\leq C2^{\frac{t}{4}}$.
		So we can deduce that on the event $\mathcal U_n(C_6,C_7,B)$, if $\gamma<\gamma_0 <(4C_0)^{-1}$ we have that
		\begin{equation}
			\sum_{y\in B}G_B(x,y)\leq C(b+1)^3\cdot C_6C_7\sum_{t\geq1} 2^{\frac{\gamma^2n}{2}-(C_0\gamma-\frac{1}{4})t}2^{2m+2n}\zeta_n^2\leq C(b+1)^3\cdot C_6C_72^{\frac{\gamma^2n}{2}}2^{2m+2n}\zeta_n^2
		\end{equation}
		for some absolute constant $C$. Since $|B| = (b^2-a^2)2^{2m+2n} \zeta_n^2$, this then completes the proof.
	\end{proof}

	\begin{proof}[Proof of Lemma \ref{green}]
		We take $C_6,C_7$ from Lemma \ref{greenlem1}. Noting that $C_6C_7=o(q^{-0.9})$, we complete the proof by Lemma \ref{greenlem2}.
	\end{proof}
	\begin{figure}
		\centering
		\includegraphics[width=60mm]{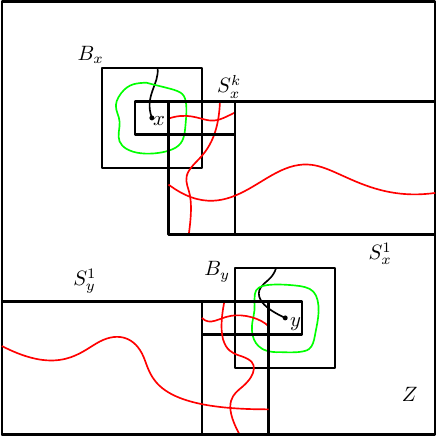}
		\caption{Connection of paths}
		\label{fig4}
	\end{figure}
	
	Lemma \ref{green} gives an upper bound for the exit time of a box. Now we use Lemma \ref{returnlem3} to get a lower bound on the exit time.
	
	\begin{lem}\label{greenlem3}
		Fix $0<a<d<b$. For $\gamma<\gamma_0$ and for all $m\geq 0$ and $q\in (0,1)$, there exists $\beta = \beta(q, a, b, d)$, such that the following holds for all dyadic annulus $B$ with inner radius $a 2^m$ and with outer radius $b 2^m$. If we let $\mathcal C\subset B$ be the boundary of the square concentric to $B$ with radius $d 2^m$, then we have
		\begin{equation}
			\mathbb{P}\left(\sum_{y\in B} G_B(x,y)\geq \beta2^{\frac{\gamma^2n}{2}}|B|, \mbox{ for all } x\in \mathcal C\right)\geq 1-q.
		\end{equation}
	\end{lem}
	Note that $\mathcal C$ can be covered by dyadic boxes $Y_1,\ldots, Y_r$ with centers $v_1,\ldots,v_r$ and side length $\lambda 2^m$, where $\lambda,r$ are constants depending only on $a,b,d$. For each box $Y_i$, let $B_1(Y_i),B_2(Y_i),B_3(Y_i),B_4(Y_i)$ be boxes with the same center, such that $B_1(Y_i)=Y_i$ and the side length of $B_{j+1}(Y_i)$ is 4 time that of $B_j(Y_i)$ for $j=1, 2, 3$. By taking sufficiently small $\lambda$ we may assume $B_4(Y_i)\subset B$. We consider the event 
	\begin{equation}\label{green11}
		\begin{aligned}
			&\mathcal V_n(C_8,C_9,B)=\left\{R^n(\mbox{around }B_3(Y_i)\setminus B_2(Y_i))\tilde \pi_n (B_2(Y_i)\setminus B_1(Y_i))\geq C_8, \mbox{ for all } 1\leq i \leq r\right\}\\
			&\cap \left\{R^n(\mbox{around } B_3(Y_i)\setminus B_2(Y_i))\leq C_9 R^n(\mbox{across }B_4(Y_i)\setminus B_3(Y_i)),\mbox{ for all }1\leq i\leq r\right\}.
		\end{aligned}
	\end{equation}

	\begin{lem}\label{greenlem4}
		For all $b>a>0$, $m\in \mathbb Z$, and for any square annulus $B$ with inner radius $a\cdot 2^m$ and with outer radius $b\cdot 2^m$, there exists $c_1 = c_1(C_9)$ and $\xi=\xi(a,b,d)$ such that on the event $\mathcal V_n(C_8,C_9,B)$ we have
		\begin{equation}
			\sum_{y\in B} G_B(x,y)\geq \xi c_1C_82^{\frac{\gamma^2n}{2}}|B|, \mbox{ for all } x\in \mathcal C.
		\end{equation}
	\end{lem}
	\begin{proof}	
		In what follows we assume that the event $\mathcal V_n(C_8, C_9, B)$ holds. By Lemma \ref{returnlem3} (with $r_1=0$ so that $x$ is the boundary of a box), there exists $c_1=c_1(C_9)$ such that for all $y\in B_2(Y_i)\setminus B_1(Y_i)$ and $x\in Y_i$,
		$$R^n(y,\partial B_4(Y_i))+R^n(x,\partial B_4(Y_i))-R^n(x,y)\geq c_1R^n(\mbox{around } B_3(Y_i)\setminus B_2(Y_i)).$$
		Thus by \eqref{greencha}, since the Green's function $G_V$ is increasing with respect to $V$, for $x\in Y_i$ we have
		\begin{equation}
			\begin{aligned}
				\sum_{y\in B}G_B(x,y)&\geq \sum_{y\in B_2(Y_i)\setminus B_1(Y_i)} G_{B_4(Y_i)}(x,y)\\
				&\geq c_1\sum_{y\in B_2(Y_i)\setminus B_1(Y_i)}\pi_n(y) R^n(\mbox{around } B_3(Y_i)\setminus B_2(Y_i))\geq \xi c_1C_82^{\frac{\gamma^2n}{2}}|B|,
			\end{aligned}
		\end{equation}
		where the last inequality follows from the first requirement in \eqref{green11}, for some $\xi=\xi(a,b,d)$. This completes the proof.
	\end{proof}

	\begin{proof}[Proof of Lemma \ref{greenlem3}]
		By Lemma \ref{lqgne} and Markov's inequality, for some $\delta=\delta(q)$ we have
		\begin{equation}
			\mathbb P(\tilde \pi_n(B_2(Y_i)\setminus B_1(Y_i))\geq \delta \mbox{ for all }1\leq i\leq r)\geq 1-q/2.
		\end{equation} 
		Thus by Theorem \ref{thm1}, for some $C_8,C_9$ depending only on $q,a,b,d$, we have
		\begin{equation}\label{green13}
			\mathbb P\left(\mathcal V_n(C_8,C_9,B)\right)\geq 1-q.
		\end{equation}
		Then by Lemma \ref{greenlem4} we complete the proof.
	\end{proof}
	
	We may directly deduce Theorem \ref{exptime} from Lemmas \ref{green} and \ref{greenlem3}.
	\begin{proof}[Proof of Theorem \ref{exptime}]
		Recall $\chi_n =2^{(2+\gamma^2/2)n}\zeta_n^2$ and
		\begin{equation}
			E^0\tau_n=\sum_{y\in B(1)}G_n(0,y).
		\end{equation}
		By taking $a=0,b=1,m=0$ in Lemma \ref{green}, for all $q\in (0,1)$ we get that 
		\begin{equation}\label{green8}
			\mathbb{P}(E^0\tau_n\leq C_q2^{\frac{\gamma^2n}{2}}|B(1)|)\geq 1-q/2.
		\end{equation}
		Let $B$ be the square annulus centered at $(0,1/3)$ with inner radius $1/4$ and outer radius $1/2$, and let $\mathcal C$ be the boundary of the square centered at $(0,1/3)$ with radius $1/3$ (so we see that $B\subset B(1)$ and $0\in \mathcal C$).  We then apply Lemma \ref{greenlem3} with such choice of $B$ and $\mathcal C$, and we get that for all $q\in (0, 1)$ there exists $\beta = \beta (q)$ such that
		\begin{equation}\label{green9}
			\mathbb{P}(E^0\tau_n\geq \beta2^{\frac{\gamma^2n}{2}}|B(1)|)\geq 1-q/2,
		\end{equation}
		where we used the fact that $G_n(x,y)\geq G_B(x,y)$ since $B\subset B(1)$. Combining \eqref{green8} and \eqref{green9}, we can complete the proof.
	\end{proof}
	
	Next we prove a bound on the second moment for the exit time for later use.
	\begin{lem}\label{exitsec}
		Let $\mathrm{G}$ be a graph and $B\subset V(\mathrm{G})$, and let $\tau_B$ be first time for the random walk started at $x$ to exit $B$. Then
		\begin{equation}
			E\tau_B^2\leq \sum_{y,z\in B}G_{B}(x,y)G_B(y,z).
		\end{equation}
	\end{lem}
	\begin{proof}
		Let $X$ be a random walk started from $x$. Notice that
		\begin{equation}
			\tau_B=\sum_{y\in B}\sum_{i=0}^{\infty}\mathds{1}_{\{i<\tau_n,X_i=y\}}.
		\end{equation}
		So it holds that
		\begin{equation}
			\begin{aligned}
				E\tau_B^2
				&\leq 2\sum_{y,z\in B}\sum_{i\leq j}P^x(X_i=y,X_j=z,i\leq j\leq \tau_B)\\
				&= 2\sum_{y,z\in B}\sum_{i\leq j}P^x(X_i=y,i\leq \tau_B)P^{y}(X_{j-i}=z,j-i\leq \tau_B)\\
				&=2\sum_{y\in B}G_B(x,y)\sum_{z\in B}G_B(y,z),
			\end{aligned}
		\end{equation}
		completing the proof of the lemma.
	\end{proof}
	
	\section{Tightness with time parameterization}\label{sec9}
	In this section, we will prove Theorem \ref{thm3}, the tightness of the random walk with suitable time parameterization. In Section \ref{sec8}, it was proved that the expected exit time has the order $\chi_n=2^{(2+\gamma^2/2)n}\zeta_n^2$. Thus, when considering the scaling limit it is natural to scale the time by a factor of $\chi_n$. Afterwards, by the Arzela-Ascoli Theorem it suffices to prove the equicontinuity of paths and laws.
	
	In \textbf{Section \ref{subsec91}}, we first prove in Lemma \ref{equilem1} that with high probability, the exit time from a small box typically has the order of its expectation. By Lemmas \ref{green}, \ref{greenlem3} and the Paley-Zygmund inequality, we see that with positive probability the exit time is lower-bounded by its expectation up to constant. The key point in Lemma \ref{equilem1} is to enhance this to an estimate with high probability, via a multi-scale analysis. In addition, we bound the exit time from a large box in Lemma \ref{equilem4}, which then conveniently allows us to consider the random walk stopped upon exiting a large box. Based on Lemmas \ref{equilem1} and \ref{equilem4}, we prove equicontinuity of paths in Lemma \ref{equilem2}.
	
	In \textbf{Section \ref{subsec92}}, we couple two random walks with starting points close to each other using a similar method as in Section \ref{sec7} (see Section \ref{subsec73}). This is incorporated in Lemma \ref{equilem3}, yielding the equicontinuity of laws for random walks with respect to starting points.
	
	Recall that for $x\in \mathbb Z_n^2$, $X^{(n,x)}$ denotes the random walk on $\mathbb Z_n^2$ started from $x$. Also, recall that $\hat{X}^{(n,x)}:[0,\infty)\to \mathbb{R}^2$ is the piecewise linear interpolation of the process $t\to X^{(n,x)}_{\chi_nt}$.
	
	We will prove the tightness of the random function
	$P^{(n)}$
	from $\mathbb{R}^2$ to the space of probability measures on continuous curves $[0,\infty)\to \mathbb{R}^2$. Before proceeding to the proof, we recall a few metrics from Remark \ref{rmk2}, and we point out that these are the same metrics employed in \cite{berestycki2022random}. We emphasize that in this section our metric on the curve space does depend on the time parameterization.
	
	\subsection{Equicontinuity of paths}\label{subsec91}
	In what follows, we often drop the superscripts $(x, n)$ from $P$ for notation convenience when it is clear from the context.
	\begin{lem}\label{equilem1}
		For $k\geq 0$, let $S$ and $\hat{S}$ be dyadic squares with the same center and with side lengths $2^{-k}$ and $3\cdot2^{-k}$, respectively. For all $\epsilon>0$, there exists $\delta = \delta(\epsilon) > 0$ such that the following holds for all sufficiently large $n$: with $\mathbb P$-probability at least $1-\epsilon$ we have
		\begin{equation}\label{equi12}
			P^{(n,x)}(X^{(n,x)} \mbox{ exits }\hat{S} \mbox{ before time }\delta 2^{-2k}\chi_n)<\epsilon\mbox{ for all } x\in S.
		\end{equation}
	\end{lem}

	We first prove a lemma for the random environment. Let $m\geq 0$ be an integer to be determined. Let $T\subset S$ be a dyadic box with side length $2^{-m-k}$. For $k+3\leq i\leq k+m$, define $\mathcal{A}^i_T$ to be the square annulus concentric with $T$ which has inner radius $2^{-i+7/3}$ and outer radius $2^{-i+8/3}$, define $\mathcal{B}^i_T$ to be the concentric square annulus with inner radius $2^{-i+2}$ and outer radius $2^{-i+3}$, and define $\mathcal{C}^i_T$ to be the boundary of the concentric square with radius $2^{-i+5/2}$. (See Figure \ref{fig12} for an illustration.)
	
	\begin{figure}
		\begin{minipage}[c]{0.48\textwidth}
			\centering
			\includegraphics[width=60mm]{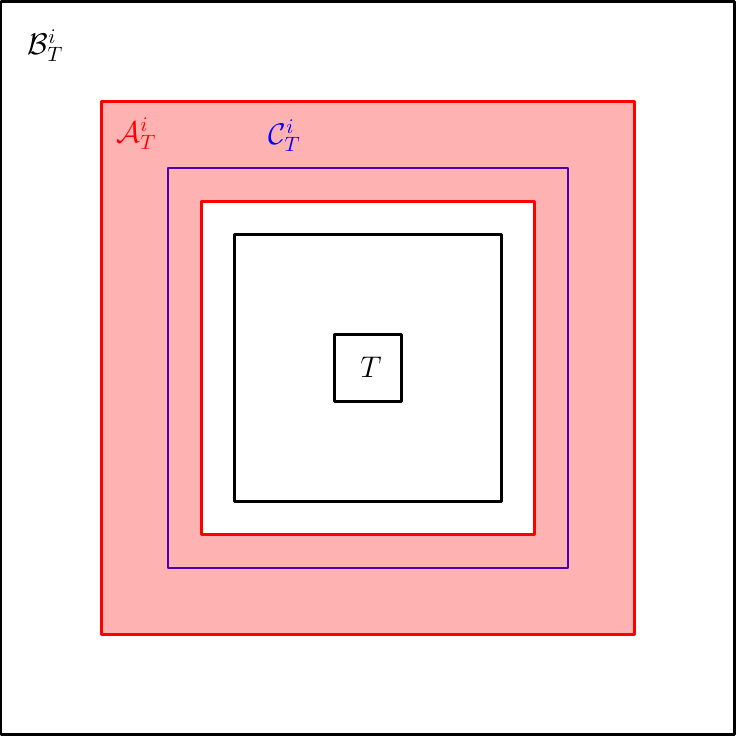}
			\caption{Illustration of annuli.}
			\label{fig12}
		\end{minipage}
		\begin{minipage}[c]{0.48\textwidth}
			\centering
			\includegraphics[width=60mm]{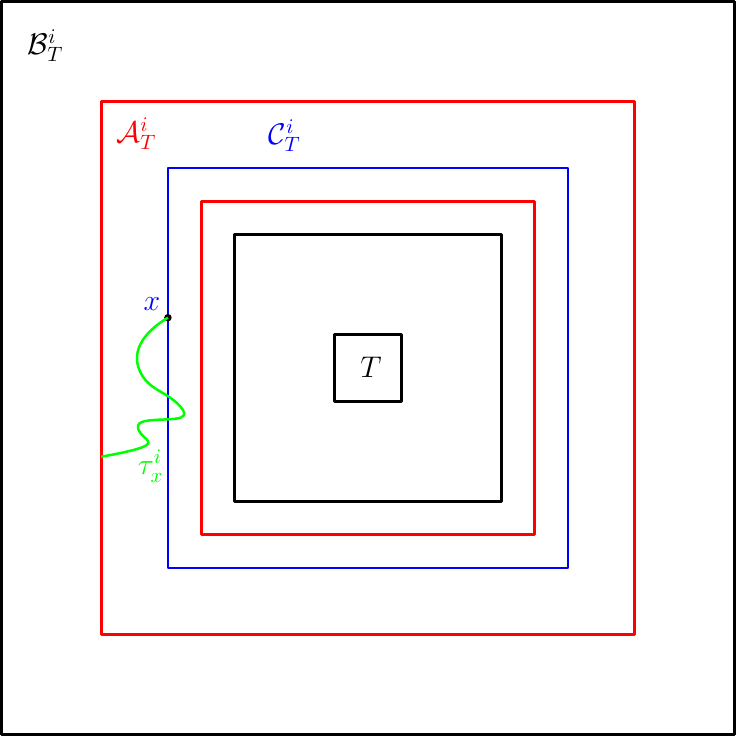}
			\caption{Illustration of a random walk.}
			\label{fig13}
		\end{minipage}
	\end{figure}
	
	\begin{lem}\label{equilem5}
		For all $\mathsf C ,\mathsf c >0$, there exist constants $C_1=C_1(\mathsf C ,\mathsf c),c_1=c_1(\mathsf C ,\mathsf c)$ such that for all $\gamma<\gamma_0$, if we define $I_T$ to be the set of $i$ with $k+3\leq i\leq k+m$ such that
		\begin{equation}\label{equi29}
			\sum_{y\in \mathcal{A}^i_T}G_n(x,y)\geq c_12^{-2i}\chi_n \mbox{ for all } x\in \mathcal C^i_T \And \sum_{y\in \mathcal{A}^i_T}G_n(x,y)\leq C_12^{-2i}\chi_n \mbox{ for all } x\in \mathcal A^i_T,
		\end{equation}
		then we have
		\begin{equation}\label{equi15}
			\mathbb{P}(|I_T|\geq \mathsf c  m)\geq 1-\exp(-\mathsf C m).
		\end{equation}
	\end{lem}

	\begin{proof}
		We consider the decomposition of $\phi_{n}$, similar to the decomposition in subsection \ref{subsec63}, as follows:
		\begin{equation}
			\phi_{n}=\phi_{0,i}+\phi_{i,n}^{\mathrm{f}}+\phi_{i,n}^{\mathrm{c}},
		\end{equation}
		where 
		\begin{equation}
			\phi_{i,n}^{\mathrm f}(x)=\int_{2^{-2n}}^{2^{-2k}}\int_{\mathcal{B}^i_T}p_{\frac{t}{2}}(x,y)W(dy,dt) \And\phi_{i,n}^{\mathrm c}(x)=\int_{2^{-2n}}^{2^{-2k}}\int_{\mathbb{R}^2\setminus\mathcal{B}^i_T}p_{\frac{t}{2}}(x,y)W(dy,dt).
		\end{equation}
		By analogues of Lemmas \ref{n2l} and \ref{n3l} (where we consider the field $\phi$ instead of $\psi$, and the proofs can be adapted verbatim), for each pair $(\mathsf C ,\mathsf c)$ there exists $C_2= C_2(\mathsf C ,\mathsf c)$ such that the following holds. If we define $N_T$ to be the number of $i$ with $k+3 \leq i\leq k+m$ such that
		\begin{equation}\label{equi4}
			\max_{u_i,v_i\in\mathcal{B}^i_T}\left(\phi_{0,i}(u_i)-\phi_{0,i}(v_i)\right)< C_2 \And \max_{u_i\in\mathcal{A}^i_T}|\phi^{\mathrm c}_{i,n}(u_i)|< C_2,
		\end{equation}
		then we have
		\begin{equation}\label{equi13}
			\mathbb{P}(N_T\geq \mathsf c  m)\geq 1-\exp(-\mathsf C m).
		\end{equation}
		We define $\mathcal U_{i,n}(C_3,C_4,B)$ and $\mathcal U_{i,n}^{\mathrm f}(C_3,C_4,B)$ as $\mathcal U_n(C_3, C_4, B)$ (recall \eqref{green10}) but with $\phi_{i,n}$ and $\phi_{i,n}^\mathrm{f}$ in place of $\phi_{n}$, respectively; we define $\mathcal V_{i,n}(C_5,C_6,B)$ and $\mathcal V_{i,n}^{\mathrm f}(C_5,C_6,B)$ as $\mathcal V_n(C_5, C_6, B)$ (recall \eqref{green11}) but with $\phi_{i,n}$ and $\phi_{i,n}^\mathrm{f}$ in place of $\phi_{n}$, respectively. By the scaling invariance of $\phi$, we know that $\phi_{i,n}$ on $\mathcal{A}^i_T$ has the same law as $\phi_{n-i}$ on $2^i\mathcal{A}^i_T$. Thus by Lemma \ref{greenlem1}, for all $q\in (0,1)$ there exist $C_7=C_7(q)$ and $C_8=C_8(q)$ such that
		\begin{equation}\label{equi27}
			\mathbb{P}(\mathcal U_{i,n}(C_7,C_8,\mathcal{A}^i_T))\geq 1-q/4.
		\end{equation}
		Similarly, since $\phi_{i,n}$ on $\mathcal{A}^i_T$ has the same law as $\phi_{n-i}$ on $2^i\mathcal{A}^i_T$, by \eqref{green13}, for all $q\in (0,1)$ there exist $C_9=C_9(q)$ and $C_{10}=C_{10}(q)$ such that
		\begin{equation}\label{equi28}
			\mathbb{P}(\mathcal V_{i,n}(C_9, C_{10},\mathcal{A}^i_T))\geq 1-q/4.
		\end{equation}
		Combining \eqref{equi27} and \eqref{equi28} we get that
		\begin{equation}
			\mathbb{P}(\mathcal U_{i,n}(C_7,C_8,\mathcal{A}^i_T)\cap \mathcal V_{i,n}(C_9,C_{10},\mathcal{A}^i_T))\geq 1-q/2.
		\end{equation}
		Thus by \eqref{equi4} and Markov's inequality, we can apply Lemma \ref{comp} to get that for some $C_3,C_4,C_5,C_6$ depending only on $q$,
		\begin{equation}
			\mathbb{P}(\mathcal U_{i,n}^{\mathrm f} (C_3,C_4,\mathcal{A}^i_T)\cap \mathcal V_{i,n}^{\mathrm f} (C_5,C_6,\mathcal{A}^i_T))\geq 1-q.
		\end{equation}
		We define $I^{\mathrm f}_T$ to be the set of $i$ with $k+3 \leq i\leq k+m$ such that
		\begin{equation}\label{equi11}
			 \mathcal U_{i,n}^{\mathrm f} (C_3,C_4,\mathcal{A}^i_T)\cap \mathcal V_{i,n}^{\mathrm f} (C_5,C_6,\mathcal{A}^i_T) \mbox{ occurs}.
		\end{equation}
		Note that $|I^{\mathrm f}_T|$ is a sum of independent variables by the independence of the field $\{\phi^{\mathrm{f}}_{i, n}
		(v): v\in \mathcal A^i_T\}$ over $i$. Thus, we can apply the Hoeffding's inequality and obtain the following: for $q$ sufficiently close to $1$ depending only on $(\mathsf C, \mathsf c)$, we have
		\begin{equation}\label{equi14}
			\mathbb{P}(|I^{\mathrm f}_T|\geq \mathsf c m)\geq 1-\exp(-\mathsf C m).
		\end{equation}		
		Now for all $i$ with $k+3 \leq i\leq k+m$ such that \eqref{equi4} and \eqref{equi11} hold, we must have $$\mathcal U_{i,n}(e^{\gamma C_2}C_3,e^{\gamma C_2}C_4,\mathcal{A}^i_T)\cap \mathcal V_{i,n} (e^{-\gamma C_2}C_5,e^{\gamma C_2}C_6,\mathcal{A}^i_T) \mbox{ occurs}.$$ In addition, if we take an arbitrary point $u_i\in \mathcal{A}^i_T$, by $\eqref{equi4}$, \eqref{green10} and \eqref{green11} we must have $$\mathcal U_{n}(e^{\gamma \phi_i(u_i)+\gamma C_2}e^{\gamma C_2}C_3,e^{-\gamma \phi_i(u_i)+\gamma C_2}e^{\gamma C_2}C_4,\mathcal{A}^i_T)\cap \mathcal V_{n} (e^{-2\gamma C_2}C_5,e^{2\gamma C_2}C_6,\mathcal{A}^i_T) \mbox{ occurs}.$$
		On this event, by Lemmas \ref{greenlem2} and \ref{greenlem4} we see that \eqref{equi29} holds with $C_1=100C^\prime e^{4\gamma C_2}C_3C_4$ and $c_1=0.01c_2e^{-2\gamma C_2}C_5$, where $C^\prime$ is an absolute constant we take from Lemma \ref{greenlem2} and $c_2=c_2(e^{2\gamma C_2}C_6)$ (note that the factor of $e^{\gamma \phi_i(u_i)}$ and $e^{-\gamma \phi_i(u_i)}$ cancel when computing the Green's function).
		Thus we can complete the proof by \eqref{equi13} and \eqref{equi14}.
	\end{proof}
	
	\begin{proof}[Proof of Lemma \ref{equilem1}]
		We take constants $C_1,c_1$ from Lemma \ref{equilem5}. Taking $\mathsf C = 3$, $\mathsf c = 2/3$ and $m$ sufficiently large, we see that with $\mathbb P$-probability at least $1-\epsilon$, $|I_T|\geq m/3$ holds for all dyadic boxes $T\subset S$. On this event, for all $x\in S$, let $T\subset S$ be the dyadic box with side length $2^{-m-k}$ that contains $x$. We write $\tau^i_x$ as the exit time of $\mathcal{A}^i_T$ by the random walk started at $x$. (See Figure \ref{fig13} for an illustration.) Then by Lemma \ref{exitsec}, for each $i\in I_T$ and $x\in \mathcal{C}^i$, we have $E\tau_x^i\geq c_1 2^{-2i-2}\chi_n$ and $E(\tau^i_x)^2\leq 2C_1^22^{-2i+2}\chi_n^2$. By Paley-Zygmund's inequality, we get that for $i\in I_T$
		\begin{equation}
			P(\tau_x^i\geq c_1 2^{-2i-3}\chi_n)\geq 2^{-10}c_1^2/C_1^2,\mbox{ for all } x\in \mathcal{C}^i_T.
		\end{equation}
		Thus by the strong Markov property,
		\begin{equation}
			P(X^{(n,x)} \mbox{ exits }\hat{S} \mbox{ before time }c_1 2^{-3}2^{-2(k+m)} \chi_n)<(1-2^{-10}c_1^2/C_1^2)^{m/3},\mbox{ for all } x\in S.
		\end{equation}
		This then completes the proof of the lemma by taking $m$ sufficiently large.
	\end{proof}
	
	The following lemma, as well as its proof, is highly similar to that of Lemma \ref{equilem1}.
	\begin{lem}\label{equilem4}
		For each $L>1$, there exists a sufficiently large $l$ depending only on $L$ such that the following holds for all $n \geq 1$. With $\mathbb P$-probability at least $1-1/3L$ we have 
		\begin{equation}
			P(X^{(n,x)} \mbox{ exits }B(3\cdot 2^l) \mbox{ before time }L\chi_n)<1/3L,\mbox{ for all } x\in B(2^l).
		\end{equation}
	\end{lem}
	We fix $m>0$ whose value will be determined later. Let $Y$ be a dyadic box with side length $1$. For $i\in \mathbb Z$ with $1\leq i\leq m$, define $\mathcal{A}^i_Y$ to be the square annulus concentric with $Y$ which has inner radius $2\cdot 8^i$ and outer radius $4\cdot 8^{i}$, define $\mathcal{B}^i_Y$ to be the square annulus concentric with $Y$ which has inner radius $8^i$ and outer radius $8^{i+1}$, and define $\mathcal{C}^i_Y$ to be the boundary of the concentric square with radius $3\cdot 8^i$. In addition, we define the field
	\begin{equation}
		\phi^i(x)=\int_{2^{-2n}}^{1}\int_{\mathcal{B}^i_Y}p_{\frac{t}{2}}(x,y)W(dy,dt) \And \phi_{\mathrm c}=\phi_{n}-\phi^i.
	\end{equation}
	Just as in Lemma \ref{equilem1}, we will apply a multi-scale analysis argument. The following technical lemma gives the ``almost'' independence of the field $\phi_n$ in $\mathcal{A}^i_Y$.
	
	\begin{lem}\label{equilem6}
		For all box $Y$ with side length 1 and for all $\mathsf C,\mathsf c>0$, there exists a constant $C_{11} = C_{11}(\mathsf C, \mathsf c)$ such that, if we write $N_Y$ to be the the number of $i$ with $1\leq i\leq m$ such that
		\begin{equation}\label{equi17}
			\max_{u_i\in \mathcal{A}^i_Y}|\phi_{\mathrm c}(u_i)|<C_{11},
		\end{equation}
		then we have
		\begin{equation}\label{equi19}
			\mathbb P(N_Y\geq \mathsf c m)\geq \exp(-\mathsf Cm).
		\end{equation}
	\end{lem}

	\begin{proof}
		This lemma follows directly by Lemma \ref{coarse}.
	\end{proof}

	\begin{lem}\label{equilem7}
		For all box $Y$ with side length 1 and for all $\mathsf C,\mathsf c>0$, there exist constants $C_{12},c_3$ depending only on $(\mathsf C,\mathsf c)$ such that the following holds: for all $\gamma<\gamma_0$, if we define $I_Y$ to be the set of $i$ with $1\leq i\leq m$ such that
		\begin{equation}\label{equi32}
			\sum_{y\in \mathcal{A}^i_Y}G_n(x,y)\geq c_32^{2i}\chi_n \mbox{ for all } x\in \mathcal C^i_Y \And \sum_{y\in \mathcal{A}^i_Y}G_n(x,y)\leq C_{12}2^{2i}\chi_n \mbox{ for all } x\in \mathcal A^i_Y,
		\end{equation}
		then we have
		\begin{equation}\label{equi21}
			\mathbb{P}(|I_T|\geq \mathsf c m)\geq 1-\exp(-\mathsf C m).
		\end{equation}
	\end{lem}
	\begin{proof}
		We take $C_{11}$ from Lemma \ref{equilem6}, depending only on $(\mathsf C,\mathsf c)$, such that \eqref{equi19} holds. By Lemma \ref{greenlem1}, for all $q\in (0,1)$ there exist $C_{13}=C_{13}(q)$ and $C_{14}=C_{14}(q)$ such that
		\begin{equation}\label{equi30}
			\mathbb{P}(\mathcal U_n(C_{13},C_{14},\mathcal{A}^i_Y))\geq 1-q/4.
		\end{equation}
		By Lemma \ref{lqgne} and Theorem \ref{thm1}, for all $q\in (0,1)$ there exists $C_{15}=C_{15}(q),C_{16}=C_{16}(q)$ such that
		\begin{equation}\label{equi31}
			\mathbb{P}(\mathcal V_n(C_{15},C_{16},\mathcal{A}^i_Y))\geq 1-q/4.
		\end{equation}
		Combining \eqref{equi30} and \eqref{equi31} we get that
		\begin{equation}
			\mathbb{P}(\mathcal U_n(C_{13},C_{14},\mathcal{A}^i_Y)\cap \mathcal V_n(C_{15},C_{16},\mathcal{A}^i_Y))\geq 1-q/2.
		\end{equation}
		 We define events $\mathcal U^i(C_{17},C_{18},B)$ and $\mathcal V^i(C_{19},C_{20},B)$ as $\mathcal U_n(C_{17},C_{18},B)$ and $\mathcal V_n(C_{19},C_{20},B)$ but with $\phi^i$ in place of $\phi_{n}$, respectively. Thus by Lemma \ref{comp} and Markov's inequality, we can deduce that for some $C_{17},C_{18},C_{19},C_{20}$ depending only on $q$ we have
		\begin{equation}
			\mathbb{P}(\mathcal U^i(C_{17},C_{18},\mathcal{A}^i_Y)\cap \mathcal V^i(C_{19},C_{20},\mathcal{A}^i_Y))\geq 1-q.
		\end{equation}
		We define $I_Y^{\prime}$ to be the set of $i$ with $1 \leq i\leq m$ such that
		\begin{equation}\label{equi18}
			\mathcal U^i (C_{17},C_{18},\mathcal{A}^i_Y)\cap \mathcal V^i (C_{19},C_{20},\mathcal{A}^i_Y) \mbox{ occurs}.
		\end{equation}
		Note that $|I_Y^{\prime}|$ is a sum of independent variables by the independence of the field $\{\phi^i
		(v): v\in \mathcal A^i_Y\}$ over $i$. Thus, we can apply the Hoeffding's inequality and obtain the following: for $q$ sufficiently close to $1$ depending only on $(\mathsf C, \mathsf c)$, we have
		\begin{equation}\label{equi20}
			\mathbb{P}(|I_Y^{\prime}|\geq \mathsf c m)\geq 1-\exp(-\mathsf C m).
		\end{equation}		
		Now for all $i$ with $1 \leq i\leq m$ such that \eqref{equi17} and \eqref{equi18} hold, we must have $$\mathcal U_n(e^{\gamma C_{11}}C_{17},e^{\gamma C_{11}}C_{18},\mathcal{A}^i_T)\cap \mathcal V_n(e^{-\gamma C_{11}}C_{19},e^{\gamma C_{11}}C_{20},\mathcal{A}^i_T) \mbox{ occurs}.$$
		On this event, by Lemmas \ref{greenlem2} and \ref{greenlem4} we have \eqref{equi32} holds with $C_{12}=100C^\prime e^{2\gamma C_{11}}C_{17}C_{18}$ and $c_3=0.01c_4e^{-\gamma C_{11}}C_{19}$, where $C^\prime$ is an absolute constant from Lemma \ref{greenlem2} and $c_4=c_4(e^{\gamma C_{11}}C_{20})$.
		Thus we can complete the proof by \eqref{equi19}, \eqref{equi20} and Lemmas \ref{greenlem2} and \ref{greenlem4} (where \eqref{equi19} and \eqref{equi20} control the environment and Lemmas \ref{greenlem2} and \ref{greenlem4} yield properties for the random walk provided with a ``desirable'' environment.).
	\end{proof}

	\begin{proof}[Proof of Lemma \ref{equilem4}]
		We take constants $C_{12},c_3$ from Lemma \ref{equilem7}. Taking $\mathsf C = 8$, $\mathsf c = 2/3$ and $m$ sufficiently large, we see that with $\mathbb P$-probability at least $1-1/3L$, $|I_Y|\geq m/3$ holds for all dyadic boxes $Y\subset B(8^{m+1})$. For $x\in B(8^{m+1})$, let $Y\subset B(8^{m+1})$ be the dyadic box with side length $1$ that contains $x$. We write $\tau^i_x$ as the exit time of $\mathcal{A}^i_Y$ by the random walk started at $x$. Then by Lemma \ref{exitsec}, for each $i\in I_Y$ and $x\in \mathcal{C}^i_Y$, we have $E\tau_x^i\geq c_3 8^{2i-2}\chi_n$ and $E(\tau^i_x)^2\leq 2C_{12}^28^{2i+2}\chi_n^2$. By Paley-Zygmund's inequality, we get that for $i\in I_Y$
		\begin{equation}
			P(\tau_x^i\geq c_3 8^{2i-3}\chi_n)\geq 2^{-10}c_3^2/C_{12}^2,\mbox{ for all } x\in \mathcal{C}^i_Y.
		\end{equation}
		Thus by the strong Markov property and by considering $i\in I_Y$ with $i\geq m/6$ (and there are at least $m/6$ such $i$'s)
		\begin{equation}
			P(X^{(n,x)} \mbox{ exits } B(3\cdot8^{m+1}) \mbox{ before time }c_3 8^{m/3-3} \chi_n)<(1-\frac{2^{-10}c_3^2}{C_{12}^2})^{ m/6},\mbox{ for all } x\in B(8^{m+1}).
		\end{equation}
		This then completes the proof of the lemma by taking $l = 3m+3$ where $m$ is taken sufficiently large such that $c_3 8^{m/3 - 3} \geq L$ and $(1-\frac{2^{-10}c_3^2}{C_{12}^2})^{ m/6}<1/3L$.
	\end{proof}
	
	Recall that $\hat{X}^{(n,x)}:[0,\infty)\to \mathbb{R}^2$ is the piecewise linear interpolation of the process $t\to X^{(n,x)}_{\chi_nt}$. We prove the equicontinuity of the paths in the following lemma.
	\begin{lem}\label{equilem2}
		For each $L>0$, there exists $\delta$ depending only on $L$ such that the following holds. For each $n>0$, with $\mathbb P$-probability at least $1-1/L$,
		\begin{equation}
			P(|\hat{X}^{(n,x)}_s-\hat{X}^{(n,x)}_{s^{\prime}}|\leq 1/L, \mbox{ for all } s, s^{\prime}\in [0,L]\mbox{ with }|s-s^{\prime}|\leq \delta)\geq 1-1/L,\mbox{ for all } x\in B(L).
		\end{equation}
	\end{lem}
	\begin{proof}
		Note that for all box $B$, if $X^{(n,x)}$ exits $B$ before time $L\chi_n$, then $\hat X^{(n,x)}$ also exits $B$ before time $L$. By Lemma \ref{equilem4}, there exists $l = l(L)$ such that the following holds: if $B$ and $\hat{B}$ are concentric dyadic squares with side lengths $2^l$ and $3\cdot 2^l$ that contain $B(L)$, then with $\mathbb P$-probability at least $1-1/3L$ we have
		\begin{equation}\label{equi5}
			P^{(n,x)}(\hat{X}^{(n,x)} \mbox{ exits }\hat{B} \mbox{ before time }L)<1/3L,\mbox{ for all } x\in B(L).
		\end{equation}
		
		Now we fix $k$ such that $2^{-k}\leq 1/10L$. For $x\in \mathbb R^2$, recall that $ S_k(x)$ is the dyadic square in $\mathcal S_k$ which contains $x$, and $\hat S_k(x)$ is the square with the same center as $S_k(x)$ and with side length $3\times2^{-k}$. For $n>0$, let $\tau^{(n,x)}_0=0$ and inductively define
		$$
			\tau^{(n,x)}_i=
			\begin{cases}
				\inf\{s>\tau^{(n,x)}_{i-1}:\hat X^{(n,x)}_s\in \partial \hat{S}(\hat X^{(n)}_{\tau^{(n,x)}_{i-1}})\} & \mbox{if }\hat X^{(n,x)}_{\tau^{(n,x)}_{i-1}}\in B, \\
				\tau^{(n,x)}_{i-1} & \mbox{if }\hat X^{(n,x)}_{\tau^{(n,x)}_{i-1}}\notin B.
			\end{cases}
		$$
		By Lemma \ref{exitlem2} (with a slight difference where we substitute $B(1)$ with $B$, and the proof still works), there exists $t = t (L, k, l)$ such that, with $\mathbb P$-probability at least $1-1/3L$,
		\begin{equation}\label{equi6}
			P(\tau^{(n,x)}_t=\tau^{(n,x)}_{t-1})\geq 1-1/3L,\mbox{ for all } x\in B(L).
		\end{equation}		
		Take $\epsilon>0$ depending only on $L$ (while $\epsilon$ appears to depend on $k, l, t$, these were all chosen depending only on $L$) such that $$2^{2(k-l)}\epsilon\leq 1/3L \And (1-\epsilon)^t\geq 1-1/3L.$$ By Lemma \ref{equilem1} and a union bound, we have that for some $\delta>0$ depending only on $\epsilon$, with $\mathbb P$-probability at least $1-2^{2(k+l)}\epsilon$,
		\begin{equation}\label{equi7}
			P(\hat X^{(n,y)} \mbox{ exits }\hat{S}_k(y) \mbox{ before time }\delta 2^{-2k})<\epsilon,\mbox{ for all } y\in B.
		\end{equation}
		It is clear that the environment satisfies \eqref{equi5}, \eqref{equi6} and \eqref{equi7} with $\mathbb P$-probability at least $1-1/L$, and thus we may assume that this holds in what follows. On this event (in particular, we use \eqref{equi6} and \eqref{equi7} below), we apply the strong Markov property and get that (denote by $\mathtt R_x$ the minimum $r$ such that $\tau^{n, x}_r = \tau^{n, x}_{r-1}$)
		$$
			P(\tau^{(n,x)}_r-\tau^{(n,x)}_{r-1}\geq \delta 2^{-2k},\mbox{ for all }1\leq r\leq t \wedge \mathtt R_x)\geq (1-\epsilon)^t\geq 1-1/3L, \mbox{ for all } x\in B(L),
		$$
		$$
			P(\mathtt R_x \leq t)\geq 1-1/3L,\mbox{ for all } x\in B(L).
		$$
		In addition, by \eqref{equi5}, with $P$-probability at least $1-1/3L$ it holds that $\tau^{n, x}_{\mathtt R_x} > L$. Altogether, we derive that
		$$
			P(|\hat{X}^{(n,x)}_s-\hat{X}^{(n,x)}_{s^{\prime}}|\leq 2\sqrt{2}\cdot 2^{-k},\mbox{ for all } s,s^{\prime}\in [0,L] \mbox{ with }|s-s^{\prime}|\leq \delta 2^{-2k})\geq 1-1/L,\mbox{ for all } x\in B.
		$$
		Since $k,\delta$ depend only on $L$, we complete the proof.
	\end{proof}
	
	\subsection{Equicontinuity of laws}\label{subsec92}
	\begin{lem}\label{equilem3}
		For each $L>1$, there exists $\delta = \delta(L)$ such that for each $n>0$, with $\mathbb P$-probability at least $1-1/L$ the following holds. For each $x,y\in B(L)$ with $|x-y|\leq \delta$, the Prokhorov distance between $P^{(n,x)}$ and $P^{(n,y)}$ is at most $1/L$.
	\end{lem}
	
	\begin{proof}
		We fix $k>0$ depending only on $L$ whose value will be determined later. We define $\mathcal{D}_k$ to be the set of all dyadic boxes in $B(L)$ as well as their shifts by $(0,2^{-k-1})$, $(2^{-k-1},0)$, and $(2^{-k-1},2^{-k-1})$. For $S\in \mathcal{D}_k$ that contains $x$, recall that $\hat{S}$ is the square concentric to $S$ with side length $3\cdot 2^{-k}$. Also define $\tau^x_S$ to be the first time that $\hat{X}^{(n,x)}$ exits $\hat S$. By Lemmas \ref{exitlem1} and \ref{exitlem}, for all $\mathsf C>0$ and $k<a<n$, there exists $\alpha=\alpha(\mathsf C)$ such that the following holds: if $x,y\in S$ with $|x-y|\leq 2^{-a}$, we have that with $\mathbb P$-probability at least $1-\exp(-\mathsf C (a-k))$,
		\begin{equation}
			d_{\mathrm{TV}}(\hat{X}^{(n,x)}_{\tau^x_S},\hat{X}^{(n,y)}_{\tau^y_S})\leq \exp(-\alpha(a-k)).
		\end{equation}
		By a union bound over $S\in \mathcal D_k$, we derive the following result on continuity: with $\mathbb P$-probability at least $1-2^{2(L+k)+10} \exp ( -\mathsf C ( a - k ))$, for each $x,y\in B(L)$ with $|x-y|\leq 2^{-a}$, we have that
		\begin{equation}\label{equi22}
			d_{\mathrm{TV}}(\hat{X}^{(n,x)}_{\tau^x_S},\hat{X}^{(n,y)}_{\tau^y_S})\leq \exp(-\alpha(a-k))
		\end{equation}
		holds for all $S \in \mathcal D_k$ such that $x,y\in S$ (note that by our definition of $\mathcal D_k$ such $S$ exists). Assuming the environment satisfies \eqref{equi22} for all $x,y\in B(L)$ with $|x-y|\leq 2^{-a}$ and for all corresponding $S$, by the strong Markov property we can construct a coupling $(\Omega,\mathcal{F},Q)$ between $\hat{X}^{(n,x)}_{\tau^x_S+t}$ and $\hat{X}^{(n,y)}_{\tau^y_S+t}$ such that 
		\begin{equation}
			Q(\hat{X}^{(n,x)}_{\tau^x_S+t}=\hat{X}^{(n,y)}_{\tau^y_S+t}, \mbox{ for all } t\geq 0)\geq 1-\exp(-\alpha(a-k)).
		\end{equation}
		By taking $\mathsf C=10$, $\delta=2^{-a}$ and $a$ sufficiently large depending only on $L$, we have that with $\mathbb P$-probability at least $1-1/3L$,
		\begin{equation}\label{equi8}
			\begin{aligned}
				&\mbox{for all }x,y\in B(L) \mbox{ with } |x-y|\leq 2^{-a} \mbox{ and }S \mbox{ that contains }x,y,\\ \mbox{ there exists}&\mbox{ a coupling }Q \mbox{ such that } Q(\hat{X}^{(n,x)}_{\tau^x_S+t}=\hat{X}^{(n,y)}_{\tau^y_S+t}, \mbox{ for all } t\geq 0)\geq 1-1/4L.
			\end{aligned}
		\end{equation}
		
		Now let $L^\prime > 3L$ depend only on $L$, whose value will be determined later. By Lemma \ref{equilem2}, there exists $\beta = \beta(L^\prime)$ such that with $\mathbb P$-probability at least $1-1/3L$
		\begin{equation}\label{equi9}
			P(|\hat{X}^{(n,x)}_s-\hat{X}^{(n,x)}_{s^{\prime}}|\leq 1/L^{\prime}, \mbox{ for all } s, s^{\prime}\in [0,L]\mbox{ with }|s-s^{\prime}|\leq \beta)\geq 1-1/4L,\mbox{ for all } x\in B(L^{\prime}).
		\end{equation}
		
		By Lemma \ref{greenlem1}, there exists some $C_q=o(q^{-0.9})$ depending only on $q$ such that for each $S\in \mathcal D_k$, with $\mathbb P$-probability at least $1-q$,
		\begin{equation}\label{equi23}
			E\tau^x_S\leq C_q2^{-2k},\mbox{ for all } x\in S.
		\end{equation}
		By a union bound over $S\in \mathcal D_k$, we see that with $\mathbb P$-probability at least $1 - 10 \cdot 2^{2k} L^2 q$ we have \eqref{equi23} holds for all $S\in \mathcal D_k$. Therefore, on this event, we can apply Markov's inequality (to $\tau^x_S$) and deduce that
		\begin{equation}
			P(\tau^x_S\leq 4C_qL2^{-2k})\geq 1-1/4L,\mbox{ for all } x\in B(L).
		\end{equation}
		Take $q$ sufficiently small depending only on $L$ such that $10\cdot2^{2k}L^2q=1/4L$. Since $C_q=o(q^{-0.9})$, for some absolute constant $C$ we have $4C_qq^{0.9}\leq C$ for all $q\in (0,1)$, and thus we have
		\begin{equation}
			4C_qL2^{-2k}=4C_q\cdot q\cdot (q\cdot 2^{2k})^{-1}\cdot L\leq Cq^{0.1}\cdot (q\cdot 2^{2k})^{-1}L\leq 100C2^{-0.2k} L^4,
		\end{equation}
		where in the last inequality we used the fact that $q\leq 2^{-0.2k}$. By taking $k$ sufficiently large depending only on $L$ such that $100C2^{-0.2k} L^4\leq \beta$ and $4\cdot 2^{-k}\leq 1/L^{\prime}$ (the latter condition on $k$ is to be used in \eqref{equi24} below), we then get that with $\mathbb P$-probability at least $1-1/3L$,
		\begin{equation}\label{equi10}
			P(\tau^x_S\leq \beta)\geq 1-1/4L,\mbox{ for all } x\in B(L).
		\end{equation}
		
		On the intersection of the events in \eqref{equi8}, \eqref{equi9} and \eqref{equi10}, 
		for all $S\in \mathcal D_k$ and for all $x, y\in S$, there exists a coupling $Q$, such that the following hold with $Q$-probability at least $1-1/L$: 
		\begin{equation}\label{equi33}
			\hat{X}^{(n,x)}_{\tau^x_S+t}=\hat{X}^{(n,y)}_{\tau^y_S+t}, \mbox{ for all } t\geq 0;
		\end{equation}
		\begin{equation}\label{equi34}
			|\hat{X}^{(n,x)}_s-\hat{X}^{(n,x)}_{s^{\prime}}|\leq 1/L^{\prime}, \mbox{ for all } s, s^{\prime}\in [0,L]\mbox{ with }|s-s^{\prime}|\leq \beta;
		\end{equation}
		\begin{equation}\label{equi35}
			\tau^x_S\leq \beta \And \tau^y_S\leq \beta.
		\end{equation}
		We may assume without loss that $\tau^x_S\leq \tau^y_S$ (since otherwise an analogue of the following holds) and derive that: for $t\leq \tau^x_S\leq \beta$,
		\begin{equation}\label{equi24}
			|\hat X^{(n,x)}_t-\hat X^{(n,y)}_t|\leq 4\cdot 2^{-k}\leq 1/L^{\prime};
		\end{equation}
		for $\tau^x_S\leq t\leq \tau^y_S\leq \beta$, since $|t-\tau^x_S|\leq \beta$,
		\begin{equation}\label{equi25}
			|\hat X^{(n,x)}_t-\hat X^{(n,y)}_t|\leq |\hat X^{(n,x)}_{\tau^x_S}-\hat X^{(n,y)}_t|+|\hat X^{(n,x)}_t-\hat X^{(n,x)}_{\tau^x_S}|\leq 4\cdot 2^{-k}+1/L^{\prime}\leq 2/L^{\prime};
		\end{equation}
		for $\tau^y_S\leq t\leq L$, writing $t=\tau^x_S+u=\tau^y_S+v$, we then have $|u-v|\leq \beta$, and thus we have (by \eqref{equi33} and \eqref{equi34})
		\begin{equation}\label{equi26}
			|\hat X^{(n,x)}_t-\hat X^{(n,y)}_t|= |\hat X^{(n,x)}_{\tau^x_S+u}-\hat X^{(n,y)}_{\tau^y_S+v}|=|\hat X^{(n,x)}_{\tau^x_S+u}-\hat X^{(n,x)}_{\tau^x_S+v}|\leq 1/L^{\prime}.
		\end{equation}
		Combining \eqref{equi24}, \eqref{equi25} and \eqref{equi26}, we can then obtain that on the intersection of the events in \eqref{equi33}, \eqref{equi34} and \eqref{equi35},
		\begin{equation}
			d_{\mathrm{Unif}}(\hat X^{(n,x)},\hat X^{(n,y)})=\int_0^{\infty}e^{-T}\wedge \sup_{t\in[0,T]}|\hat{X}^{(n,x)}_t-\hat{X}^{(n,y)}_t|\leq \int_0^{L}\frac{2}{L^{\prime}}dT+\int_{L}^{\infty}e^{-T}dT.
		\end{equation}	
		This implies that we can take $L^\prime$ sufficiently large depending only on $L$ such that on the intersection of \eqref{equi8}, \eqref{equi9} and \eqref{equi10} we have that for all $x,y\in B(L)$ with $|x-y|\leq 2^{-a}$,
		\begin{equation}
			\exists \mbox{ coupling } Q \mbox{ such that }Q(d_{\mathrm{Unif}}(\hat X^{(n,x)},\hat X^{(n,y)})\leq 1/L)\geq 1-1/L.
		\end{equation}
		Since the intersection of \eqref{equi8}, \eqref{equi9} and \eqref{equi10} occurs with $\mathbb P$-probability at least $1-1/L$, we complete the proof by the definition of the Prokhorov distance.
	\end{proof}
	
	\subsection{Conclusion}
	In this subsection we consider the space $C(\mathbb{R}^2,\mathrm{Prob}(C([0,\infty), \mathbb{R}^2)))$ of continuous functions $\mathbb{R}^2\to \mathrm{Prob}(C([0,\infty), \mathbb{R}^2))$, so that we can apply the Arzela-Ascoli theorem. We will write $\{P^x\}_{x\in \mathbb R^2}$ for an element in $C(\mathbb{R}^2,\mathrm{Prob}(C([0,\infty), \mathbb{R}^2)))$, where $P^x\in \mathrm{Prob}(C([0,\infty), \mathbb{R}^2))$. We will need the following lemma, which is similar to \cite[Lemma~6.3]{berestycki2022random}. We include a proof merely for completeness since our metric associated with the space $C(\mathbb{R}^2,\mathrm{Prob}(C([0,\infty), \mathbb{R}^2)))$ is slightly different.
	\begin{lem}\label{equilem8}
		Let $\mathcal{K}\subset C(\mathbb{R}^2,\mathrm{Prob}(C([0,\infty), \mathbb{R}^2)))$. Then $\mathcal{K}$ is pre-compact  (i.e., the closure of $\mathcal{K}$ is compact) if the following conditions hold.
		\begin{enumerate}
			\item\label{con1} For each $L>0$ and $x\in \mathbb{R}^2$, there exists $\delta$ depending only on $x,L$ such that the following holds. For each $P = \{P^x\}_{x\in \mathbb R^2}\in\mathcal{K}$, if $X$ is sampled from $P^x$, then with $P^x$-probability at least $1-1/L$,
			\[|X_s-X_{s^{\prime}}|\leq 1/L, \mbox{ for all } s, s^{\prime}\in [0,L]\mbox{ with }|s-s^{\prime}|\leq \delta.\]			
			\item\label{con2} For each $L>1$, there exists $\delta$ depending only on $L$ such that the following holds. For each $P = \{P^x\}_{x\in \mathbb R^2}\in\mathcal{K}$ and for $x,y\in B(L)$ with $|x-y|\leq \delta$, the Prokhorov distance between $P^x$ and $P^y$ is at most $1/L$.
		\end{enumerate}
	\end{lem}
	\begin{proof}
		We first check that for each $x$, $\{P^x\}_{P\in\mathcal{K}}$ is pre-compact. To this end, for $\zeta\in(0,1)$ and for $k>0$, let $\delta_k$ be as in Condition \eqref{con1} for $L=L_k=2^k/\zeta$. By taking a union bound over all $k$, we get that if $P\in\mathcal{K}$ and $X$ is sampled from $P^x$ then
		it holds with $P^x$-probability at least $1-\zeta$ that
		\begin{equation}\label{pre1}
			|X_s-X_{s^{\prime}}|\leq 1/L_k, \mbox{ for all } s, s^{\prime}\in [0,L_k]\mbox{ with }|s-s^{\prime}|\leq \delta_k, \mbox{ for all } k\geq 0.
		\end{equation}
		Let $A_\zeta \subset C([0, \infty), \mathbb R^2)$ be the collection of paths started at $x$ and satisfying \eqref{pre1}. Then $A_\zeta$ is pre-compact via a diagonal argument. Since $P^x(A_{\zeta})\geq 1-\zeta$, we can deduce that $\{P^x\}_{P\in \mathcal{K}}$ is tight, and thus is pre-compact by Prokhorov's Theorem.
		
		Now since Condition \eqref{con2} holds, for each sequence $\{P_n\}\subset \mathcal{K}$, since $\{P^x_n\}_{n\geq 1}$ is pre-compact for each fixed $x$, by the Arzela-Ascoli theorem, for each $L>0$ there exists a subsequence $\{P_{n_k}\}$ that converges uniformly in $B(L)$. Thus by a diagonal argument there exists a further subsequence that converges uniformly in $B(L)$ for all $L>0$, and thus converges in the metric $d$ (as in \eqref{metric}). So $\mathcal{K}$ is pre-compact.
	\end{proof}
	
	\begin{proof}[Proof of Theorem \ref{thm3}]
		Recall the definition of $P^{(n)}$ in \eqref{Pdef}, which directly gives the continuity of $P^{(n)}$ with respect to $d_{\mathrm{prok}}$. Thus we have $P^{(n)}\in C(\mathbb{R}^2,\mathrm{Prob}(C([0,\infty), \mathbb{R}^2)))$.
		
		For each $\zeta>0$, by Lemmas \ref{equilem2} and \ref{equilem3}, for $L_k=2^{k}/\zeta$ there exists $\delta_k$ such that, for each $n$ the following holds with $\mathbb P$-probability at least $1-1/L_k$:
		\begin{equation}\label{pre3}
			P^{(n)}(|\hat{X}^{(n,x)}_s-\hat{X}^{(n,x)}_{s^{\prime}}|\leq 1/L, \mbox{ for all } s, s^{\prime}\in [0,L]\mbox{ with }|s-s^{\prime}|\leq \delta_k)\geq 1-1/L,\mbox{ for all } x\in B(L_k),
		\end{equation}
		and in addition,
		\begin{equation}\label{pre4}
			d_{\mathrm{Prok}}(P^{(n,x)},P^{(n,y)})\leq 1/L_k, \mbox{ for all }x,y\in B(L_k) \mbox{ with }|x-y|\leq \delta_k.
		\end{equation}
		
		Now we define $\mathcal{K}_{\zeta}$ be the subset of $C(\mathbb{R}^2,\mathrm{Prob}(C([0,\infty), \mathbb{R}^2)))$ satisfying \eqref{pre3} and \eqref{pre4} for all $k\geq 0$.	By Lemma \ref{equilem8} we get that $\mathcal{K}_{\zeta}$ is pre-compact.
		Taking a union bound over all $k\geq 0$, we get that $\mathbb{P}(P^{(n)}\in \mathcal{K}_{\zeta})\geq 1-\zeta$ for all $n>0$. So $\{P^{(n)}\}$ is tight.
	\end{proof}
	
	\bibliography{tightnessbib}
\end{document}